\documentclass{amsart}

\usepackage[mathscr]{eucal}
\usepackage{amssymb}
\usepackage{upgreek}
\usepackage{graphicx}
\usepackage{psfrag} 
\usepackage[dvipsnames]{xcolor}
\usepackage[normalem]{ulem}
\usepackage{amsthm}
\usepackage{bbm}
\usepackage{enumerate}
\usepackage[color=yellow]{todonotes}
\usepackage{array}
\usepackage{amsmath}

\usepackage{hyperref}
\hypersetup{colorlinks=true,linkcolor={Brown},citecolor={Brown},urlcolor={Brown}}

\usepackage{cleveref}
\crefname{Thm}{Theorem}{Theorems}
\crefname{Rem}{Remark}{Remarks}
\crefname{Prop}{Proposition}{Propositions}
\crefname{Cor}{Corollary}{Corollaries}
\crefname{Cons}{Construction}{Constructions}
\crefname{Exa}{Example}{Examples}
\crefname{Lem}{Lemma}{Lemmas}
\crefname{Rec}{Recollection}{Recollections}
\crefname{subsection}{Subsection}{Subsections}

\numberwithin{equation}{section}
\makeindex

\usepackage[all]{xy}
\CompileMatrices

\newdir{ >}{{}*!/-10pt/\dir{>}}

\swapnumbers 

\newtheorem{Thm}[equation]{Theorem}
\newtheorem{Prop}[equation]{Proposition}
\newtheorem{Lem}[equation]{Lemma}
\newtheorem{Cor}[equation]{Corollary}

\theoremstyle{remark}
\newtheorem{Rem}[equation]{Remark}
\newtheorem{Def}[equation]{Definition}

\newtheorem{Not}[equation]{Notation}
\newtheorem{Exa}[equation]{Example}

\newtheorem{Cons}[equation]{Construction}
\newtheorem{Conv}[equation]{Convention}

\newtheorem{Rec}[equation]{Recollection}

\theoremstyle{definition}
\newtheorem*{Ack}{Acknowledgements}
\newtheorem*{AI}{Statement about use of A.I}

\newcommand{\nc}{\newcommand}
\nc{\dmo}{\DeclareMathOperator}

\dmo{\Ab}{Ab}
\dmo{\Aut}{Aut}
\dmo{\biMack}{\mathsf{Mack}} 
\dmo{\End}{End}
\dmo{\Fun}{\mathsf{Fun}} 
\dmo{\Hom}{Hom}
\dmo{\incl}{incl}
\dmo{\Ind}{Ind}
\dmo{\ind}{ind}
\dmo{\Mor}{Mor}%
\dmo{\Obj}{Obj}
\dmo{\pr}{pr}
\dmo{\Res}{Res}
\dmo{\Spanname}{{\sf Span}}
\dmo{\triv}{triv}
\dmo{\Kname}{K}
\dmo{\Kar}{Kar}
\dmo{\lename}{leq}
\dmo{\proj}{proj}
\dmo{\forget}{forget}
\dmo{\Sub}{Sub}
\dmo{\Mat}{M}
\dmo{\indec}{conn}

\nc\noloc{\nobreak\mspace{6mu plus 1mu}{:}\nonscript\mkern-\thinmuskip\mathpunct{}\mspace{2mu}}
\nc{\ababs}{{\sl ab absurdo}}
\nc{\ADD}{\mathsf{ADD}} 
\nc{\SADD}{\mathsf{SADD}} 
\nc{\adjto}{\rightleftarrows}
\nc{\adj}{\dashv\,}
\nc{\aka}{{a.\,k.\,a.}\ }
\nc{\bbN}{\mathbb{N}}
\nc{\bbZ}{\mathbb{Z}} 
\nc{\bbC}{\mathbb{C}} 
\nc{\cat}[1]{\mathcal{#1}}
\nc{\CAT}{\mathsf{CAT}}
\nc{\doublequot}[3]{#1\backslash #2/#3}
\nc{\Ecell}{\rotatebox[origin=c]{90}{$\Downarrow$}} 
\nc{\eg}{{\sl e.g.}\ } 
\nc{\eps}{\varepsilon}
\nc{\equalby}[1]{\overset{\textrm{#1}}{=}}
\nc{\Funadd}{\Fun_{\amalg}}
\nc{\GG}{\mathbb{G}}
\nc{\GGi}{\GG_{\approx}}
\nc{\GGiso}{\GG_{\simeq}}
\nc{\GGisoconn}{\GGiso^{\indec}}
\nc{\gpd}{\groupoid}
\nc{\groupoidf}{\groupoid{}^{\smallfaithful}}
\nc{\gpdf}{\groupoidf}%
\nc{\groupoid}{\mathsf{gpd}}
\nc{\HGK}{\doublequot HGK}
\nc{\KGH}{\doublequot KGH}
\nc{\hook}{\hookrightarrow}
\nc{\id}{\mathrm{id}}
\nc{\Id}{\mathrm{Id}}
\nc{\ie}{{\sl i.e.}\ }
\nc{\into}{\mathop{\rightarrowtail}}
\nc{\inv}{^{-1}}
\nc{\isocell}[1]{\underset{#1}{\isoTo}}
\nc{\lisocell}[1]{\underset{#1}{\overset{\sim}{\Leftarrow}}}
\nc{\isotoo}{\overset{\sim}{\longrightarrow}}
\nc{\isoto}{\overset{\scriptscriptstyle\sim}{\to}}
\nc{\leto}{\buildrel \approx\over\to}
\nc{\kk}{\Bbbk}
\nc{\leps}{{}^{\ell}\eps}
\nc{\leta}{{}^{\ell}\eta}
\nc{\etalp}{{}_{\scriptscriptstyle\oplus\!}\eta}
\nc{\epslp}{{}_{\scriptscriptstyle\oplus}\eps}
\nc{\etaup}{{}^{\scriptscriptstyle\oplus\!}\eta}
\nc{\epsup}{{}^{\scriptscriptstyle\oplus\!}\eps}
\nc{\appr}{\approx}
\nc{\apprto}{\xto{\scriptscriptstyle\appr}} 
\nc{\gH}{{{}^{{\scriptstyle g}\!}H}}
\nc{\gHsmall}{{{}^{{}^{g}\!\!}H}}

\nc{\mackfun}[1]{\Fun^{\mathsf{mack}}_{\scriptscriptstyle\ADD}(#1)}
\nc{\resfun}[1]{\Fun^\mathsf{res}_{\scriptscriptstyle\ADD}(#1)}
\nc{\conjfun}[1]{\Fun^\mathsf{conj}_{\scriptscriptstyle\ADD}(#1)}
\nc{\twofun}{2\textrm{-}\!\Fun}

\nc{\loccit}{{\sl loc.\ cit.}}
\nc{\lto}{\leftarrow}
\nc{\Mack}[1]{(Mack\,\ref{it:Mack-#1})}
\nc{\Mid}{\,\big|\,}
\nc{\mmod}{\textrm{-}\mathsf{mod}}%
\nc{\ffree}{\,\textrm{-}\mathsf{free}}
\nc{\cA}{\cat{A}}
\nc{\cB}{\cat{B}}
\nc{\cW}{\cat{W}}
\nc{\Ap}{\cat{A}^+}
\nc{\cM}{\cat{M}}
\nc{\NEcell}{\rotatebox[origin=c]{135}{$\Downarrow$}} 
\nc{\NWcell}{\rotatebox[origin=c]{-135}{$\Downarrow$}} 
\nc{\onto}{\mathop{\twoheadrightarrow}}
\nc{\op}{{\mathrm{op}}}
\nc{\qquadtext}[1]{\qquad\textrm{#1}\qquad}
\nc{\quadtext}[1]{\quad\textrm{#1}\quad}
\nc{\ra}{\rightarrow}
\nc{\reps}{{}^{r\!}\eps}
\nc{\restr}[1]{{|_{\scriptstyle #1}}}
\nc{\reta}{{}^{r\!}\eta}
\nc{\sbull}{{\scriptscriptstyle\bullet}}
\nc{\SEcell}{\rotatebox[origin=c]{45}{$\Downarrow$}} 
\nc{\SET}[2]{\big\{\,#1\Mid#2\,\big\}}
\nc{\smallfaithful}{\mathsf{f}}
\nc{\smat}[1]{\left(\begin{smallmatrix} #1 \end{smallmatrix}\right)}
\nc{\Spanfname}{\Spanname^{\smallfaithful}}
\nc{\sset}{\textrm{-}\mathsf{set}}
\nc{\SWcell}{\rotatebox[origin=c]{-45}{$\Downarrow$}} 
\nc{\too}{\mathop{\longrightarrow}\limits}
\nc{\undersett}[1]{\underset{\scriptstyle #1}}
\nc{\xto}[1]{\xrightarrow{#1}}
\nc{\xinto}[1]{\overset{#1}{\,\into\,}}
\nc{\xToo}[1]{\overset{#1}{\,\Too\,}}
\nc{\xlto}{\xleftarrow}
\nc{\GGf}{\GG^{\smallfaithful}}
\NewDocumentCommand{\naj}{o m}{\nabla^{\IfNoValueTF{#1}{j}{#1}}_{{\!#2}}}
\nc{\najH}{{\naj{H}}}
\nc{\najHp}{{\naj{j_!H'}}}
\nc{\najK}{{\naj{K}}}
\nc{\najY}{{\naj{Y}}}
\nc{\naiuY}{\naj[i]{u_!Y}}

\nc{\To}{\Rightarrow}
\nc{\xTo}[1]{\overset{#1}{\,\To\,}}
\nc{\isoTo}{\overset{\sim}{\To}}
\nc{\Too}{\Longrightarrow}

\nc{\Indzero}[1]{{#1}_\circ}
\nc{\pile}{\pi_0^{\lename}}

\nc{\sgn}{\operatorname{sgn}}
\nc{\ukk}{\underline{\kk}}

\nc{\ut}{\tilde{u}}
\nc{\vt}{\tilde{v}}
\nc{\ft}{\tilde{f}}
\nc{\Ptil}{\widetilde{P}}
\nc{\tr}{\mathrm{tr}}
\nc{\gammaqp}{\gamma_{_{QP}}}
\nc{\I}{\cat{I}}
\nc{\F}{\cat{F}}
\nc{\G}{\cat{G}}
\nc{\C}{\cat{C}}
\nc{\D}{\cat{D}}
\nc{\E}{\cat{E}}

\NewDocumentCommand{\SpanG}{o}{\Spanname_{\GG/\IfNoValueTF{#1}{G}{#1}}}
\NewDocumentCommand{\SpanfG}{o}{\Spanfname_{\GG/\IfNoValueTF{#1}{G}{#1}}}

\NewDocumentCommand{\Alp}{o}{\IfNoValueTF{#1}{\cA}{#1}_{\scriptscriptstyle\oplus}}
\nc{\Alpp}{\Alp[\cA']}
\nc{\Alpt}{\Alp[t]}

\NewDocumentCommand{\Bup}{o}{\IfNoValueTF{#1}{\cB}{#1}^{\scriptscriptstyle\oplus}}
\nc{\Bupp}{\Bup[{\cB'}]}
\nc{\Bupt}{\Bup[t]}
\nc{\Bupr}{\Bup[r]}
\nc{\rup}{\Bup[r]}

\NewDocumentCommand{\Burn}{o}{\Omega_{\IfNoValueTF{#1}{\kk}{#1}}} 
\nc{\BurnG}{\Burn(G)} 

\nc{\Aup}{\Bup[\cA]}
\dmo{\mh}{m}
\nc{\KZ}{{\Kname_0}}
\nc{\cast}{\circledast}
\nc{\Ak}{\cA_{\kk}}
\nc{\Alpk}{\Alp[(\Ak)]}
\nc{\Aupk}{\Bup[\Ak]}
\nc{\Rk}{\cR_{\kk}}

\nc{\stname}{\S}
\nc{\st}[1]{\stname_{#1}}
\nc{\stH}{\st{H}}
\nc{\stHp}{\st{H'}}
\nc{\stK}{\st{K}}
\nc{\stL}{\st{L}}
\nc{\stP}{\st{P}}
\nc{\stX}{\st{X}}
\nc{\stY}{\st{Y}}
\nc{\stT}{\st{T}}
\nc{\stQ}{\st{Q}}

\nc{\gray}[1]{{\color{gray}#1}}
\nc{\green}[1]{{\color{ForestGreen}{#1}}}

\nc{\Hatice}[1]{{\color{RoyalBlue}#1}}
\nc{\Hatout}[1]{\Hatice{\sout{#1}}}
\nc{\Paul}[1]{{\color{ForestGreen}#1}}
\nc{\Pout}[1]{\Paul{\sout{#1}}}

\begin{document}


\title{Mackeyfication of equivariant categories~I}
\author{Paul Balmer}
\author{Hatice Mutlu}
\date{\today}

\address{\ \vfill
\noindent UCLA Mathematics Department, Los Angeles, CA 90095-1555, USA}

\urladdr{https://www.math.ucla.edu/~balmer}

\urladdr{https://mutluhtc.github.io}

\begin{abstract} \normalsize
Colloquially speaking, `equivariant categories' refer to families of additive categories~$\mathcal{A}(G)$ depending 2-functorially on a finite group~$G$.
We construct approximations of equivariant categories by Mackey 2-functors, both on the left and on the right.
The idea is to enlarge~$\mathcal{A}$ in a minimal way to make induction appear.
These `mackey\-fications' are inspired by Boltje's work with ordinary Mackey 1-functors.
We also relate our left and right mackeyfications via a mark transformation. Finally we discuss examples.
\end{abstract}

\thanks{}


\maketitle


\vskip-\baselineskip\vskip-\baselineskip\vskip-\baselineskip
\tableofcontents
\vskip-\baselineskip\vskip-\baselineskip
\vskip-\baselineskip


%
\begin{center}
\textsc{Preamble}
\end{center}
\label{sec:preamble}%
In the context of ordinary Mackey functors of abelian groups~\cite{Green71,Dress73}, building on work of Deligne~\cite{Deligne73}, Dress~\cite{Dress71} and Th\'evenaz~\cite{Thevenaz88}, Robert Boltje~\cite{Boltje98} proposed two constructions of Mackey functors~$A_+$ and~$B^+$ associated respectively to `restriction' functors~$A$ and to `conjugation' functors~$B$.

In this paper, we propose a categorification of these ideas using Mackey 2-functors of additive categories, in the sense of~\cite{BalmerDellAmbrogio20}.
While we focus exclusively on the 2-categorical story, the reader familiar with Boltje~\cite{Boltje98} will easily see the parallels.
A precise comparison, under $K$-theoretic decategorification and under $\Hom$-decategorification \`a la~\cite{BalmerDellAmbrogio24}, will appear in the upcoming~\cite{BalmerMutlu27pp}.

Expertise in 2-category theory is not necessary for this paper. All our 2-categories are close cousins of~$\CAT$, the domestic 2-category of categories, functors and natural transformations, for instance the 2-subcategory $\ADD\subseteq \CAT$ of additive categories and additive functors, or the full 2-subcategory $\gpd\subseteq \CAT$ of finite groupoids.

%
\section{Introduction}
\label{sec:introduction}%
%

Let~$\GG$ be a category of finite groups of interest, \eg all finite groups (for a `global' theory), or only the subgroups of a fixed group~$\Gamma$ (for a `local' theory).
We view $\GG$ as a 2-category whose 2-cells encode conjugation (\Cref{Exa:gps}).
The term `equivariant category' refers broadly to categories~$\cA(G)$ varying 2-functorially with~$G$ in~$\GG$.
More precisely, we use three species of equivariant categories.

First, we have \emph{restriction 2-functors}, \ie the plain 2-functors $\cA\colon \GG^\op\to \ADD$. They consist of an additive category $\cA(G)$ for every group~$G$ in~$\GG$, a restriction $u^*\colon \cA(G)\to \cA(H)$ for every morphism $u\colon H\to G$ in~$\GG$, and a natural isomorphism $\alpha^*\colon u^*\isoTo v^*$ for every 2-cell~$\alpha\colon u\isoTo v$, with the usual functoriality conditions.

Second and most beloved, we have the \emph{Mackey 2-functors} $\cM\colon \GG^{\op}\to\ADD$, namely the restriction 2-functors that admit induction: For every injective morphism $i\colon H\into G$, we require functors $i_*\colon \cM(H)\to \cM(G)$ that are two-sided adjoint to restriction~$i^*$ and that satisfy base-change Mackey formulas. The axiomatization introduced with Dell'Ambrogio in~\cite{BalmerDellAmbrogio20} is recalled in~\Cref{sec:Mackey}.
An entire chapter of~\cite{BalmerDellAmbrogio20} is dedicated to the many standard examples of Mackey 2-functors.

The third species, called \emph{conjugation 2-functors}, will be discussed shortly.

Consider the 2-category $\resfun{\GG}:=\twofun(\GG^{\op},\ADD)$ of restriction 2-functors on~$\GG$, with transformations compatible with restriction as morphisms, and modifications as 2-cells (\Cref{Rec:2-functors}).
Its 2-subcategory $\mackfun{\GG}$ of Mackey 2-functors has the same 2-cells but fewer morphisms: only those that preserve induction (\Cref{Rec:Mack(G)}).
The 2-functor $\mackfun{\GG}\hook \resfun{\GG}$, which forgets induction, admits a left biadjoint that is our first mackeyfication~$\Alp[(-)]$.
As usual, the unit~$\eta\colon \cA\to \Alp$ goes out to the right. \emph{Left} adjoint but \emph{right} mackeyfication.

\begin{Thm}[{\Cref{Thm:Alp-UP}}]
\label{Thm:Alp-UP-intro}%
Let $\cA$ be a restriction 2-functor on~$\GG$. There exists a Mackey 2-functor $\Alp$ and a morphism $\eta\colon \cA\to \Alp$ of restriction 2-functors such that every such morphism $t\colon \cA\to \cM$ into a Mackey 2-functor~$\cM$ factors
uniquely\,${}^{(}$\footnote{\label{footnote:unique}\,Uniqueness is `up to unique invertible modification' in a sense made precise in the text.}${}^{)}$
as $t\cong \hat{t}\circ\eta$ for a morphism of Mackey 2-functors $\hat{t}\colon \Alp\to\cM$:
\[
\kern8em\xymatrix@C=2em@R=.3em{
*+[F]{\cA} \ar[rr]^-{\eta}_-{} \ar[rd]_-{\forall\,t}^(.55){}
&
& \Alp \ar@{-->}[ld]^(.4){\exists!\,\hat{t}}_{}
\\
& \cM
}
\]
\end{Thm}

As with Boltje's $(-)^+$, our second construction $\Bup[(-)]$ is more delicate to state.
Here enter the `conjugation 2-functors', those equivariant categories without induction but also without restriction.
To formalize this, let $\GGiso$ be the 2-full subcategory of~$\GG$ with only isomorphisms as morphisms.
We call $\conjfun{\GG}:=\twofun(\GGiso^\op,\ADD)$ the 2-category of \emph{conjugation 2-functors} on~$\GG$.
Precomposing with the inclusion $\GGiso\hook\GG$ yields a faithful 2-functor $\resfun{\GG}\hook\conjfun{\GG}$ that forgets restrictions.
Its right biadjoint yields our second mackeyfication $\Bup[(-)]$.

\begin{Thm}[{\Cref{Thm:Bup-UP}}]
\label{Thm:Bup-UP-intro}%
Let $\cB$ be a conjugation 2-functor on~$\GG$. There exists a restriction 2-functor $\Bup$ and a morphism $\eps\colon \Bup\to \cB$ of conjugation 2-functors, such that every such morphism $t\colon \cA\to \cB$ from a restriction 2-functor~$\cA$ factors
uniquely\,${}^{(\textrm{\rm\ref{footnote:unique}})}$
as $t\cong \eps\circ\tilde{t}$ for a  morphism $\tilde{t}\colon \cA\to\Bup$ of restriction 2-functors:
\[
\xymatrix@C=2em@R=.3em{
\Bup \ar[rr]^-{\eps}_-{}
&& *+[F]{\cB}
\\
& \cA \ar[ru]_-{\forall\,t}^(.55){} \ar@{-->}[lu]^(.6){\exists!\,\tilde{t}}_{}
}\kern8em
\]
\end{Thm}

As the word `Mackey' does not appear in~\Cref{Thm:Bup-UP-intro}, we seem to encounter a tongue-twisting `left restrictification' of~$\cB$.
By chance, $\Bup$ is always \emph{Mackey}:
\begin{Thm}[{\Cref{Thm:Bup-Mackey}}]
\label{Thm:Bup-Mackey-intro}%
The above construction~$\cB\mapsto \Bup$ yields a 2-functor $\conjfun{\GG}\to \mackfun{\GG}$ from conjugation 2-functors to \emph{Mackey} 2-functors.
\end{Thm}
Hence we call this right biadjoint~$\Bup[(-)]$ the \emph{left mackeyfication} but emphasize that it is \emph{not} right biadjoint to~$\mackfun{\GG}\hook \resfun{\GG}$, nor to the composite $\mackfun{\GG}\hook\conjfun{\GG}$. See~\Cref{Rem:Bup-not-THAT-adjoint}.
We did warn the reader that~$\Bup[(-)]$ would be more delicate! In summary, we have the following picture:
\begin{equation}
\label{eq:big-picture}%
\vcenter{\xymatrix@C=1.2em@M=.6em@R=1.5em{
\epslp_{\cM}\colon \Alp[\cM]\to \cM
& \Alp 
& \mackfun{\GG} \ar@{^(->}[d]_-{\adj\ }^-{\forget}
& \Bup
\\
\eta=\etalp_{_{\cA}}\colon \cA\to \Alp \ar@{}[u]|(.5){\textrm{(Thm\,\ref{Thm:Alp-UP-intro})}}
& \cA \ar@{|->}[u]
& \resfun{\GG} \ar@{^(->}[d]^-{\ \adj}_-{\forget} \ar@<1em>@/^1em/[u]^-{\Alp[(-)]}
&& \etaup_{_{\cA}}\colon \cA\to \Aup
\\
&& \conjfun{\GG} \ar@<-1em>@/_1em/[u]_-{\Bup[(-)]}
& \cB \ar@{|->}[uu]_(.75){\textrm{(Thm.\,\ref{Thm:Bup-Mackey-intro})}}
& \eps=\epsup_{_{\cB}}\colon \Bup\to \cB  \ar@{}[u]|(.5){\textrm{(Thm\,\ref{Thm:Bup-UP-intro})}}
}}\kern-1.2em
\end{equation}
In the center we displayed vertically the two forgetful functors, from Mackey 2-functors to restriction 2-functors, and from restriction 2-functors to conjugation 2-functors. These forgetful 2-functors have biadjoints,~$\Alp[(-)]$ and~$\Bup[(-)]$ respectively.
Hence we have units and counits, decorated with~${}^\oplus$ and~${}_\oplus$ to distinguish the two cases.
The unit~$\etalp_{\cA}$ is the~$\eta$ of~\Cref{Thm:Alp-UP-intro}, and the counit~$\epsup_{\cB}$ is~$\eps$ in~\Cref{Thm:Bup-UP-intro}. The unit $\etaup_{\cA}$ will play a role below.
For completeness, we mention the counit $\epslp_{\cM}\colon\Alp[\cM]\to \cM$ in~$\mackfun{\GG}$ for every Mackey 2-functor~$\cM$.

The \emph{existence} of our biadjoints may surely be obtained by general theory: Kan extensions of bicategories~\cite{Street74,Kelly05} and `swallowtail equations' in tricategories~\cite{GordonPowerStreet95}.
Instead of unpacking all this machinery in our setting, we give the \emph{explicit} \Cref{Cons:Alp-category,Cons:Bup-category}, for~$\Alp$ and~$\Bup$ respectively.
We pay the price by having to prove the biadjunctions, mostly by constructing the (co)units. Having explicit formulas turns out to be useful for computations anyway.

Another thing that is not provided by the abstract machinery is the little miracle of~\Cref{Thm:Bup-Mackey-intro}: The right biadjoint $\Bup[(-)]$ lands in Mackey 2-functors.
We use this fact to construct a comparison between the two mackeyfications~$\Alp$ and~$\Aup$, when they both make sense, \ie for~$\cA$ a restriction 2-functor.
Indeed, we already mentioned in~\eqref{eq:big-picture} the unit for the second adjunction~$\etaup_{_{\cA}}\colon \cA\to \Aup$ in~$\resfun{\GG}$.
Since the 2-functor $\Aup$ is kind enough to be Mackey, we can apply~\Cref{Thm:Alp-UP-intro} to~$\cM=\Aup$ and~$t=\etaup_{_{\cA}}$ to get a~$\hat{t}\colon \Alp\to \cM=\Aup$.
\begin{Thm}[{\Cref{sec:mark}}]
\label{Thm:mark-intro}%
For every restriction 2-functor~$\cA$, there exists a canonical morphism~$\mu_{\cA}\colon \Alp\to \Aup$ of Mackey 2-functors:
\[
\mu_{\cA}\,\overset{\textrm{\rm def}}{=}\,\widehat{\big(\etaup_{_{\cA}}\big)}\colon \Alp \to \Aup
\]
such that the following composite in $\conjfun{\GG}$ is the identity of~$\cA$:
\[
\xymatrix{
\cA \ar[r]^-{\etalp_{\cA}} & \Alp \ar[r]^-{\mu_{\cA}} & \Aup \ar[r]^-{\epsup_{\cA}} & \cA.
}
\]
\end{Thm}
We give an explicit formula for $\mu_{\cA}$ in~\Cref{Thm:mark-explicit}. Since that formula resembles the classical `mark homomorphism' of~\cite[Section~2]{Boltje98} we call $\mu_{\cA}\colon \Alp\to \Aup$ the \emph{mark transformation}.
It is usually not an equivalence (\Cref{Exa:mark-noniso}).

To appreciate these results, let us pick the simplest restriction 2-functor~$\cA$ we can think of, namely a constant one. Even with this rather dull input, the output is quite interesting. The following statement is a summary of~\Cref{sec:examples}.

\begin{Thm}
\label{Thm:Burnside-intro}%
Let $\GG$ be the 2-category of finite groups, \emph{faithful} homomorphisms, and conjugations as 2-cells. Let $\kk$ be a field and $\cA(G)$ the additive category of finite-dimensional $\kk$-vector spaces, independently of~$G$ in~$\GG$, with identity restrictions.
\begin{enumerate}[\rm(a)]
\item
\label{it:Burnside-lp}%
The category~$\Alp(G)$ is equivalent to the ordinary $\kk$-linear \emph{Burnside category} $\Burn(G)$ of~$G$, whose objects are finite $G$-sets~$X$, and whose morphisms from~$X$ to~$Y$ are  $\kk$-linear combinations of spans~$X\lto Z\to Y$; see~\Cref{Rec:Burnside-cat}.
\smallbreak
\item
\label{it:Burnside-up}%
The category $\Aup(G)$ is equivalent to the sum $\bigoplus_{(H)}\kk \big(N_G(H)/H\big)\mmod$, over any choice of representatives~$H$ of subgroups of~$G$ up to conjugation, of the categories of $\kk$-linear representations of the Weyl group of~$H$ in~$G$.
\smallbreak
\item
\label{it:Burnside-mark}%
The mark transformation~$\mu_{\cA,G}\colon \Alp(G)\to \Aup(G)$ sends a $G$-set~$X$ to the tuple of $\kk(N_G(H)/H)$-modules, whose $H$-th entry is the permutation module~$\kk(X^H)$.
\end{enumerate}
\end{Thm}

It is remarkable that our abstract~\Cref{Thm:Alp-UP-intro,Thm:Bup-UP-intro} turn a trivial~$\cA$ into two beautiful equivariant categories~$\Alp$ and~$\Aup$ and that~\Cref{Thm:mark-intro} recovers arguably the most interesting connection $\mu_{\cA}\colon \Alp\to \Aup$ between them.

Our long-term project is to categorify Boltje's canonical induction.
So far, we present in this paper the two mackeyfications and the mark transformation.
We prove universal properties that justify the constructions conceptually.
Our answers are not some tricategorical abstractions but are described explicitly.
And they recover interesting categories in examples.
In view of the current page-length, we make a cut at this point and defer further investigation to subsequent work.

The published literature closest to our theme seems to be the works of Boltje, Raggi-C\'ardenas and Valero-Elizondo~\cite{BoltjeEtAl19} and of Calde\-r\'on~\cite{Calderon25}.
These papers generalize~\cite{Boltje98} to (fibered) biset functors. Yet, they do not invoke Mackey 2-functors and remain mostly 1-categorical, treating linear functors from specific biset categories to categories of modules.

The outline of this article is now easy to read off the table of contents. Let us only add a few comments.
As in~\cite{BalmerDellAmbrogio20}, we use the language of finite groupoids to streamline Mackey formulas; see~\Cref{sec:basics}.
Two notions will play an important role in the construction of~$\Alp$. First we have `local equivalences' of groupoids, the smallest class of morphisms that is compatible with additivity and contains group isomorphisms; see~\Cref{ssec:leq}.
Secondly, we have `traces' in the context of Frobenius adjunctions; see~\Cref{sec:trace}.

\begin{AI}
None of the creative work in this paper involved artificial intelligence. Most proofs consist in guessing appropriate constructions, very often units and counits of adjunctions and biadjunctions.
Such guesses create a lot of drudge work: The constructions must be well-defined, functorial, natural, and compatible with the relevant 2-categorical structures, including comma categories and span categories. Those verifications are proverbially `left to the reader'. Nowadays they can be replicated by generative artificial intelligence.
We did test ChatGPT on such verifications, with great success.

In any case, the authors retain full responsibility for mathematical correctness.
\end{AI}

\begin{Ack}
We thank Ivo Dell'Ambrogio for precious comments.
\end{Ack}

\goodbreak

%
\section{Recollections and preparations}
\label{sec:basics}%
%
\subsection{Basics on 2-categories}
\label{ssec:2-categories}%
%
\begin{Conv}
We refer to 0-cells as `objects', we refer to 1-cells as `morphisms', and we let 2-cells be 2-cells.
As already said, the only 2-categories that we employ here are variations on the basic example of the 2-category $\CAT$ of categories, whose objects are categories, whose morphisms are functors and whose 2-cells are natural transformations.
For instance:
\begin{enumerate}[\rm(1)]
\item
We denote by~$\gpd$ the 2-category of finite groupoids (with finitely many objects and finitely many morphisms), functors and natural transformations.
As every 2-cell is invertible, $\gpd$ is a so-called (2,1)-category.
\smallbreak
\item
We already used in the Introduction the 2-category~$\ADD$ of (large) additive categories, additive functors and natural transformations.
\end{enumerate}
\end{Conv}
\begin{Conv}
\label{Conv:isos}%
We denote isomorphisms by~$\cong$, be they invertible morphisms or invertible 2-cells. We denote (1-cell) equivalences by~$\simeq$.
To lighten notation, we follow standard practice and parsimoniously denote by~$=$ a few canonical isomorphisms that can safely be treated as identities.
\end{Conv}

\begin{Not}
\label{Not:cast}%
In a 2-category, it is common to write only~$\alpha$ for the 2-cell~$\alpha$ suitably whiskered, when the whiskering is clear from context. When composing 2-cells that only match after whiskering, we shall write~$\beta\circledast\alpha$ instead of~$\beta\circ\alpha$, to indicate that $\alpha$ and~$\beta$ are adjusted and then composed. For instance, in the diagram
\[
\beta\cast\alpha\colon\kern5em
\vcenter{
\xymatrix@C=2em{
\sbull \ar[r]_-{q}
&
\sbull \ar[r]_-{s} \ar@/^1.3em/[rr]^-{r}_-{\Downarrow{\alpha}}
&
\sbull \ar[r]_-{t} \ar@/_1.6em/[rr]_-{v}^-{\Downarrow{\beta}}
&
\sbull \ar[r]_-{u}
&
\sbull \ar[r]_-{w}
&
\sbull
}}
\]
the 2-cells $\alpha\colon r\To t s$ and~$\beta\colon u t\To v$ cannot be composed but $\beta\cast\alpha$ has the obvious meaning~$(\beta s)\circ(u\alpha)\colon u r\To v s$. Going the full length, we shall often denote by~$\beta\cast\alpha$ the obvious 2-cell~$(w\beta s q)\circ(w u\alpha q)\colon w u r q\To w v s q$ in the above picture.
\end{Not}

\begin{Rec}
\label{Rec:2-functors}%
For every 2-category~$\GG$ (for instance $\GG=\gpd$) we denote by
\[
\twofun({\GG}^{\op},{\ADD})
\]
the 2-category of all contravariant 2-functors from~$\GG$ to~$\ADD$.
Details can be found in~\cite[Appendix~A.1]{BalmerDellAmbrogio20}.
An object~$\cA$ in~$\twofun({\GG}^{\op},{\ADD})$ consists of an additive category~$\cA(G)$ for every object~$G$ in~$\GG$, an additive restriction functor $u^*=\cA(u)\colon \cA(G)\to \cA(H)$ for every morphism~$u\colon H\to G$ in~$\GG$, and a natural transformation $\alpha^*=\cA(\alpha)\colon$ $u^*\To (u')^*$ for every 2-cell~$\alpha\colon u\To u'$. This data is required to be strictly compatible with identities and compositions.
(There is a somewhat cumbersome extension of the theory to \emph{pseudo}-functors~$\cA$.
Note that pseudo-functors can be strictified, by~\cite[Section~4.2]{Power89}.)

The morphisms in~$\twofun({\GG}^{\op},{\ADD})$ are always \emph{pseudo-natural transformations} $t\colon \cA\to \cA'$; they consist of an additive functor $t_{G}\colon \cA(G)\to \cA'(G)$ for every object~$G$ in~$\GG$ and a natural isomorphism $t_u\colon \cA'(u)\circ t_{G}\isoTo t_{H}\circ\cA(u)\colon \cA(G)\to \cA'(H)$ for every morphism~$u\colon H\to G$ in~$\GG$, expressing the compatibility of~$t$ with the two restrictions in~$\cA$ and~$\cA'$, these are subject to `obvious' axioms, see~\cite[A.1.15]{BalmerDellAmbrogio20}.
We call such a morphism~$t$ \emph{strictly 2-natural} if the compatibility isomorphisms~$t_u$ are identities.

The 2-cells $m\colon t\To t'\colon \cA\to \cA'$ are the \emph{modifications}; they consist of 2-cells in~$\ADD$ (natural transformations) $m_G\colon t_G\To t'_{G}\colon \cA(G)\to \cA'(G)$ for every object~$G$ in~$\GG$ compatible with the isomorphisms~$t_u$ in the `obvious' sense, see~\loccit
\end{Rec}

\begin{Rec}
\label{Rec:biadjunction}%
Let $\Phi\colon  \mathcal{K} \to \mathcal{L}$ and $\Psi\colon  \mathcal{L} \to \mathcal{K}$ be two 2-functors between 2-categories, or more generally pseudo-functors between bicategories. A \emph{biadjunction} $\Phi\adj \Psi$ consists of an equivalence of categories
\begin{equation}
\label{eq:biadjunction}%
\mathcal{L}(\Phi X,Y)\simeq\mathcal{K}(X,\Psi Y)
\end{equation}
that is pseudo-natural in both variables. We shall construct such biadjunctions by giving a pair of pseudo-natural transformations, the unit $\eta\colon  \Id_{\mathcal{K}} \To \Psi\Phi$ and the counit $\eps\colon  \Phi\Psi \To \Id_{\mathcal{L}}$, and two invertible modifications $\alpha\colon \Id_{\Phi}\Rrightarrow(\eps\Phi)\circ (\Phi\eta)$ and $\beta\colon \Id_{\Psi}\Rrightarrow(\Psi\eps)\circ(\eta\Psi)$. These must satisfy coherence conditions, called the `swallowtail equations', see~\cite{GordonPowerStreet95}. In our applications, the two triangle composites will be equal to the identity, as pseudo-natural transformations, and therefore the modifications~$\alpha$ and~$\beta$ may be chosen to be identities. Furthermore, in both applications (\Cref{Prop:Alp-unit-counit,Prop:Bup-unit-counit}), one of the unit and counit is strictly 2-natural. In such a case, the swallowtail equations are automatic.
\end{Rec}

\subsection{Basics on groupoids and isocommas}\
\label{ssec:groupoids}%
%
\begin{Not}
\label{Not:gpd}%
We denote finite groupoids by letters such as $G,H,K,\ldots$.
We use the symbol $\into$ to indicate faithfulness.
We reserve $\hook$ to emphasize inclusion.
We write $\pi_0(G)$ for the finite set of connected components of~$G$.
We have a decomposition $G=\sqcup_{H\in\pi_0(G)}H $ and denote the fully
faithful inclusions by~$\incl_H\colon H\hook G$.

\end{Not}

\begin{Exa}
\label{Exa:gps}%
Groups are viewed as one-object groupoids. Functors between them are group homomorphisms.
Given two group homomorphisms $f_1,f_2\colon H\to G$, the 2-cells $f_1\To f_2$ are the conjugations $\gamma_{g}\colon f_1\isoTo f_2$, where $g\in G$ satisfies $f_2={}^{g\!}f_1$.
A homomorphism is an equivalence in~$\gpd$ if and only if it is an isomorphism.

Conjugation by an element $g\in G$ provides both the morphisms~$c_g:={}^g(-)$, say, from~$H\le G$ to~$K\le G$ when ${}^g H\le K$, as well as the above 2-cells~$\gamma_g$.
Note that even when $f_1=f_2$ and ${}^g f_1=f_1$ the 2-cell is trivial~$\gamma_g=\id_{f_1}$ only for~$g=1$.

For instance, suppose that $g\in N_G(H)$ normalizes~$H$. Then we have an automorphism $c_g\colon H\isoto H$, which is trivial $c_g=\id_H$ only if~$g\in C_G(H)$ centralizes~$H$. This~$c_g$ is related to the identity by the 2-cell $\gamma_g\colon \id_H\isoTo c_g$ only when~$g\in H$, and even when $g\in Z(H)$ is central, the 2-cell $\gamma_g\colon \id_H\isoTo \id_H$ is non-trivial unless~$g=1$.
\end{Exa}

\begin{Rem}
Let us make a heuristic comment on the role that 2-cells play in 2-functors~$\cA\colon \gpd^\op\to \ADD$, in particular over the 2-subcategory of groups, as in~\Cref{Exa:gps}.
Functoriality yields restrictions~$u^*=\cA(u)\colon \cA(G)\to \cA(H)$ for all group homomorphisms~$u\colon H\to G$, so that $(u_1u_2)^*=u_2^* u_1^*$.
For 2-cells~$\gamma_g\colon u\isoTo v$, we obtain isomorphisms $\gamma_g^*=\cA(\gamma_g)\colon u^*\isoTo v^*$.
Formally, these $\gamma_g^*$ can be thought of as more variance (`higher restrictions') but they are perhaps better understood as \emph{relations}.
For instance $\gamma_g\colon\id_{G}\isoTo c_g\colon G\to G$ provides an isomorphism $c_g^*\cong \id_{\cA(G)}$, thus trivializing the conjugation action~$c_g^*$ of~$G$ on~$\cA(G)$.
If we stripped~$\gpd$ from all non-identity 2-cells, we could still consider 1-functors from group(oid)s to additive categories, with no imposed relations between~$u^*$ and~$v^*$ when~$u\neq v$. Remembering 2-cells forces $G$-conjugate homomorphisms~$H\to G$ to yield isomorphic restrictions~$\cA(G)\to \cA(H)$.
\end{Rem}

\begin{Rec}\label{Def:isocomma}
Let $i\colon H\to G$ and $u\colon K\to G$ be morphisms in~$\gpd$ with common target, also known as a \emph{cospan}~$H\xto{i}G\xlto{u}K$.
The \emph{isocomma groupoid} $(i/u)$
\begin{equation}
\label{eq:isocomma}
\vcenter{
\xymatrix@C=1em@R=1em{
& (i/u) \ar@{ ->}[rd]^{\pr_2} \ar@{ ->}[ld]_{\pr_1} \ar@{}[dd]|{\isocell{\gamma_{i/u}}} & \\
H \ar@{ ->}[rd]_i &  & K \ar@{ ->}[ld]^u \\
& G &
}}
\end{equation}
has for objects the triples $(a,b,g)$, where $a\in \Obj(H)$ and $b\in \Obj(K)$, and $g\colon i(a)\isoto u(b)$
is an isomorphism in~$G$. Its morphisms $(a_1,b_1,g_1)\to (a_2,b_2,g_2)$ are pairs of morphisms~$(h\colon a_1\to a_2\,,\,k\colon b_1\to b_2)$ such that $g_2 i(h)=u(k)g_1$ in~$G$.
The functors $\pr_1\colon (i/u)\to H$ and $\pr_2\colon (i/u)\to K$ are the obvious projections, and the natural isomorphism $\gamma_{i/u}\colon i\pr_1\To u\pr_2$ is given by~$g$ at each object~$(a,b,g)$.
Sometimes the emphasis is on~$H$ and~$K$ and we shall write
\[
H\times_G K:=(i/u)
\]
when the `structure morphisms' $i\colon H\to G$ and~$u\colon K\to G$ are clear from context.
\end{Rec}

\begin{Rem}
\label{Rem:why-gpd}%
The main reason for using groupoids in~\cite{BalmerDellAmbrogio20} is that the above groupoid~$H\times_G K=(i/u)$ is usually not a group even if $G$ is, and if $H$ and~$K$
are subgroups. Indeed, in that case,
$\pi_0((i/u))$ is in canonical bijection with~$\KGH$. For each double coset~$C\in\KGH$, the component of~$(i/u)$
corresponding to~$C$ is non-canonically equivalent to $H\cap K^g$ for any choice of~$g\in C$.

Mackey formulas relate on one side induction from $H$ to~$G$ followed by restriction to~$K$, and on the other side the sum of the restrictions to the various~$H\cap K^g$ followed by induction along~${}^g(-)\colon H\cap K^g\into K$. With the isocomma~\eqref{eq:isocomma} in hand, this Mackey relation becomes the much simpler $u^*\circ i_*\cong (\pr_2)_*\circ \pr_1^*$.
The language of groupoids avoids unhealthy choices of coset representatives.

The small price to pay for using finite groupoids instead of finite groups is to make sure that all (contravariant) 2-functors $\cA\colon \gpd^\op\to \ADD$ that we use are \emph{additive}. See~\Cref{Rem:additive=>groups}. In the above situation, the composite $(\pr_2)_*\circ \pr_1^*$ then becomes a sum of~$(\pr_{2,C})_*\circ (\pr_{1,C})^*$ indexed by~$C\in\KGH$, as one expects.
\end{Rem}

\begin{Not}
\label{Not:<>}%
Consider an isocomma~\eqref{eq:isocomma}.
For a groupoid~$T$, every morphism $w\colon T\to (i/u)$ yields two morphisms $w_1:=\pr_1 w\colon T\to H$, $w_2:=\pr_2 w\colon T\to K$ and a 2-cell $\alpha:=\gamma_{i/u}w\colon i w_1\isoTo u w_2$. Conversely, every such triple~$(w_1,w_2,\alpha)$ defines a unique morphism~$w\colon T\to (i/u)$ given by~$w(a)=(w_1(a),w_2(a),\alpha_a)$ for every object~$a\in T$ and $w(f)=(w_1(f),w_2(f))$ on morphisms. We denote it by
\[
\langle w_1,w_2,\alpha \rangle\colon T\to (i/u).
\]
For 2-cells, given two morphisms $w,w'\colon T\to (i/u)$, say $w=\langle w_1,w_2,\alpha\rangle$ and $w'=\langle w_1',w_2',\alpha'\rangle$ and a pair of transformations~$\theta_i\colon w_i\To w_i'$ for $i=1,2$, such that $\theta_2\cast\alpha=\alpha'\cast\theta_1\colon i w_1\To u w_2'$, there exists a unique natural transformation
\[
\langle\theta_1,\theta_2\rangle\colon w\To w'
\]
such that $\pr_i\langle\theta_1,\theta_2\rangle=\theta_i\colon w_i\To w_i'$ for $i=1,2$.
In summary, we get an isomorphism of groupoids
\begin{equation}
\label{eq:isocomma-translate}%
\gpd\big(T,(i/u)\big) \isoto \big(\gpd(T,i)/\gpd(T,u)\big)
\end{equation}
where $\gpd(G_1,G_2)$ is the groupoid of functors from~$G_1$ to~$G_2$ and natural transformations and $\gpd(G_1,f)$ is post-composition by~$f$ as usual.
\end{Not}

\subsection{Local equivalences}\
\label{ssec:leq}%

We now highlight an elementary but important class of morphisms in~$\gpd$.
\begin{Def}
\label{Def:loc-equiv}%
A morphism~$s\colon H\to G$ in~$\gpd$ is called a \emph{local equivalence} if for every component $L\in\pi_0(H)$ the composite $L\into H\xto{s} G$ is fully faithful. In other words, every connected component of~$H$ is equivalent via~$s$ to some connected component of~$G$.
We write $s\colon H\apprto G$ to say that $s$ is a local equivalence. Local equivalences are closed under composition and 2-isomorphism.
\end{Def}

\begin{Rem}
\label{Rem:local-equiv-iso}%
In $\gpd$, a functor $s\colon P\to Q$ is a local equivalence if and only if
$
\End_P(a)\isoto\End_Q(s(a))
$
is an isomorphism for every object $a\in P$.
In particular, the only local equivalences between groups (one-object groupoids) are isomorphisms.
\end{Rem}

Let us give a more abstract description, that shall work beyond~$\gpd$.

\begin{Exa}
\label{Exa:folding}%
Let $K\in \gpd$ and $n\ge 0$. Then the canonical \emph{folding} functor $\nabla_K^{(n)}=(\id_K,\ldots,\id_K)\colon K^{\sqcup n}=K\sqcup\cdots\sqcup K\too K$ is a local equivalence.
For $K\neq\varnothing$, the morphism~$\nabla_K^{(n)}$ is only fully faithful for $n\le 1$ and an equivalence for~$n=1$.
\end{Exa}

These foldings are essentially the only examples of local equivalences:

\begin{Lem}
\label{Lem:leq-folding}%
Let $s\colon H\apprto G$ be a local equivalence in~$\gpd$. Then there exists a unique
function $n\colon \pi_0(G)\to \bbN=\bbZ_{\ge0}$ and a (strict) factorization of~$s$ as
\begin{equation}
\label{eq:factor-locequiv}%
\vcenter{
\xymatrix@C=1em{
H \ar[rr]^-{\tilde{s}}_-{\simeq}
&& {\displaystyle\bigsqcup_{K\in\pi_0(G)}K^{\sqcup n(K)}} \ar[rrrr]^-{\bigsqcup_K\ \nabla_K^{(n(K))}}
&&&& {\displaystyle\bigsqcup_{K\in\pi_0(G)}K} \ \ar@{=}[r] & \ G
}}
\end{equation}
where $\tilde{s}$ is an equivalence and each~$\nabla_K^{(n)}\colon K^{\sqcup n}\apprto K$ is a folding (\Cref{Exa:folding}),
and this factorization is unique up to permutation of the factors~$ K^{\sqcup n}$.
\end{Lem}
\begin{proof}
Let $n(K)=|\pi_0(s)\inv(K)|$ be the number of components $L$ of~$H$ for which $s$ factors via $\incl_K\colon K\into G$ (in $\gpd$ this simply means $s(L)\subseteq K$ on objects), in which case $s$ restricts to an equivalence between~$L$ and~$K$. These equivalences define~$\tilde{s}$.
The numbering of $\{1,\ldots,n(K)\}\isoto \pi_0(s)\inv(K)$ is unique up to permutation.
\end{proof}

\begin{Rem}
When $s$ is not essentially surjective, one could separate the components~$K$ in the essential image of~$s$, for which $n(K)>0$, from those with~$n(K)=0$.
\end{Rem}

Local equivalences (\Cref{Def:loc-equiv}) are closed under pull-back:
\begin{Lem}
\label{Lem:leq-pb}%
\label{Lem:local-equiv-under-isocomma}%
Consider an isocomma in~$\gpd$ as in~\eqref{eq:isocomma}.
If $u$ is a local equivalence then~$\pr_1$ is a local equivalence.
\end{Lem}
\begin{proof}
For any $(a,b,g)\in (i/u)$, the following diagram of sets is a pullback
\[
\vcenter{
\xymatrix@C=2em@R=1em{
\End_{(i/u)}(a,b,g) \ar@{ ->}[rrr]^-{\pr_2}\ar@{ ->}[d]_{\pr_1}
&&& \End_K(b) \ar@{ ->}[d]^{u}
\\
\End_H(a)\ar@{ ->}[r]^-{i}
& \End_G(i(a)) \ar@{ ->}[rr]^{g(-)g\inv}_-{\simeq}
&& \End_G(u(b))
}}
\]
by the definition of isocomma. The claim follows from
\Cref{Rem:local-equiv-iso}.
\end{proof}

Every morphism has a part that is a local equivalence. Let us isolate it.
\begin{Def}
\label{Def:approx}%
For every morphism $p\colon X\to Y$ in~$\gpd$, we denote by
\begin{equation}
\label{eq:approx}%
\xymatrix@C=5em{
X=X^{p\approx}\sqcup X^{p\not\approx} \ar[r]^-{\big(p^\approx\quad p^{\not\approx}\big)}
& \ Y
}
\end{equation}
the canonical decomposition of~$X$ where $X^{p\approx}$ is the coproduct of those components of~$X$ on which $p$ is fully faithful and $X^{p\not\approx}$ is the coproduct of all remaining components of~$X$. By construction, $p^\approx$ is a local equivalence $X^{p\approx}\apprto Y$.
It is convenient to refer to~$X^{p\approx}$ as the \emph{$\approx$-locus} of~$p$.
\end{Def}

\begin{Rem}
\label{Rem:approx-faithful}%
Consider two composable morphisms~$X\xto{p}Y\xto{q}Z$ in~$\gpd$. It is unfortunately
not true that $(q\circ p)^\approx=q^\approx \circ p^\approx$ in general.
For instance, it fails if~$q$ is a non-invertible group homomorphism that admits a section and $p$ is one of those sections, for then $X^{(qp)\approx}=X$ but~$X^{p\approx}=\varnothing$.
Yet, if~$q$ is \emph{faithful} then $X^{(qp)\approx}\subseteq X^{p\approx}$ and
$p^{\approx}(X^{(qp)\approx})\subseteq Y^{q\approx}$ and therefore
 $(q\circ p)^\approx=q^\approx \circ p^\approx$. (These relations are easy from~\Cref{Rem:local-equiv-iso}.)
Consequently, it is preferable to only use faithful morphisms when dealing
with~$(-)^\approx$ to avoid non-functoriality.
\end{Rem}

\subsection{The (2,1)-category~$\GG$}\
\label{ssec:GG}%

Although the reader can focus attention on~$\gpd$ for the rest of the paper, our results actually hold more generally, as in~\cite{BalmerDellAmbrogio20} and Dell'Ambrogio~\cite{DellAmbrogio22b}.
\begin{Conv}
\label{Conv:GG}%
We reset the notation of the introduction and denote by
\[
\GG
\]
a (2,1)-category of `finite groupoids of interest', from the following list.
\begin{enumerate}[\rm(1)]
\item
\label{it:exa-gpd}%
We can take $\GG=\gpd$, all finite groupoids, as in~\Cref{Not:gpd}.
\smallbreak
\item
\label{it:exa-gpdf}%
We can take $\GG=\gpdf$ the 2-full subcategory of~$\gpd$ with only faithful morphisms.
This context is useful with equivariant categories~$\cA(G)$ that only have restrictions to subgroups, but no restriction along general group homomorphisms, like $\cA(G)=\mathsf{stmod}(kG)$ the stable module category of $kG$-modules.
\smallbreak
\item
\label{it:exa-gpdf/G}%
We can take $\GG=\gpdf_{/\Gamma}$ the 2-category of groupoids faithfully embedded into a fixed `ambient group'~$\Gamma$ as in~\cite[Definition~B.0.6]{BalmerDellAmbrogio20}. This allows a `local' theory, only involving the subgroups of~$\Gamma$ and conjugation-inclusions, whereas the more `global' examples~\eqref{it:exa-gpd} and~\eqref{it:exa-gpdf} involve all finite groups.
\end{enumerate}
We believe that our results extend to \emph{spannable} (2,1)-categories~$\GG$ in the sense of Dell'Ambrogio~\cite[Definition~3.11]{DellAmbrogio22b}, possibly adding the hypothesis that every~$\pi_0(G)$ is finite, but we have not verified the details.
\end{Conv}

\begin{Rem}
Everything we said in this section about~$\gpd$, namely connected components (including finiteness of~$\pi_0$), faithfulness, isocommas, local equivalences, and the decomposition~\eqref{eq:approx}, also makes sense in any (2,1)-category~$\GG$ as in~\Cref{Conv:GG}, via the forgetful functor to~$\gpd$.
For instance, our~$\GG$ contains all local equivalences of its underlying groupoids.
For isocommas one can use that every object~$T$ defines a 2-functor $\GG(T,-)\colon \GG\to \gpd$, to translate statements between~$\gpd$ and~$\GG$ via the analogue of~\eqref{eq:isocomma-translate}; this is explained in~\cite[Remark~2.1.6]{BalmerDellAmbrogio20}.
\end{Rem}

If we need to restrict to local equivalences, we shall use the following notation.
\begin{Def}
\label{Def:GGi}%
We denote by~$\GGi$ the 2-full subcategory with the same objects as~$\GG$ but with only local equivalences (\Cref{Def:loc-equiv}) as morphisms.
\end{Def}
\begin{Rem}
\label{Rem:GGi-is-a-GG}%
In fact, $\GGi$ could be added to the list of~\Cref{Conv:GG}, as an admissible input (2,1)-category~$\GG$, although a somewhat dull one.
It is in some sense a minimal (2,1)-category containing all group isomorphisms (as the $\GGiso$ of the Introduction) as well as the inclusions $\incl_{H_i}\colon H_i\into H_1\sqcup H_2$ necessary to formulate additivity as in \Cref{Rem:why-gpd}. See also \Cref{Def:restriction-2-functor} below.
\end{Rem}

\subsection{Restriction and conjugation 2-functors}\
\label{ssec:res-conj}%

%
\begin{Def}
\label{Def:restriction-2-functor}%
A \emph{restriction 2-functor}~$\cA\colon \GG^\op\to \ADD$ is a 2-functor that is \emph{additive}, meaning that the canonical functor of additive categories
\[
(\incl_1^*\quad \incl_2^*)\colon \cA(G_1\sqcup G_2)\too\cA(G_1)\oplus\cA(G_2)
\]
is an equivalence for every $G_1,G_2\in\GG$, where $\incl_j\colon G_j\into G_1\sqcup G_2$ denotes the inclusion. This definition implies~$\cA(\varnothing)\simeq0$.
We denote by
\[
\resfun{\GG}
\]
the full 2-subcategory of~$\twofun({\GG}^{\op},{\ADD})$ of restriction 2-functors $\GG^\op\to \ADD$, pseudo-natural transformations and modifications as in~\Cref{Rec:2-functors}.
In~\cite{BalmerDellAmbrogio20} this $\resfun{\GG}$ was denoted $~2\Fun_{\sqcup}(\GG^\op,\ADD)$.
\end{Def}

\begin{Def}
\label{Def:conjugation-2-functor}%
A \emph{conjugation 2-functor}~$\cB$ on~$\GG$ is an additive 2-functor on~$\GGi$, that is,
a 2-functor~$\cB\colon \GGi^\op\to \ADD$ such that
the canonical functor
\[
(\incl_1^*\quad \incl_2^*)\colon
\cB(G_1\sqcup G_2)\too \cB(G_1)\oplus\cB(G_2)
\]
is an equivalence for all~$G_1,G_2\in\GG$ as above. We denote by
\[
\conjfun{\GG}
\]
the 2-subcategory of~$\twofun(\GGi^{\op},{\ADD})$ of conjugation 2-functors (\Cref{Rec:2-functors}) with the same morphisms and 2-cells.
\end{Def}

\begin{Rem}
\label{Rem:additive=>groups}%
Additivity of~$\cA$ means that $\cA(G)\simeq\oplus_{H\in\pi_0(G)}\cA(H)$ via the functors $\cA(\incl_{H})\colon \cA(G)\to \cA(H)$ given by the inclusion of components~$H\in\pi_0(G)$.
Hence additive 2-functors~$\cA$ on groupoids are essentially characterized by their values on groups (see \cite[Lemma~4.3.2]{BalmerDellAmbrogio20}).
And the same holds for conjugation 2-functors since the $\incl_{H}\colon H\apprto G$ are local equivalences, hence belong to~$\GGi$.
This explains the simplified language used in the Introduction, where $\GG$ only consisted of finite groups and where $\GGiso$ meant `isomorphism-only' (see~\Cref{Rem:local-equiv-iso}).

We do unpack some of our constructions $\Alp(G)$ and~$\Bup(G)$, for $G$ a group, when we discuss examples in~\Cref{sec:examples}. This is fine for one~$G$ at a time but the 2-functors $\Alp(-)$ and~$\Bup(-)$ will be much cleaner when we allow all finite groupoids as input, instead of insisting on finite groups only, especially when we study induction and Mackey formulas (see~\Cref{Rem:why-gpd}).
\end{Rem}

\subsection{Comma 2-categories}\
\label{ssec:commas}%

We recall a classical construction for the (2,1)-category~$\GG$ of~\Cref{Conv:GG}. See~\cite[Definition~A.1.21]{BalmerDellAmbrogio20} if necessary.

\begin{Def}
\label{Def:GG/G}%
Let $G\in \GG$ be a fixed object. The \emph{comma 2-category}
\[
\GG/G
\]
consists of objects of~$\GG$ \emph{over~$G$}, that is, morphisms~$(H\to G)$ in~$\GG$ with target~$G$.
We shall often just write~$H$ and call the tacit morphism $H\to G$ the \emph{structure morphism} of~$H\in\GG/G$. To invoke the structure morphism of~$H$, we write it
\[
\stH\colon H\to G.
\]
A 1-morphism $f\colon H\to K$ in~$\GG/G$ consists of a morphism~$f\colon H\to K$ in~$\GG$ together with a given compatibility isomorphism with the structure morphisms, namely a 2-cell $\stH\isoTo \stK f\colon H\to G$, called the \emph{structure 2-cell} of~$f$. Again, we usually keep it tacit and just write~$f\colon H\to K$. When needed, the structure 2-cell is denoted
\[
\st{f}\colon \stH\isoTo \stK f.
\]
The 2-cells~$f\To f'\colon H\to K$ in~$\GG/G$ are simply given by 2-cells~$\alpha\colon f\To f'$ in~$\GG$ compatible with the structure 2-cells, meaning $\alpha\cast\st{f}=\st{f'}$ as 2-cells~$\stH\To \stK f'$. (Recall~\Cref{Not:cast} for~$\cast$.)
There is no tacit information in 2-cells.

The compositions in~$\GG/G$ are the evident ones and make $\GG/G$ a (2,1)-category.
\end{Def}

\begin{Exa}
We already encountered an example in our~\Cref{Conv:GG}, with the 2-category $\gpdf_{/\Gamma}$ of groupoids faithfully embedded in a fixed `ambient group'~$\Gamma$, which is the $\GG/G$ of~\Cref{Def:GG/G}, for $\GG=\gpdf$ and $G$ the one-object groupoid~$\Gamma$.
\end{Exa}
\begin{Exa}
\label{Exa:end(H)-in-gpdf/G}%
Let $G$ be a finite group and $H,K\leq G$ be subgroups, viewed as objects of~$\gpdf/G$ with~$\stH=\incl_H\colon H\into G$ and similarly for~$K$.
A morphism from~$H$ to~$K$ in~$\gpdf/G$ is a group homomorphism~$f\colon H\to K$ with structure 2-cell~$\st{f}\colon \incl_H\isoTo \incl_K\circ f\colon H\into G$. The latter must be of the form~$\st{f}=\gamma_{g}$ for some~$g\in G$ such that $\incl_K\circ f={}^g \incl_{H}\colon H\into G$ (\Cref{Exa:gps}) which forces~$f=c_g={}^g(-)\colon H\into K$ and therefore $g\in N_G(H,K)$. In particular, the only (local) equivalences~$H\isoto K$ are the conjugations and all endomorphisms are automorphisms.
In particular, we have an isomorphism
\[
\xymatrix@C=.5em@R=.2em{
N_G(H) \ar[rr]^-{\cong}
&& \End_{\gpdf/G}(H) \ar@{=}[r]
& \Aut_{\gpdf/G}(H)
\\
g \ar@{|->}[rr]
&& \quad c_g\colon H\isoto H
& \textrm{with}\quad\st{c_g}=\gamma_g\,.
}
\]
Note that for~$g\in N_G(H)$ the condition $c_g=\id_{H}$ in~$\gpdf$ alone would mean that $g$ belongs to the centralizer of~$H$ in~$G$ but here we are considering~$c_g$ \emph{with its structure 2-cell~$\st{c_g}=\gamma_g$} and this pair~$(c_g,\gamma_g)$ is only equal to~$\id_{H\into G}$ in~$\gpdf/G$ if $g=1$.

Furthermore, given two such endomorphisms~$c_g,c_{g'}\colon H\to H$ in~$\gpdf/G$, a 2-cell between them corresponds to some $\gamma_h\colon c_g\isoTo c_{g'}$ for~$h\in H$ such that $c_{g'}={}^h c_{g}=c_{h g}$ \emph{and} $\gamma_{g'}=\gamma_{h g}\colon \incl_H\isoTo \incl_H\colon H\into G$, which forces~$g'=hg$.
In short, such a 2-cell exists if and only if~$[g]=[g']$ in~$N_G(H)/H$.
\end{Exa}

\begin{Conv}
\label{Conv:comma-isocomma}%
The 2-category $\GG/G$ inherits isocommas from~$\GG$, with a small subtlety. Indeed, given a cospan $H\xto{i}L\xlto{u}K$ in~$\GG/G$, one can form~$H\times_L K=(i/u)$ in~$\GG$ and this object is clearly `over'~$G$, actually in at least four (isomorphic) ways, via~$\st{H}\circ\pr_1$, or via~$\st{L}\circ i\circ\pr_1$, or via~$\st{L}\circ u\circ\pr_2$, or via~$\stK\circ\pr_2$:
\begin{equation}
\label{eq:isocomma-comma}%
\vcenter{
\xymatrix@C=1.5em@R=1em{
& H\times_L K =(i/u) \ar@{ ->}[ld]_(.65){\pr_1} \ar@{ ->}[rd]^(.65){\pr_2}
 \ar@{}[dd]|-{\isocell{\gamma_{i/u}}}
\\
H \ar@{ ->}@/_1em/[rdd]_{{\stH}} \ar[rd]^-{i}
 \ar@[gray]@{}[rrdd]_(.3){\isocell{\st{i}}}
&& K \ar@{ ->}@/^1em/[ldd]^{{\stK}} \ar[ld]_-{u}
 \ar@[gray]@{}[lldd]^(.3){\lisocell{\st{u}}}
\\
& L \ar[d]^-{\stL}
\\
& {G} &
}}
\end{equation}
We choose the structure morphism for~$H\times_L K$ to be the leftmost path: $\st{H\times_L K}:=\stH\circ \pr_1$.
This choice minimizes the inverses~$(-)\inv$ in the structure 2-cells for the projections~$\pr_i$, which are $\st{\pr_1}=\id_{\stH\pr_1}$ and $\st{\pr_2}=\st{u}\inv\cast\gamma_{i/u}\cast\st{i}$ the obvious composite `across'~\eqref{eq:isocomma-comma}.
In this way, $\gamma_{i/u}$ becomes an actual 2-cell in~$\GG/G$.
\end{Conv}

Let us say a word about how the comma 2-category~$\GG/G$ varies with~$G$.
\begin{Cons}
\label{Cons:j!j*-GG/}%
Fix a 1-morphism $j\colon G'\to G$ in~$\GG$.
We have two 2-functors
\[
j_!\colon \GG/G'\to \GG/G
\qquadtext{and}
j^*\colon \GG/G\to \GG/G'
\]
on the comma 2-categories. On objects, they are given by
\[
j_!H'=H'
\qquadtext{and}
j^*H=G'\times_G H=(j/\stH)
\]
with the obvious structure morphisms $\st{j_!H'}=j\,\st{H'}\colon H'\to G'\to G$ and $\st{j^*H}=\pr_1\colon G'\times_G H\to G'$ the first projection.
On morphisms, they are given by
\[
j_!f'=f'
\qquadtext{and}
j^*\!f=G'\times_G f=\langle\pr_1,f\pr_2,\st{f}\cast\gamma_{j/\stH}\rangle
\]
with structure 2-cell~$\st{j_!f'}=j\st{f'}$ and $\st{j^*f}$ the identity of~$\pr_1=\pr_1\circ \ j^*\!f$:
\[
\vcenter{\xymatrix@C=1em@R=3em{
H' \ar@[gray][d]_-{\gray{\st{H'}}} \ar@{=}[rd]
 \ar@{}[rrrd]|(.25){\ \gray{\isocell{\st{f'}}}}
 \ar[rr]^-{f'}
&& K' \ar@[lightgray][lld]^(.4){\gray{\st{K'}}\!\!\!}|(.65){\hole} \ar@{=}[rd]
&
\\
G' \ar[rd]_-{j}
& j_!H' \ar@[gray][d]_(.3){\gray{\st{j_!H'}}\!\!} \ar[rr]^-{j_!f'} \ar@{}[rd]|(.4){\gray{\isocell{\st{j_!f'}}}}
&& j_!K' \ar@[gray][lld]^-{\gray{\st{j_!K'}}}
&
\\
& G &
}}
\!\!\text{and}
\vcenter{\xymatrix@C=.7em@R=3em{
j^*H=G'\times_G H \ar@[gray][d]_-{\gray{\st{j^*H}=\pr_1}} \ar@<.2em>[rd]|(.5){\vphantom{I_{j_j}}\pr_2}
 \ar@{}[rrrd]|(.25){\ \gray{\quad\isocell{\st{j^*\!f}=\id}}}  \ar@{}[rdd]|(.25){\isocell{\gamma_{j/\stH}}}
 \ar[rr]^-{j^*\!f}
&& j^*K=G'\times_G K \ar@[lightgray][lld]^(.4){\gray{\pr_1}\!\!\!}|(.65){\hole} \ar[rd]^-{\pr_2} \ar@{}[ldd]|(.25){\isocell{\gamma_{j/\stK}}}
&
\\
G' \ar[rd]_-{j}
& H \ar@[gray][d]_-{\gray{\stH}} \ar[rr]^-{f} \ar@{}[rd]|(.3){\gray{\isocell{\st{f}}}}
&& K \ar@[gray][lld]^-{\gray{\stK}}
&
\\
& G &
}}
\]
In this picture, we grayed out all structure morphisms and 2-cells. Note that the top square on the right-hand side commutes on the nose: $\pr_2 \,j^*\!f=f\pr_2$.
Finally, on 2-cells the 2-functors $j_!$ and~$j^*$ are given as follows (recall~\Cref{Not:<>}):
\[
j_!(\alpha)=\alpha
\qquadtext{and}
j^*(\alpha)=G'\times_G\alpha=\langle\id_{\pr_1},\alpha\pr_2\rangle.
\]
These 2-functors are biadjoints $j_!\adj j^*$. The unit $\eta\colon \Id\To j^*j_!$ is the pseudo-natural transformation of 2-functors $\GG/G'\to \GG/G'$ given on every object~$H'\in\GG/G'$ by
\begin{equation}\label{eq:j!j*-eta}%
\eta_{H'}=\langle\st{H'},\id_{H'},\id_{j \st{H'}}\rangle\colon H'\to (j/j \st{H'})\quadtext{with structure cell}\id_{\st{H'}}
\end{equation}
and compatibility 2-isomorphism $j^*j_!(f')\circ \eta_{H'}\isoTo \eta_{K'}\circ f'$ for every $f'\colon H'\to K'$ in~$\GG/G'$ given by $<\st{f'},\id>$.
The counit $\eps\colon j_! j^*\To \Id$ is the pseudo-natural transformation of 2-functors $\GG/G\to \GG/G$ given on every object~$H\in\GG/G$ by
\begin{equation}\label{eq:j!j*-eps}%
\eps_{H}=\pr_2\colon (j/\stH)\to H\quadtext{with structure cell}\gamma_{j/\stH}
\end{equation}
and strict compatibility $f\circ \eps_{H}=\eps_{K}\circ j_!j^* f$ for every $f\colon H\to K$ in~$\GG/G$.
Direct computation gives the \emph{equalities} $(\eps j_!)\circ(j_!\eta)=\id_{j_!}$ and~$(j^*\eps)\circ(\eta j^*)=\id_{j^*}$.
\end{Cons}

\begin{Rem}
\label{Rem:eta-eps-j!j*}%
Observe for later use that if~$j$ is faithful then $\eta_{H'}\colon H'\to G'\times_G H'$ in~\eqref{eq:j!j*-eta} is fully faithful (see~\cite[Example~3.1.15]{BalmerDellAmbrogio20}). Indeed, in that case, $\pr_2$ is faithful and $\pr_2\circ \eta_{H'}=\id_{H'}$ gives the result.
\end{Rem}

\begin{Rem}
\label{Rem:j*-GG}%
The construction $\GG/-$ is strictly functorial covariantly, since $(jk)_!=j_!k_!$ for all composable morphisms $G''\xto{k}G'\xto{j}G$.
It follows that we get pseudofunctorial right adjoints, \ie $\GG/-$ is contravariantly \emph{pseudo}functorial, meaning that the composite $\GG/G\xto{j^*}\GG/G'\xto{k^*}\GG/G''$ only agrees with~$(jk)^*$ up to a coherent natural isomorphism $k^*\circ j^*\cong (jk)^*$. In telegraphic style this reads $G''\times_{G'}(G'\times_{G}-)\cong G''\times_{G}-$.
This canonical isomorphism is close enough to an identity and we treat it as such.
\end{Rem}
\begin{Cons}
\label{Cons:(j/-)=>(k/-)}%
The covariant and contravariant functors of~\Cref{Cons:j!j*-GG/} are 2-functorial in~$\GG$. Let $\alpha\colon j\isoTo k\colon G'\to G$ be a 2-cell. We have a natural transformation $\alpha_!\colon j_!\isoTo k_!$ given by $(\alpha_!)_{H'}=\id_{H'}\colon j_!H'\to k_!H'$ with~$\alpha$ hiding in the structure 2-cell as~$\alpha\st{H'}$. We also have a natural transformation $\alpha^*\colon j^*\isoTo k^*$ given by~$(\alpha^*)_{H}=\langle\pr_1,\pr_2,\gamma_{j/\stH}\cast\alpha\inv\rangle\colon j^*(H)=(j/\stH)\to (k/\stH)=k^*(H)$, with identity (of~$\pr_1$) as structure 2-cell. (Amusingly, $(-/\stH)$ tends more naturally to be \emph{contra}variant in the 2-cells. To stick with overall convention in the field, we have applied the harmless involution $\alpha\mapsto \alpha\inv$ on 2-cells.)
\end{Cons}

\begin{Def}
\label{Def:GGf}%
We denote by $\GGf/G$ the variant of the comma 2-category of~$\GG$ over~$G$, where all structure morphisms~$\st{H}$ are required to be faithful.
This forces all morphisms~$f\colon H\to K$ in~$\GGf/G$ to be faithful too, since $\stK\circ f\cong\stH$ is faithful.
The functor~$j^*$ of~\Cref{Cons:j!j*-GG/} restricts to this setting for any~$j$ but the functor~$j_!$ only makes sense if~$j\colon G'\into G$ is faithful.
In that case, we still have $j_!\adj j^*$.
\end{Def}

\subsection{Span 2-categories}\
\label{ssec:spans}%

We can now consider spans inside the comma 2-category~$\GG/G$. This resembles, but should not be confused with, the bicategory of spans~$\Spanname(\GG)$ considered in~\cite[Chapter~5]{BalmerDellAmbrogio20}. (For the specialists, we are going to consider a 2-category $\Spanname_{\GG/G}(H,K)$ whose 1-truncation is the 1-category of morphisms from~$H$ to~$K$ in the bicategory $\Spanname(\GG/G)$ of~\cite{BalmerDellAmbrogio20}. So we go up one level in the cell tree.)
\begin{Cons}
\label{Cons:Span}%
Fix $G\in \GG$ and two objects~$H,K\in \GG/G$ in the comma 2-category. So $H$ and~$K$ are objects of~$\GG$ with tacit structure morphisms~$\stH\colon H\to G$ and~$\stK\colon K\to G$.
We define the \emph{span 2-category}~$\SpanG(H,K)$ as follows.

An object $P=(P,p_1,p_2)$ in~$\SpanG(H,K)$ is an object $P\in\GG/G$ in the comma 2-category together with two morphisms $p_1\colon P\to H$ and~$p_2\colon P\to K$ in~$\GG/G$, that we shall call the left and right \emph{wing morphisms} of the span:
\begin{equation}
\label{eq:Span-object}%
\vcenter{
\xymatrix@C=4em@R=1em{
& P \ar@{ ->}@/_.5em/[ld]_-{p_1} \ar@{ ->}@/^.5em/[rd]^-{p_2} \ar@[gray]@{ ->}[dd]|-{\gray\stP}  \ar@{}[ldd]|(.45){\gray{\lisocell{\st{p_1}}}} \ar@{}[rdd]|(.45){\gray{\isocell{\st{p_2}}}} & \\
H \ar@[gray]@{ ->}[rd]_{\gray\stH} & & K \ar@[gray]@{ ->}[ld]^{\gray\stK} \\
& \gray{G} &
}}
\end{equation}
By definition of the comma 2-category~$\GG/G$, the data of~$(P,p_1,p_2)$ comes with tacit information~$(\st{})$, namely the structure morphism~$\stP\colon P\to G$ and the two structure 2-cells $\st{p_1}\colon \stP\isoTo \stH p_1$ and~$\st{p_2}\colon \stP\isoTo \stK p_2$ as pictured above, in gray.
These structure 2-cells~$\st{p_i}$ of the wings will be referred to as the \emph{structure 2-cells} of~$P$.

A morphism $s=(s,\sigma_1,\sigma_2)$ between objects $P=(P,p_1,p_2)$ and~$Q=(Q,q_1,q_2)$ in $\SpanG(H,K)$ consists of a morphism $s\colon P\to Q$ in~$\GG/G$ (with its structure 2-cell $\st{s}\colon \st{P} \isoTo \st{Q}s$) together with two \emph{wing cells} $\sigma_i\colon p_i\isoTo q_i s$ in~$\GG/G$ for $i=1,2$.
These `wing cells' do not involve the structure morphisms but must be compatible with structure 2-cells, as every 2-cell in~$\GG/G$.
In expanded form, the morphism~$s=(s,\sigma_1,\sigma_2)$ is a diagram in~$\GG$:
\begin{equation}
\label{eq:Span-morphism}%
\vcenter{
\xymatrix@C=6em@R=2em{
& P \ar@{ ->}@/_2em/[ldd]_-{p_1} \ar@{ ->}@/^2em/[rdd]^-{p_2} \ar@[gray]@/_1em/@{..>}[ddd]|-{\gray{\stP}} \ar@[gray]@{}[ldd]|(.2){\gray{\lisocell{\st{p_1}}}} \ar@{}[rdd]|(.2){\gray{\isocell{\st{p_2}}}}
 \ar[d]|-{\vphantom{I_I}s} \ar@{}[ldd]|(.4){\isocell{\sigma_1}} \ar@{}[rdd]|(.4){\lisocell{\sigma_2}}
\\
& Q \ar@{ ->}@/_.5em/[ld]_-{q_1} \ar@{ ->}@/^.5em/[rd]^-{q_2} \ar@[gray]@{..>}@/^1em/[dd]|-{\gray{\st{Q}}}
 \ar@[gray]@{}[ldd]|{\gray{\lisocell{\st{q_1}}}} \ar@[gray]@{}[rdd]|{\gray{\isocell{\st{q_2}}}}
 \ar@[gray]@{}[d]|(.6){\gray{\isocell{\st{s}}}}
& \\
H \ar@[gray]@{..>}[rd]_{\gray{\stH}}
&& K \ar@[gray]@{..>}[ld]^{\gray{\stK}} \\
& \gray{G} &
}}
\end{equation}
(spot~$\st{s}$ in the center) such that $\sigma_1\cast\st{p_1}=\st{q_1}\cast\st{s}$ as 2-cell~$\stP\To \stH q_1 s\colon P\to G$ and $\sigma_2\cast\st{p_2}=\st{q_2}\cast\st{s}$ as 2-cell~$\stP\To \stK q_2 s\colon P\to G$.

Finally, a 2-cell between two morphisms $s=(s,\sigma_1,\sigma_2)$ and~$s'=(s',\sigma'_1,\sigma'_2)$ from $P$ to~$Q$ in~$\SpanG(H,K)$ is just a 2-cell~$\alpha\colon s\isoTo s'$ in~$\GG$ with compatibility conditions, first with the structure 2-cells: $\st{s'}=\alpha\cast\st{s}\colon \stP\To \st{Q} s'$ (to be a 2-cell in~$\GG/G$) and then with the wing cells $\sigma'_i=\alpha\cast\sigma_i\colon p_i\To q_i s'$ for $i=1,2$.

For clarity, the reader may first examine~\eqref{eq:Span-object} and~\eqref{eq:Span-morphism} while temporarily disregarding the grayed-out structure data.
They could also mentally gray-out the $\sigma_i$.

All compositions in~$\SpanG(H,K)$ are the obvious ones, as in~$\GG/G$.
\end{Cons}

\begin{Def}
\label{Def:Spanf}%
Again, we have a variant $\SpanfG(H,K)=\Spanname_{\GGf/G}(H,K)$ where we only allow faithful morphisms everywhere (once the structure morphisms are faithful then all wing morphisms and all morphisms of spans must be faithful).
\end{Def}

\begin{Rem}
\label{Rem:j!j*-Span}%
One could define $\Spanname_{\mathbb{D}}(H,K)$ for any choice of objects~$H,K$ in a (2,1)-category~$\mathbb{D}$, the above being the cases $\mathbb{D}=\GG/G$ or~$\GGf/G$. This construction is natural in~$\mathbb{D}$ in the essentially straightforward way: a 2-functor $F\colon \mathbb{D}\to \mathbb{E}$ yields $\Spanname_{\mathbb{D}}(H,K)\to \Spanname_{\mathbb{E}}(F(H),F(K))$ by applying~$F$ to everything in sight.
For us, given $j\colon G'\to G$ in~$\GG$, the two 2-functors $j_!\colon \GG/G'\to \GG/G$ and~$j^*\colon \GG/G\to \GG/G'$ of~\Cref{Cons:j!j*-GG/} induce 2-functors still denoted
\[
j_! \colon \SpanG[G'](H',K')\to \SpanG(j_!H',j_!K')
\]
for every $H',K'\in \GG/G'$, and similarly for every $H,K\in \GG/G$
\[
j^* \colon \SpanG(H,K)\to \SpanG[G'](j^*H,j^*K).
\]
\end{Rem}

%
\section{Mackey business}
\label{sec:Mackey}%
%

Let us recall Mackey 2-functors from~\cite{BalmerDellAmbrogio20}.
See~\Cref{Conv:GG} for~$\GG$.

\subsection{Mackey squares}\
\label{ssec:Mackey-squares}%

First, we close the notion of isocomma square~\eqref{eq:isocomma} under equivalences.
\begin{Def}[{\cite[Definition~2.2.1]{BalmerDellAmbrogio20}}]
\label{Def:Mackey-square}%
Consider a 2-cell $\alpha\colon i\circ v\To u\circ j$ in~$\GG$
\begin{equation}
\label{eq:Mackey-square}
\vcenter{
\xymatrix@C=1em@R=.5em{
& L \ar@{ ->}[ld]_-{v} \ar@{ ->}[rd]^-{j} \ar@{}[dd]|{\isocell{\alpha}}
\\
H \ar@{ ->}[rd]_-{i}
&& K \ar@{ ->}[ld]^-{u}
\\
& G
}}
\end{equation}
By slight abuse of language, we refer to it as `the square~$(\alpha)$' especially in larger diagrams.
This square induces a morphism denoted $w=\langle v,j,\alpha\rangle \colon L\to (i/u)$ as in~\eqref{eq:isocomma-translate}.
If $w$ is an equivalence, the square~$(\alpha)$ is called a \emph{Mackey square}.
\end{Def}
\begin{Exa}
\label{Exa:equi-Mackey}%
If $i$ and~$j$ are equivalences then the square~\eqref{eq:Mackey-square} is Mackey. See~\cite[Example~2.2.5]{BalmerDellAmbrogio20}. In the same vein, if $i\colon H\to G$ is fully faithful then the square
\[
\vcenter{
\xymatrix@C=1em@R=.5em{
& H \ar@{=}[ld]_-{} \ar@{=}[rd]^-{} \ar@{}[dd]|{\isocell{\id}}
\\
H \ar@{ >->}[rd]_-{i}
&& H \ar@{ >->}[ld]^-{i}
\\
& G
}}
\]
is Mackey. Indeed $\Delta_i=\langle \id,\id,\id_i \rangle \colon H\to (i/i)$ is fully faithful because $i$ is faithful (\cite[Proposition~3.1.3]{BalmerDellAmbrogio20}) and $\Delta_i$ is essentially surjective because~$i$ is full.
\end{Exa}

\begin{Exa}
\label{Exa:Mackey-square}%
Let $G$ be a finite group and~$i\colon H\into G$ and~$u\colon K\into G$ be inclusions of subgroups (\Cref{Exa:gps}).
We write a double coset in~$\KGH$ as~$KgH$ to indicate that we \emph{choose} a representative~$g$ in it. Then the following square is Mackey
\[
\xymatrix@C=.5em@R=1em{
& {}\coprod\limits_{(KgH)\in\KGH} H\cap K^{g} \ar[ld]_-{v} \ar[rd]^-{j} \ar@{}[dd]|-{\isocell{\alpha}}
\\
H \ar[rd]_-{i}
&& K \ar[ld]^-{u}
\\
& G
}
\]
where, on each component $H\cap K^g$, the morphism~$v$ is given by the inclusion, the morphism~$j$ is given by the conjugation-inclusion~$c_g={}^g(-)\colon H\cap K^g\into K$ and the 2-cell~$\alpha$ is given by the conjugation~$\gamma_g={}^g(-)\colon i\,v\isoTo u\,j$. (See \Cref{Exa:gps}.)
Indeed, the comparison morphism $w=\langle v,j,\alpha\rangle\colon \coprod_{(KgH)} H\cap K^{g}\to (i/u)$ maps the single object~$\sbull$ of~$H\cap K^g$ to $(\sbull,\sbull,g)$ in~$(i/u)$. It is an equivalence.
\end{Exa}

\begin{Rem}
\label{Rem:Mackey-square}%
Elementary properties of isocommas and Mackey squares can be found in~\cite[Section~2.1]{BalmerDellAmbrogio20} and~\cite[Section~3]{BalmerDellAmbrogio21}.
In particular, we shall use additivity of isocommas in each variable (\cite[Lemma~3.8]{BalmerDellAmbrogio21}) and the following:
\end{Rem}
\begin{Lem}[{\cite[Lemma~3.9]{BalmerDellAmbrogio21}}]
\label{Lem:Mackey-square-arith}
In a configuration of $2$-cells in~$\GG$ as follows
\[
\xymatrix@C=2em@R=1.5em{
& \ar[ld] \ar[rd]
\\
\ar[rd] \ar@{}[rr]|-{\isocell{\alpha}}&&\ar[ld]\ar[rd]
\\
& \ar[rd] \ar@{}[rr]|-{\isocell{\beta}} &&\ar[ld]
\\
&&&
}
\]
suppose that the bottom square~{$(\beta)$} is Mackey. Then the top square~{$(\alpha)$} is Mackey if and only if the obvious composite square~$(\beta\cast\alpha)$ is Mackey.
\qed
\end{Lem}

Here is an immediate application.
\begin{Lem}
\label{Lem:pb-Mackey}%
Let $j\colon G'\to G$ in~$\GG$ and $f\colon H\to K$ in~$\GG/G$ (\Cref{Def:GG/G}). Then the following commutative square (\Cref{Cons:j!j*-GG/}) is Mackey:
\[
\xymatrix@C=1em@R=1em{& j^*H \ar[rd]^-{\pr_2} \ar[ld]_-{j^*f} \ar@{}[dd]|{=}
\\
j^*K \ar[rd]_-{\pr_2}
&& H \ar[ld]^-{f}
\\
& K.
}
\]
In particular, if $f$ is a local equivalence then so is~$j^*f$.
\end{Lem}
\begin{proof}
By definition of~$j^*(-)=G'\times_G-=(j/-)$, we have the following diagram
\[
\xymatrix@C=1em@R=1em{
&& j^*H \ar[rd]^-{\pr_2} \ar[ld]_-{j^*f} \ar@{}[dd]|{=} \ar@/_3em/[lldd]^(.6){\ \ =}_(.7){\pr_1}
\\
& j^*K \ar[rd]^-{\pr_2} \ar[ld]_(.4){\pr_1} \ar@{}[dd]|-{\isocell{\gamma_{j/\stK}}}
&& H \ar[ld]_-{f} \ar@/^2.5em/[lldd]^(.4){\stH}_-{\lisocell{\st{f}}\!}
\\
G' \ar[rd]_-{j} && K \ar[ld]^(.4){\stK}
\\
& G
}
\]
whose lower-left square is an isocomma by definition of~$j^*K$. The composite square is isomorphic to the (outside) isocomma~$(j/\stH)$; indeed $f$ is a morphism in~$\GG/G$ (with $\st{f}\colon \stH\isoTo \stK\circ f$) and $j^*f:=\langle \pr_1,f\pr_2,\st{f}\cast\gamma_{j/\stH}\rangle$ satisfies the two equalities under~$\pr_i$ and whiskers~$\gamma_{j/\stK}$ into~$\st{f}\cast\gamma_{j/\stH}$. We conclude by~\Cref{Lem:Mackey-square-arith}.

The final statement about local equivalences follows by~\Cref{Lem:leq-pb}.
\end{proof}

\begin{Def}
\label{Def:cube}%
By a \emph{commutative cube of 2-cells in~$\GG$} we mean a diagram
\begin{equation}
\label{eq:cube}%
\vcenter{
\xymatrix@L=2pt@C=20pt@R=20pt{
	& P'
	\ar@{}[dd]|(.4){\isocell{\gamma'}}
	\ar@{->}[drrr]^-{r}
	\ar[dl]_-{p'}
    \ar[dr]^(.65){q'}|(.4){\hole}
    \ar@{}[rrrrdd]|(.4){\isocell{\kappa}}
    \ar@{}[rrdd]|(.2){\isocell{\delta}}
    \\
	H'
	\ar[drrr]^(.2){s}
	\ar[dr]_{u'}
    \ar@{}[rrrrdd]|-{\isocell{\alpha}}
    && K'
	\ar@{}[dr]|{\isocell{\beta}}
	\ar[dl]|(.53){\hole}_(.3){v'\!}
	\ar[drrr]|-{\hole}^(.3){t} &&
	P
	\ar@{}[dd]|(.6){\isocell{\gamma}}
	\ar[dr]^{q}
	\ar[dl]_(.3){p\!\!} & \\
	& G'
	\ar[drrr]_-{j} &&
	H
	\ar[dr]_<<<{u}
    && K
	\ar[dl]^{v}  \\
	&&&& G &
}}
\end{equation}
whose six faces
$
\alpha\colon j u'\To u s,\
\beta\colon j v'\To v t,\
\delta\colon s p'\To p r,\
\kappa\colon t q'\To q r,\
\gamma\colon u p\To v q$ and $
\gamma'\colon u' p'\To v' q'
$
satisfy $\kappa\cast \beta \cast \gamma'=\gamma\cast \delta\cast \alpha\colon j u' p' \To vqr$ (in~\Cref{Not:cast}).
\end{Def}

We can construct such a cube by pulling back any square~$(\gamma)$ via~\Cref{Cons:j!j*-GG/}.
To be precise, we view $P$ as an object over~$G$ via the leftmost path~$u\circ p$, as we did in~\Cref{Conv:comma-isocomma} for the isocomma. We immediately get:

\begin{Lem}
\label{Lem:cube-exists}%
Given a square~$(\gamma)$ in~$\GG$ as in the `front' of~\eqref{eq:cube} and given a morphism~$j\colon G'\to G$, define a cube by applying~$j^*=G'\times_G-$ to the square~$(\gamma)$ with $r,s,t$ the second projections~$\pr_2\colon G'\times_G(-)\to (-)$, so that $H'=G'\times_G H=(j/u)$ and~$K'=G'\times_G K=(j/v)$ and $P'=G'\times_G P=(j/up)$ and similarly $p',q',u',v'$ are the images of~$p,q,u,v$ under~$j^*$.
The faces~$(\alpha)$ and~$(\beta)$ are isocommas: $\alpha=\gamma_{j/u}$ and~$\beta=\gamma_{j/v}$ and the 2-cells $\delta=\id\colon s p'=p r$ and~$\kappa=\id\colon t q'=q r$ are identities. The back-cell~$\gamma'$ is~$j^*\gamma$. This yields a commutative cube~\eqref{eq:cube} such that all four `side' squares $(\alpha)$, $(\beta)$, $(\delta)$ and~$(\kappa)$ are Mackey squares.
\end{Lem}
\begin{proof}
The squares $(\alpha)$ and~$(\beta)$ are isocomma squares, hence Mackey, while $(\delta)$ and~$(\kappa)$ are Mackey by~\Cref{Lem:pb-Mackey}.
\end{proof}

\begin{Lem}
\label{Lem:cube-pb}%
Consider a commutative cube of 2-cells as in~\eqref{eq:cube}, with an opposite pair of `side' squares being Mackey: say $(\alpha)$ and~$(\kappa)$.
If the front square~$(\gamma)$ is Mackey then so is the back square~$(\gamma')$.
\end{Lem}
\begin{proof}
Apply~\Cref{Lem:Mackey-square-arith}: $(\kappa\inv\cast\gamma)$ is Mackey, hence so is the isomorphic~$(\gamma'\cast\alpha\inv)$, hence so is~$(\gamma')$.
\end{proof}

\begin{Lem}
\label{Lem:cube}%
Consider a commutative cube of 2-cells as in~\eqref{eq:cube} and suppose that the front and back squares~$(\gamma)$ and~$(\gamma')$ are Mackey squares.
\begin{enumerate}[\rm(a)]
\item
\label{it:cube-leq}%
If $s$, $t$ and~$j$ are local equivalences then so is~$r$.
\smallbreak
\item
\label{it:cube-2-of-3}%
If $(\beta)$ is a Mackey square then so is $(\delta)$.
\smallbreak
\item
\label{it:cube-tricky}%
If $t$ and~$j$ are equivalences then the square $(\delta)$ is Mackey.
\end{enumerate}
\end{Lem}

\begin{proof}
Part~\eqref{it:cube-leq} is an explicit exercise, using~\Cref{Rem:local-equiv-iso} in~$\gpd$.
Part~\eqref{it:cube-2-of-3} follows from~\Cref{Lem:Mackey-square-arith}: $(\beta\cast\gamma')$ is Mackey, hence the isomorphic~$(\gamma\cast\delta)$ is Mackey, hence so is~$(\delta)$. Part~\eqref{it:cube-tricky} follows from~\eqref{it:cube-2-of-3} and~\Cref{Exa:equi-Mackey}.
\end{proof}

Let us combine Mackey squares with the $\approx$-loci of~\Cref{Def:approx}.

\begin{Lem}
\label{Lem:approx-Mackey}%
Consider a Mackey square~\eqref{eq:Mackey-square} with $u$ and~$v$ local equivalences
\[
\vcenter{
\xymatrix@C=1em@R=.5em{
& L \ar@{ ->}[ld]_v \ar@{ ->}[rd]^j \ar@{}[dd]|{\isocell{\alpha}} & \\
H \ar@{ ->}[rd]_i &  & K \ar@{ ->}[ld]^u \\
& G.\!\! &}}
\]
Then $v$ maps the $\approx$-locus $L^{j\approx}$ of~$j$ into the $\approx$-locus $H^{i\approx}$ of~$i$, it maps the complement $L^{j\not\approx}$ into the complement~$H^{i\not\approx}$, and the resulting two squares are Mackey
\begin{equation}
\label{eq:approx-Mackey}%
\vcenter{
\xymatrix@C=2em@R=1em{
& L^{j\approx} \ar@{ ->}[ld]_{v\restr{\approx}} \ar@{ ->}[rd]^{j^\approx} \ar@{}[dd]|{\isocell{{\alpha\restr{\approx}}}} & \\
H^{i\approx} \ar@{ ->}[rd]_{i^\approx} &  & K_{} \ar@{ ->}[ld]^{u_{}} \\
& G_{} &}}
\qquadtext{and}
\vcenter{
\xymatrix@C=2em@R=1em{
& L^{j\not\approx} \ar@{ ->}[ld]_{v\restr{\not\approx}} \ar@{ ->}[rd]^{j^{\not\approx}} \ar@{}[dd]|{\isocell{{\alpha\restr{\not\approx}}}} & \\
H^{i\not\approx} \ar@{ ->}[rd]_{i^{\not\approx}} &  & K_{} \ar@{ ->}[ld]^{u_{}} \\
& G_{} &}}
\end{equation}
where $v\restr{\approx}:=v\restr{L^{j\approx}}$ and $v\restr{\not\approx}:=v\restr{L^{j\not\approx}}$, and where ${\alpha\restr{\approx}}:=\alpha\restr{L^{j\approx}}$ and ${\alpha\restr{\not\approx}}:=\alpha\restr{L^{j\not\approx}}$.
\end{Lem}
\begin{proof}
For every connected component $M$ of~$L$, let $I\into H$ be the connected component of~$v(M)$ and $J\into K$ the connected component of~$j(M)$. We get a square
\[
\vcenter{
\xymatrix@C=1em@R=.5em{
& M \ar@{ ->}[ld]_-{v\restr{M}} \ar@{ ->}[rd]^-{j\restr{M}} \ar@{}[dd]|{\isocell{\alpha\restr{M}}}
\\
I \ar@{ ->}[rd]_-{i\restr{I}}
&& J \ar@{ ->}[ld]^-{u\restr{J}}
\\
& G
}}
\]
Since $u$ and~$v$ are local equivalences, $u\restr{J}\colon J\to G$ is fully faithful and $v\restr{M}\colon M\to I$ is an equivalence. Hence $i\restr{I}$ is fully faithful if and only if~$j\restr{M}$ is fully faithful. This gives the decomposition $v$ as a coproduct (without crossed components $L^{j\not\approx}\to H^{i\approx}$ or $L^{j\approx}\to H^{i\not\approx}$) and the two resulting squares are Mackey by additivity of isocommas.
\end{proof}

\subsection{Mackey 2-functors}\
\label{ssec:Mackey 2-functors}%

We recall the fundamental notion of~\cite[Definition~2.3.5]{BalmerDellAmbrogio20}.
\begin{Def}
\label{Def:2-Mackey}%
A \emph{Mackey 2-functor} $\cM$ on our 2-category~$\GG$ (\Cref{Conv:GG}) is a contravariant 2-functor $\cM\colon \GG^\op\to \ADD$ that satisfies the following properties:
\begin{enumerate}[\rm({Mack}\,1)]
\item
\label{it:Mack-1}%
The 2-functor~$\cM$ is additive (\Cref{Def:restriction-2-functor}).
\smallbreak
\item
\label{it:Mack-2}%
For every faithful $i\colon H\into G$ in~$\GG$, restriction $i^*\colon \cM(G)\to \cM(H)$ admits a (special Frobenius) two-sided adjoint~$i_*\adj i^*\adj i_*$ called `induction'.
\smallbreak
\item
\label{it:Mack-3}%
For every Mackey square~\eqref{eq:Mackey-square} with $i$ and~$j$ faithful, the left mate $\alpha_!$ of~$\cM(\alpha)\colon v^*i^*\isoTo j^*u^*$ and the right mate $(\alpha\inv)_*$ of~$\cM(\alpha\inv)$
\[
\alpha_!\colon j_*\,v^* \To u^*\,i_*
\qquadtext{and}
(\alpha\inv)_*\colon u^*\,i_* \To j_*\,v^*
\]
are isomorphisms. These are called the \emph{Mackey base-change formulas}.
\end{enumerate}
\end{Def}
\begin{Rec}
\label{Rec:special-Frobenius}%
For $i_*\adj i^*$ and~$i^*\adj i_*$ we need units and counits
\begin{equation}
\label{eq:4-units}%
\vcenter{
\xymatrix@R=.5em@C=5em{
	\leta\colon \Id \To i^*i_*
	& \reta\colon \Id \To i_*i^*
	\\
	\leps\colon i_* i^*\To \Id
	& \reps\colon i^*i_*\To \Id.\!
}}
\end{equation}
satisfying (separate) unit-counit relations. We write $\leta^{(i)}$, etc, if we need to emphasize~$i$.
The two-sided adjunction is called `special Frobenius' when the composite of the left unit $\leta\colon\Id\To i^*i_*$ with the right counit~$\reps\colon i^*i_*\To \Id$ is the identity transformation.
See~\cite[Appendix~A.2]{BalmerDellAmbrogio20} for the mates~$\alpha_!=\leps^{(j)}\cast\alpha^*\cast\leta^{(i)}$ and~$(\alpha\inv)_*=\reps^{(i)}\cast(\alpha\inv)^*\cast\reta^{(j)}$ in~\Mack{3}, writing~$\cast$ as in~\Cref{Not:cast}.
\end{Rec}
\begin{Rem}
\label{Rem:rectify-Mackey}%
The above~\Cref{Def:2-Mackey} is a condensed version of~\cite[Definition~2.3.5]{BalmerDellAmbrogio20}, with Ambidexterity (Mack\,4) absorbed in~\Mack{2}.
By the Rectification Theorem~\cite[Theorem~3.4.3]{BalmerDellAmbrogio20} one can choose the units and counits in~\eqref{eq:4-units} so that several additional properties become true, for instance $i_*\adj i^*\adj i_*$ being special Frobenius (Mack\,9) as indicated in~\Mack{2} above.
We can also assume (Mack\,7) which says that $\alpha_!$ is the inverse of~$(\alpha\inv)_*$ in~\Mack{3}.
Furthermore, several `harmless' simplifications can be arranged, for instance $(\id)_*=\id$ and more generally when $i^*$ is an equivalence, $i_*$ is chosen to be its inverse. We tacitly assume that (co)\-units~\eqref{eq:4-units} have been chosen with those extended properties.
\end{Rem}

\begin{Rec}[{\cite[Definition~4.2.2]{BalmerDellAmbrogio20}}]
\label{Rec:Mack(G)}%
A \emph{morphism $t\colon \cM\to \cM'$ of Mackey 2-functors} on~$\GG$ is a pseudo-natural transformation compatible with restrictions (\ie a morphism in~$\resfun{\GG}$) that preserves induction in the following sense. For every $i\colon G'\into G$ the right mate $(t_i)_*$ of the compatibility isomorphism $t_i\colon i^*\circ t_{G}\isoTo t_{G'}\circ i^*$
\[
(t_i)_*\colon t_{G}\circ i_*\To i_* \circ t_{G'}
\]
is an isomorphism of functors~$\cM(G')\to \cM'(G)$.
(Equivalently, the left mate $(t_i\inv)_!$ is an isomorphism. In that case, by~\cite[Proposition~6.3.1\,(ii)]{BalmerDellAmbrogio20}, this $(t_i\inv)_!$ is the inverse of~$(t_i)_*$.)
We define the 2-category of Mackey 2-functors
\[
\mackfun{\GG}
\]
as the 2-full subcategory of the 2-category of restriction 2-functors~$\resfun{\GG}$ with those morphisms.
In~\cite{BalmerDellAmbrogio20} it is denoted $\mathsf{Mack}(\GG)$. The above notation is coherent with the other 2-categories that we considered:
\[
\mackfun{\GG}\hook \resfun{\GG}\hook \conjfun{\GG}.
\]
These forgetful 2-functors are faithful (meaning injective on 2-cells).
\end{Rec}

%
\section{Folding induction and traces}
\label{sec:trace}%
%
\subsection{Induction along local equivalences}\
\label{ssec:leq-induction}%

Our goal in this paper is to approximate restriction 2-functors~$\cA$ and conjugation 2-functors~$\cB$ by Mackey 2-functors that will therefore admit induction as in~\Mack{2}.
It is important to notice that even a mere restriction or even a conjugation 2-functor already has \emph{some} induction, thanks to additivity. The local equivalences of~\Cref{Def:loc-equiv} are the right notion for this. We sometimes refer to this type of easy induction as `folding induction' or `folding pushforward' to distinguish it from the hard-won induction of Mackey 2-functors.

\begin{Prop}[`Folding' induction]
\label{Prop:leq-adjoints}%
Let $\cB\colon \GGi^\op\to \ADD$ be a conjugation 2-functor on~$\GG$ (\eg a restriction 2-functor~$\cA$ restricted to local equivalences).
Let $s\colon H\to G$ be a local equivalence in~$\GG$. Then $s^*\colon \cB(G)\to \cB(H)$ admits a two-sided adjoint such that $s_*\adj s^*\adj s_*$ is special Frobenius.
\end{Prop}
\begin{proof}
In view of~\Cref{Lem:leq-folding}, it really suffices to understand the adjoints to the diagonal functor
$\Delta\colon \cat{C}\to\cat{C}^{n}$
for any additive category~$\cat{C}$ and for~$n\ge 0$. Its two-sided adjoint is the biproduct $\bigoplus\colon \cat{C}^n\to \cat{C}$, $(c_1,\ldots,c_n)\mapsto c_1\oplus\cdots\oplus c_n$. Choosing the obvious units and counits, one verifies the special Frobenius property.
\end{proof}

\begin{Prop}[Base-change for folding induction]
\label{Prop:leq-base-change}%
Consider a Mackey square in~$\GG$ (\Cref{Def:Mackey-square}) with $s$ and~$t$ local equivalences
\[
\vcenter{
\xymatrix@C=1em@R=.5em{
& L \ar@{ ->}[ld]_v \ar@{ ->}[rd]^t \ar@{}[dd]|{\isocell{\alpha}} & \\
H \ar@{ ->}[rd]_s &  & K \ar@{ ->}[ld]^u \\
& G.\!\! &}}
\]
Let $s_*$ and~$t_*$ be two-sided adjoints of~$s^*$ and~$t^*$ respectively, as in~\Cref{Prop:leq-adjoints}.
\begin{enumerate}[\rm(a)]
\item
\label{it:leq-bc-1}%
Let $\cA\colon \GG^\op\to \ADD$ be a restriction 2-functor on~$\GG$.
Then the left and right mates of~$\cA(\alpha^{\pm1})$ yield inverse isomorphisms of functors $\cA(H)\to \cA(K)$:
\[
\alpha_!\colon t_*\,v^*\isoTo u^*\,s_*
\qquadtext{and}
(\alpha\inv)_*\colon u^*\,s_*\isoTo t_*\,v^*.
\]
\item
\label{it:leq-bc-2}%
Let $\cB\colon \GGi^\op\to \ADD$ be a conjugation 2-functor on~$\GG$. Then the above formulas also hold in~$\cB$ when they make sense, \ie if $u$ and~$v$ are local equivalences.
\end{enumerate}
\end{Prop}
\begin{proof}
Using additivity of isocommas (\cite[Lemma~3.8]{BalmerDellAmbrogio21}) and~\Cref{Lem:leq-folding} again, one reduces to the case of an isocomma in which $s=\nabla^{(n)}_H$ is a folding:
\[
\vcenter{
\xymatrix@C=1em@R=1em{
& (s/u) \ar@{ ->}[rd]^{\pr_2} \ar@{ ->}[ld]_{\pr_1} \ar@{}[dd]|{\isocell{\gamma}} & \\
H^{\sqcup n} \ar@{ ->}[rd]_{s=\nabla^{(n)}_{H}} &  & K \ar@{ ->}[ld]^u \\
& H.\! &
}}
\]
In that case, one checks that $(s/u)\simeq K^{\sqcup n}$ in such a way that $\pr_1=u^{\sqcup n}$, that $\pr_2=\nabla^{(n)}_{K}$ is also a folding, for the same number~$n$, and that $\gamma$ is identity.
The formulas are then easily verified. They amount to $u^*$ commuting with biproduct~$\bigoplus$, which is true in~\eqref{it:leq-bc-1} as well as in~\eqref{it:leq-bc-2} as soon as $u^*$ exists in~$\ADD$.
\end{proof}

\begin{Rem}
\label{Rem:fake-Mackey}%
An alternate formulation of~\Cref{Prop:leq-adjoints} is to consider the 2-category~$\GGi$ of~\Cref{Def:GGi} as our input~$\GG$ in~\Cref{Conv:GG}. This is a (2,1)-category with isocommas (\Cref{Lem:leq-pb}) on which every additive 2-functor is Mackey (on~$\GGi$ itself, not on~$\GG$).
In other words
\begin{equation}
\label{eq:fake-Mackey}%
\mackfun{\GGi}=\resfun{\GGi}=\conjfun{\GGi}=\conjfun{\GG}.
\end{equation}
Moreover, every transformation between additive 2-functors preserves folding induction, since its components are additive and therefore preserve finite biproducts.
See \Cref{Lem:leq-transfo-induction}.
\textsl{\`A vaincre sans p\'eril, on triomphe sans gloire}...
\end{Rem}

\begin{Rem}
\label{Rem:rectified-leq}%
To be complete, when using the above folding induction, one could also arrange units and counits to satisfy the `harmless' axioms (Mack\,5)--(Mack\,10) of \cite[Rectification Theorem~3.4.3]{BalmerDellAmbrogio20}, as explained in~\Cref{Rem:rectify-Mackey} for Mackey 2-functors. This is an easier version of the Rectification Theorem in \loccit, that we leave to the reader, for instance using~\Cref{Rem:fake-Mackey}.
We will tacitly assume from now on that (co)units are chosen in this `rectified' way.
\end{Rem}

We can combine the folding induction with the $\approx$-locus of~\Cref{Def:approx}.
\begin{Cor}
\label{Cor:approx-basechange}%
Consider a Mackey square~$(\alpha)$ with $u$ and~$v$ local equivalences
\[
\vcenter{
\xymatrix@C=1em@R=.5em{
& L \ar@{ ->}[ld]_-{v} \ar@{ ->}[rd]^-{j} \ar@{}[dd]|{\isocell{\alpha}} & \\
H \ar@{ ->}[rd]_-{i} &  & K \ar@{ ->}[ld]^-{u} \\
& G}}
\qquad
\overset{\textrm{\Cref{Lem:approx-Mackey}}}{\Longrightarrow}
\qquad
\vcenter{
\xymatrix@C=1em@R=.5em{
& L^{j\approx} \ar@{ ->}[ld]_-{v\restr{\approx}} \ar@{ ->}[rd]^-{j^\approx} \ar@{}[dd]|{\isocell{{\alpha\restr{\approx}}}} & \\
H^{i\approx} \ar@{ ->}[rd]_-{i^\approx} &  & K \ar@{ ->}[ld]^u \\
& G}}
\]
and its restriction to the $\approx$-loci of~$i$ and~$j$, which remains Mackey by~\Cref{Lem:approx-Mackey}.
For every conjugation 2-functor~$\cB$, we have a base-change isomorphism~$({\alpha\restr{\approx}})_!$ between the functors $(j^\approx)_*\,(v\restr{\approx})^* \isoTo u^*\,(i^\approx)_*$ from $\cB(H^{i\approx})$ to~$\cB(K)$, whose inverse is given by~$(\alpha\inv\restr{\approx})_*$.
\end{Cor}
\begin{proof}
This holds by base-change~\Cref{Prop:leq-base-change}\,\eqref{it:leq-bc-2} on the right-hand square.
\end{proof}

\subsection{Traces}\
\label{ssec:traces}%

Let us review the notion of trace that we shall use later on.

\begin{Def}
\label{Def:trace}%
Let $s_*\adj s^*\adj s_*$ be a special Frobenius two-sided adjunction between additive categories, say $s^*\colon \cat{C}\to \cat{D}$ and $s_*\colon \cat{D}\to \cat{C}$ (\Cref{Rec:special-Frobenius}). Given two objects $c_1,c_2$ in~$\cat{C}$ and a morphism $f\colon s^*(c_1)\to s^*(c_2)$ between their images in~$\cat{D}$, the \emph{trace of~$f$ along~$s_*$} is the morphism $\tr_{s_*}(f)\colon c_1\to c_2$ defined by applying~$s_*$ and composing with the relevant unit and counit on the sides:
\[
\xymatrix@C=2em{
\tr_{s_*}(f)\colon
& c_1 \ar[rr]^-{\reta}_-{(s^*\adj s_*)}
&& s_*s^*(c_1) \ar[rr]^-{s_*(f)}
&& s_*s^*(c_2) \ar[rr]^-{\leps}_-{(s_*\adj s^*)}
&& c_2\,.
}
\]
\end{Def}
\begin{Exa}
\label{Exa:trace-sum}%
Let $\cat{C}$ be an additive category and $\cat{D}=\cat{C}^n$.
Let $s^*=\Delta\colon \cat{C}\to \cat{D}$ be the diagonal functor and $s_*=\bigoplus\colon \cat{D}\to \cat{C}$ be its (special Frobenius) two-sided adjoint given by the biproduct and obvious (co)units. Then for every $c_1,c_2\in\cat{C}$ and every $f=(f_1,\ldots,f_n)\colon s^*c_1=(c_1,\ldots,c_1)\to s^*c_2=(c_2,\ldots,c_2)$ in~$\cat{C}^n$, the morphism $s_*(f)$ is $\smat{f_1&&0\\[-.5em]&\ddots&\\0&&f_n}\colon c_1^{\oplus n}\to c_2^{\oplus n}$ and $\tr_{\oplus}(f)=f_1+\cdots +f_n$.
\end{Exa}
\begin{Conv}
If $s^*$ is the image of some~$s$ in~$\GG$ via a 2-functor~$\GG^\op\to \ADD$ we simply write $\tr_s$ instead of $\tr_{s_*}$ and call it the \emph{trace along~$s$}.
\end{Conv}

\Cref{Exa:trace-sum} shows that the trace is not functorial, as in linear algebra $\tr(AB)\neq\tr(A)\tr(B)$. However for a constant~$\lambda$, we have~$\tr(\lambda A)=\lambda\tr(A)$. This becomes:
\begin{Lem}[Partial functoriality]
\label{Lem:trace-special-comp}%
In the situation of~\Cref{Def:trace}, we have
\[
\tr_s\big((s^*g)\circ f\circ(s^*e)\big)=g\circ\tr_s(f)\circ e
\]
for every $e\colon c_0\to c_1$ and~$g\colon c_2\to c_3$ in~$\cat{C}$ and every $f\colon s^*(c_1)\to s^*(c_2)$ in~$\cat{D}$.
\end{Lem}
\begin{proof}
This follows from naturality of~$\reta$ and~$\leps$.
\end{proof}

\begin{Lem}
\label{Lem:trace-base-change}%
Consider a square in~$\ADD$ (for instance the image of a square~\eqref{eq:Mackey-square} of groupoids by a 2-functor~$\GG^\op\to\ADD$, hence the $(-)^*$ decorations)
\[\xymatrix@C=1em@R=.5em{
& \tilde{\cat{D}} \ar@{<-}[ld]_{v^*} \ar@{<-}[rd]^{t^*} \ar@{}[dd]|{\isocell{\alpha^*}} & \\
\cat{D} \ar@{<-}[rd]_{s^*} &  & \tilde{\cat{C}} \ar@{<-}[ld]^{u^*} \\
& \cat{C} &}
\]
Suppose that $s^*$ and~$t^*$ admit special Frobenius two-sided adjoints~$s_*$ and~$t_*$ respectively, and that the left and right mates of~$\alpha^*$ yield inverse isomorphisms $\alpha_!\colon t_*\,v^*\isoTo u^*\,s_*$ and~$(\alpha\inv)_*\colon u^*\,s_*\isoTo t_*\,v^*$. Then we have%
\begin{equation}
\label{eq:trace-base-change}%
u^*\circ \tr_{s}=\tr_{t}\circ_{\alpha} v^*.
\end{equation}
More precisely, given two objects~$c_1,c_2\in\cat{C}$ and a morphism $f\colon s^*(c_1)\to s^*(c_2)$ in~$\cat{D}$, consider the objects $\tilde{c}_1=u^*(c_1)$ and $\tilde{c}_2=u^*(c_2)$ in~$\tilde{\cat{C}}$ and consider the morphism $\tilde{f}\colon t^*(\tilde{c}_1)\to t^*(\tilde{c}_2)$ obtained by `adjusting' $v^*(f)$ with~$\alpha^*$ as follows:
\begin{equation}
\label{eq:tilde-f}%
\tilde{f}:=\alpha^*\,v^*(f)\,(\alpha^*)\inv\colon t^*(u^*(c_1))\to t^*(u^*(c_2)).
\end{equation}
Then~\eqref{eq:trace-base-change} means $u^*(\tr_{s}(f))=\tr_{t}(\tilde{f})$.
\end{Lem}
\begin{proof}
Applying $u^*$ to the morphism $\tr_{s}(f)$ of~\Cref{Def:trace} gives the top row below
\begin{equation}
\label{eq:aux-bc-trace}%
\vcenter{\xymatrix@C=2em{
u^*(c_1) \ar[rr]^-{u^*(\reta^{(s)})}_-{} \ar@{=}[dd]
&& u^*s_*s^*(c_1) \ar[rr]^-{u^*s_*(f)}
&& u^*s_*s^*(c_2) \ar[rr]^-{u^*(\leps^{(s)})}
&& u^*(c_2) \ar@{=}[dd]
\\
&& t_*v^*s^*(c_1) \ar@{<-}[u]^-{(\alpha\inv)_*s^*}_-{\simeq} \ar[rr]^-{t_*v^*(f)}
&& t_*v^*s^*(c_2) \ar[u]_-{\alpha_!\,s^*}^-{\simeq}
&&
\\
u^*(c_1) \ar[rr]^-{\reta^{(t)} u^*}
&& t_*t^*u^*(c_1) \ar@{<-}[u]_-{t_*(\alpha^*)}^-{\simeq} \ar[rr]^-{t_*(\tilde{f})}
&& t_*t^*u^*(c_2) \ar@{<-}[u]_-{t_*(\alpha^*)}^-{\simeq} \ar[rr]^-{\leps^{(t)}\,u^*}
&& u^*(c_2)
}}
\end{equation}
In the middle column, the top square commutes by naturality of $(\alpha\inv)_*=(\alpha_!)\inv$,
whereas the bottom square is obtained by applying $t_*$ to the definition~\eqref{eq:tilde-f} of~$\tilde{f}$.
The two side regions commute by definition of mates, by naturality and by the unit-counit relations for~$s_*\adj s^*\adj s_*$; for instance for the right-hand region:
\[
\xymatrix@C=1.5em{
t_*v^*s^* \ar[rrrrrr]^-{{\displaystyle\alpha_!}s^*} \ar[rrd]^-{t_*v^*\,{\displaystyle\leta^{(s)}}\,s^*} \ar@{=}@/_1em/[dd]
&&&\ar@{}[d]|-{(\textrm{def.\ of }\alpha_!)}&&& u^*s_*s^* \ar[rrd]^-{u^*\,{\displaystyle\leps^{(s)}}}
\\
\ar@{}[rr]|(.3){\textrm{(unit-counit)}} && t_*v^*s^*s_*s^* \ar[rr]_-{t_*{\displaystyle\alpha^*}s_*s^*} \ar[lld]^(.3){t_*v^*s^*\,{\displaystyle\leps^{(s)}}}
&\ar@{}[d]|-{(\textrm{nat.})}& t_*t^*u^*s_*s^* \ar[rru]^-{{\displaystyle\leps^{(t)}} u^*s_*s^*} \ar[rrd]^-{t_*t^*u^*\,{\displaystyle\leps^{(s)}}}
&&{\scriptstyle(\textrm{nat.})}&& u^*
\\
t_*v^*s^* \ar[rrrrrr]_-{t_*{\displaystyle\alpha^*}}
&&&&&& t_*t^*u^* \ar[rru]_-{{\displaystyle\leps^{(t)}} u^*}
}
\]
The left-hand region is similar.
The bottom row in the commutative diagram~\eqref{eq:aux-bc-trace} is the trace of~$\tilde{f}$ with respect to~$t$ and this gives the result.
\end{proof}

%
\section{Construction of~$\Alp[(-)]$}
\label{sec:Alp}%
%

Recall that $\GG$ is our 2-category of groupoids of interest as in~\Cref{Conv:GG}.
For most of this section \emph{we fix a restriction 2-functor}, \ie an additive 2-functor
\[
\cA\colon \GG^\op\to \ADD.
\]
We construct for every $G$ in~$\GG$ an additive category $\Alp(G)$ by adjoining induction data to $\cA$ along all faithful maps
into~$G$. Roughly speaking, an object of $\Alp(G)$ is an object in~$\cA(H)$, living
over a groupoid~$H\into G$ faithfully embedded in~$G$. A morphism in~$\Alp(G)$ involves spans over~$G$, modulo the trace relations that come from the few inductions already present in~$\cA$, as explained in~\Cref{sec:trace}.

The reader should recall the comma 2-categories and the span 2-categories of~\Cref{sec:basics}, specifically, the 2-category $\GGf/G$ of objects \emph{faithfully} mapped into~$G$ (\Cref{Def:GGf}) and the 2-category $\SpanfG(H,K)$ of spans in~$\GGf/G$ (\Cref{Def:Spanf}).

\begin{Cons}
\label{Cons:Alp-category}%
Let $G\in\GG$. We construct a category $\Alp(G)$ as follows.

\smallskip
\noindent
\textbf{Objects.}
An object of $\Alp(G)$ is a pair~$(H,x)$ where $H\in \GGf/G$ is a groupoid faithfully mapped into~$G$ and $x\in\cA(H)$ is an object in the additive category provided by~$\cA$ on~$H$. Explicitly, this amounts to a pair $(\stH\colon H\into G,\ x)$, where $\stH$ is a faithful functor in~$\GG$ (called the `structure morphism' for~$H$ and usually omitted) and $x\in \cA(H)$ is an object defined over~$H$.

\smallskip
\noindent
\textbf{Morphism representatives.}
Given two objects~$(H,x)$ and~$(K,y)$ in~$\Alp(G)$, a \emph{morphism representative} $(H,x)\to (K,y)$
consists of a pair $(P,f)$ whose first entry $P=(P,p_1,p_2)\in \SpanfG(H,K)$ is a span from~$H$ to~$K$ and whose second entry $f\colon p_1^*(x)\to p_2^*(y)$
is a morphism in~$\cA(P)$ between the pull-backs of~$x$ and~$y$ along the two `wings' $p_1\colon P\to H$ and~$p_2\colon P\to K$ of the span, see~\eqref{eq:Span-object}.
In picture:
\begin{equation}
\label{eq:Alp-morphism}
\qquad\vcenter{
\xymatrix@C=2em@R=.5em{
& P \ar@[gray]@{ >.>}[dd]|-{\gray\stP} \ar@{ >->}@/_.5em/[ld]_-{p_1} \ar@{ >->}@/^.5em/[rd]^-{p_2}
 \ar@{}[ldd]|(.45){\gray{\lisocell{\st{p_1}}}}
 \ar@{}[rdd]|(.45){\gray{\isocell{\st{p_2}}}}
\\
H \ar@[gray]@{ >->}[rd]_{\gray\stH} & & K \ar@[gray]@{ >->}[ld]^{\gray\stK} \\
& \gray G &
}}
\qquad\text{together with}\quad
f\colon p_1^*(x)\to p_2^*(y) \quad\textrm{in }\cA(P).
\end{equation}
Following the artistic direction taken in~\Cref{ssec:commas}, we gray-out the structure data, to emphasize in black  the `essential' part of the morphism, namely~$(P,p_1,p_2;f)$.

\smallskip
\noindent
\textbf{Relations.}
The equivalence relation~$\appr$ on morphism representatives is the one generated by the \emph{one-step $\appr$-equivalence}
$(P,f)\appr(P',f')$ defined by the existence of a morphism $s\colon (P,p_1,p_2)\to (P',p'_1,p'_2)$ in~$\SpanfG(H,K)$ such that, on the middle objects, $s$ is a \emph{local equivalence} $P\xto{\approx}P'$ in~$\GG$ (\Cref{Def:loc-equiv}) and such that $f'$ is the \emph{trace} of~$f$ along~$s$ (\Cref{Def:trace}).
Let us unpack this relation. The structure data for~$s=(s,\sigma_1,\sigma_2)$ is displayed in full in~\eqref{eq:Span-morphism}.
The essential part is
\begin{equation}
\label{eq:Alp-relation}
\vcenter{\xymatrix@C=3em@R=.5em{
& P \ar@{ >->}@/_1em/[lddd]_-{p_1} \ar@{ >->}@/^1em/[rddd]^-{p_2} \ar@{ >->}[dd]^-{s}_-{\approx}
\ar@{}[lddd]|{\isocell{\sigma_1}} \ar@{}[rddd]|{\lisocell{\sigma_2}} &
\\
\\
& P' \ar@{ >->}[ld]^-{p_1'} \ar@{ >->}[rd]_-{p_2'}
\\
H
&& K.\!
}}
\end{equation}
For the trace of~$f$ along~$s$, one first `adjusts' $f$ using the wing cells~$\sigma_1$ and~$\sigma_2$ of~$s$
\begin{equation}
\label{eq:f-adjusted}%
\vcenter{\xymatrix{
f^\sigma:=
& s^*({p_1'}^*(x)) \ar[r]^-{(\sigma_1^*)^{-1}}_-{\simeq}
& p_1^*(x) \ar[r]^-{f}
& p_2^*(y) \ar[r]^-{\sigma_2^*}_-{\simeq}
& s^*({p_2'}^*(y))
}}
\end{equation}
to get a morphism in~$\cA(P)$ between two objects that come via~$s^*\colon \cA(P')\to \cA(P)$ and this adjusted morphism~$f^\sigma$ admits a trace along~$s_*$ in the sense of~\Cref{Def:trace}; that trace is a morphism ${p_1'}^*(x)\to {p_2'}^*(y)$ in~$\cA(P')$ and we require it to be equal to the given~$f'$.
The definition of the trace $\tr_s$ uses that $s^*\colon \cA(P')\to \cA(P)$ admits a special Frobenius two-sided adjoint~$s_*$, even though $\cA$ is only a 2-functor, and this is the `folding pushforward' of~\Cref{Prop:leq-adjoints} using that $s$ is a local equivalence.

Note in particular the equivalence relations provided by~$s\colon P\overset{\simeq}\to P'$ an actual equivalence in~$\GG/G$, which even for~$s=\id$ can involve isomorphisms~$\sigma_i\colon p_i\isoTo p_i'$ that change the wings of~$P$ up to isomorphism. This yields a special case of our $\appr$-equivalence on morphisms that we shall refer to as the \emph{strong $\appr$-equivalence}.

We write $[P,f]$ for the equivalence class of $(P,f)$ with respect to~$\appr$, and we write~$[P,p_1,p_2;f]$ when we want to emphasize the wings of the span~$P$.

\smallskip
\noindent
\textbf{Composition.}
Let $[P,f]\colon (H,x)\to (K,y)$
and $[Q,g]\colon (K,y)\to (L,z)$
be morphisms in $\Alp(G)$, with chosen representatives.
\emph{Choose} $(T,u,v,\gamma)$ any Mackey square in~$\GG/G$ for~$P$ and~$Q$ over~$K$, for instance an isocomma (\Cref{Conv:comma-isocomma}):
\begin{equation}
\label{eq:alp-composition}%
\vcenter{\xymatrix@C=2em@R=.2em{
&& T\ar@{ >->}[ld]_u \ar@{ >->}[rd]^v
\\
& P \ar@{ >->}[ld]_-{p_1} \ar@{ >->}[rd]_-{p_2} \ar@[gray]@{..>}@/_1em/[rddd]_-{\gray\stP}
& {\isocell{\gamma}}
& Q \ar@{ >->}[ld]^-{q_1} \ar@{ >->}[rd]^-{q_2} \ar@[gray]@{..>}@/^1em/[lddd]^-{\gray\stQ}
\\
H \ar@[gray]@{..>}@/_1em/[rrdd]_-{\gray\stH}
&& K \ar@[gray]@{..>}[dd]_-{\gray\stK}
&& L \ar@[gray]@{..>}@/^1em/[lldd]^-{\gray\stL}
\\
&&&
\\
&& \gray{G} & &}
}
\end{equation}
We define the composite to be
\[
[Q,g]\circ [P,f]
:=
[T,\ g\odot_\gamma f]
\]
where $T=(T,p_1 u,q_2 v)\in\SpanfG(H,L)$ has wing morphisms $p_1 u$ and~$q_2 v$ obtained by composition in~$\GG/G$, \textsl{structurando structurandis},
and the morphism $g\odot_\gamma f$ in~$\cA(T)$ is the composition of the restrictions of~$f$ and~$g$, suitably `matched' via~$\gamma$
\begin{equation}
\label{eq:odot}%
g\odot_{\gamma} f \;:=\; (p_1 u)^*(x)\xrightarrow{\,\displaystyle u^*(f)\,}(p_2 u)^*(y)
\xrightarrow[\simeq]{\,\gamma^*\,}(q_1 v)^*(y)
\xrightarrow{\,\displaystyle v^*(g)\,}(q_2 v)^*(z).
\end{equation}
The identity of $(H,x)$ is represented by $(H,\id_H,\id_H;\id_x)$.
\end{Cons}

\begin{Rem}
It is convenient to call~$H$ the `$\GG$-part' of an object~$(H,x)$ in~$\Alp(G)$ and to call~$x$ the `$\cA$-part' of~$(H,x)$.
Similarly, $P$ is the $\GG$-part of a morphism representative~$(P,f)$ and~$f$ is its~$\cA$-part.
\end{Rem}

\begin{Prop}\label{Prop:Alp-category}
For every object $G\in\GG$, the above~\Cref{Cons:Alp-category} gives a well-defined additive category $\Alp(G)$.
\end{Prop}

\begin{proof}
Using `strong $\appr$-equivalences', we see that composition does not depend on the choice of the Mackey square~\eqref{eq:alp-composition}.
To check that composition $[Q,g]\circ[P,f]$ does not depend on the representatives of the two morphisms~$(H,x)\to (K,y)\to (L,z)$, it suffices to discuss what happens with the one-step $\appr$-equivalence on one of the morphisms, one at a time, say for instance the first one.
Consider a local equivalence~$s\colon P\leto P'$, where $P=(P,p_1,p_2)$ and~$P'=(P',p'_1,p'_2)$, and let $\sigma_1\colon p_1\isoTo p'_1s$ and $\sigma_2\colon p_2\isoTo p'_2s$ be the wing cells of~$s$ as in~\eqref{eq:Alp-relation}.
We form the isocommas~$T=(p_2/q_1)$ and~$T'=(p_2'/q_1)$ to compose each representative with~$(Q,q_1,q_2;g)$ and we want to show that $(T,g\odot f)$ is one-step $\appr$-equivalent to~$(T',g\odot f')$. To see this, we construct the morphism
$w=s\times_K \id_Q = \big\langle s u,\ v,\ \gamma\cast\sigma_2^{-1}\big\rangle \colon (p_2/q_1)\to (p_2'/q_1)$:
\[
\vcenter{\xymatrix@C=18pt@R=18pt{
&& T=(p_2/q_1) \ar@{ >->}[ld]_-{u} \ar@{ >->}[rd]_-{v} \ar@{~>}[rrrd]^-{w} \ar@{}[d]|(.7){\isocell{\gamma}}
\\
& P \ar@{ >->}[rd]_-{p_2} \ar@{~>}[rrrd]^(.8){s} \ar@{}[rrrrdd]|{\isocell{\sigma_2}\quad}
&& Q \ar@{..>}[ld]_(.4){q_1\!\!} \ar@{::}[rrrd]
&& T'=(p_2'/q_1) \ar@{ >->}[ld]^(.6){\!\!u'} \ar@{ >->}[rd]^-{v'} &&
\\
&& K \ar@{=}[rrrd]
&& P' \ar@{ >->}[rd]^-{p_2'} 
&\ar@{}[d]|(.2){\isocell{\gamma'}} & Q \ar@{ >->}[ld]^-{q_1} 
\\
&&& 
&&	K 
}}
\]
We are in the situation of~\Cref{Lem:cube}\,\eqref{it:cube-tricky}, which guarantees that the square $s u = u' w$ is Mackey and consequently~$w$ is also a local equivalence by~\Cref{Lem:leq-pb}. Hence we can apply base-change~\Cref{Lem:trace-base-change} to~$f$ to get $\tr_{w}(u^*(f))=u'^*(\tr_{s}(f))$; note that we do not need to `adjust' $u^*(f)$ as in~\Cref{Lem:trace-base-change} because the 2-cell $s u\isoTo u' w$ is the identity.
Let us switch to telegraphic style, mostly treating 2-cells as identities, to get the idea of the proof, and then restore the details afterwards. We have
\[
\begin{array}{rll}
\tr_{w}(g\odot_{\gamma} f) & = \tr_{w}\big(v^*(g)\circ u^*(f)\big) & \textrm{by~\eqref{eq:odot}, suppressing~$\gamma$}
\\[.1em]
& = \tr_{w}\big(w^*(v'^*(g))\circ u^*(f)\big) &\textrm{since }v=v' w
\\[.1em]
& = v'^*(g)\circ\tr_{w}(u^*(f)) & \textrm{by partial functoriality~\Cref{Lem:trace-special-comp}}
\\[.1em]
& = v'^*(g)\circ u'^*(\tr_{s}(f)) & \textrm{by base-change~\Cref{Lem:trace-base-change}}
\\[.1em]
& = v'^*(g) \circ u'^*(f') &\textrm{since $\tr_s(f)=f'$}
\\[.1em]
& = g \odot_{\gamma'} f' &\textrm{by~\eqref{eq:odot} again.}
\end{array}
\]
This indicates that $(T,g\odot_\gamma f)$ and $(T',g\odot_{\gamma'} f')$ are one-step $\appr$-equivalent.
Let us restore the 2-cells for accuracy. We do this once, as a `proof of concept' but will leave such details to the reader in the sequel. First clarify that the wing cells of~$w$ are the following obvious 2-cells, using $u' w=s u$ and $v' w=v$:
\begin{equation}
\label{eq:aux-w-wings}%
p_1 u \overset{\sigma_1 u}\To p'_1s u=p'_1u' w
\qquadtext{and}
q_2 v\overset{\id}\To q_2 v' w .
\end{equation}
The compatibility with the structure data over~$G$ is straightforward.
Remembering the sources and targets of $f\colon p_1^*(x)\to p_2^*(y)$ and~$g\colon q_1^*(y)\to q_2^*(z)$, the above `short' proof that $\tr_{w}(g\odot_{\gamma} f) = g \odot_{\gamma'} f'$ expands into
\[
\kern-.3em\begin{array}{rll}
	\tr_{w}\kern-1em&(g\odot_{\gamma} f) = \tr_{w}\big(\id (g\odot_{\gamma} f)((\sigma_1 u)^*)\inv\big) & \textrm{writing wings of~$w$~\eqref{eq:aux-w-wings} as in~\eqref{eq:f-adjusted}}
	\\[.1em]
	&
	= \tr_{w}\big(v^*(g)\,\gamma^*\, u^*(f)((\sigma_1 u)^*)\inv\big) & \textrm{by~\eqref{eq:odot} with explicit~$\gamma$}
	\\[.1em]
	& = \tr_{w}\big(w^*(v'^*(g))\, \gamma^* u^*(f){u^*((\sigma_1^*)\inv)}\big) &\textrm{since $v=v' w$ and $(\sigma_1 u)^*=u^*(\sigma_1^*)$}
	\\[.1em]
	& = \tr_{w}\big(w^*(v'^*(g) \gamma'^*)\, u^*(\sigma_2^* f{(\sigma_1^*)\inv})\big) &\textrm{since $\gamma\cast\sigma_2\inv=\gamma' w$ by def.\ of~$w$}
	\\[.1em]
	& = v'^*(g)\,\gamma'^* \,\tr_{w}\big(u^*(\sigma_2^* f{(\sigma_1^*)\inv})\big) & \textrm{by partial functoriality~\Cref{Lem:trace-special-comp}}
	\\[.1em]
	& = v'^*(g)\,\gamma'^*\,u'^*\big(\tr_{s}(\sigma_2^* f{(\sigma_1^*)\inv})\big) & \textrm{by base-change~\Cref{Lem:trace-base-change}}
	\\[.1em]
	& =  v'^*(g)\,\gamma'^*\,u'^*(f') &\textrm{by $\tr_{s}(f^\sigma)=f'$, expanded}
	\\[.1em]
	& = g \odot_{\gamma'} f' &\textrm{by~\eqref{eq:odot} with explicit~$\gamma'$.}
\end{array}
\]

From now on, we shall write the short versions of such proofs, as in the first version above, trusting that the reader can restore the `adjusting' 2-cells and the structure over~$G$, as needed.

It is easier to check that composition is unital and associative. For the latter, one can use associativity of isocommas as in~\cite[Remark~2.1.8]{BalmerDellAmbrogio20}.

Finally, the category $\Alp(G)$ is additive. The biproduct is
\begin{equation}
\label{eq:Alp-oplus}%
(H,x)\oplus(K,y)=(H\sqcup K,(x,y))
\end{equation}
with the obvious faithful morphism $\st{H\sqcup K}:=(\stH\ \stK)\colon H\sqcup K\into G$ and where $(x,y)\in \cA(H)\oplus \cA(K)$ is identified with an object of~$\cA(H\sqcup K)$ by additivity of the 2-functor~$\cA$.
The reader can verify that this produces biproducts in~$\Alp(G)$,
using additivity properties of local equivalences and their traces.

The zero object of~$\Alp(G)$ is~$(\varnothing,0)$ and the zero morphism is~$0=[\varnothing,0]$.
Using the local equivalence $\varnothing \to P$ whose trace map is zero, one gets that
\begin{equation}
\label{eq:Alp-zero}%
[P,0]=0
\end{equation}
for any span~$P\in\SpanfG(H,K)$.

Once it admits biproducts, our category~$\Alp(G)$ becomes semi-additive, \ie has an associative and commutative addition of morphisms, compatible with composition. In our case, say, between the objects~$(H,x)$ to~$(K,y)$, it reads as follows:
\begin{equation}
\label{eq:Alp-addition}%
[P,f]
+[Q,g]
=[P\sqcup Q,(f,g)]
\end{equation}
where $P\sqcup Q$ has the obvious wings~$p_i\sqcup q_i$ for~$i=1,2$ in~$\GG/G$ and where the morphism $(f,g)\colon (p_1^*x,q_1^*x)\to (p_2^*y,q_2^*y)$ in the category~$\cA(P\sqcup Q)\cong\cA(P)\oplus \cA(Q)$ is defined component-wise.
In the special case where both morphisms have the same span~$P=Q$, $p_1=q_1$, $p_2=q_2$, we actually have
\begin{equation}
\label{eq:Alp-special-addition}%
[P,f]
+[P,g]
=[P,f+g].
\end{equation}
This uses the local equivalence $s=\nabla_P^{(2)}=(\id_P\ \id_P)\colon P\sqcup P\to P$ (with trivial wing cells $\sigma_1=\id_{(p_1\ p_1)}$ and~$\sigma_2=\id_{(p_2\ p_2)}$) for which $s^*=\Delta\colon \cA(P)\to \cA(P)\oplus\cA(P)$ is the diagonal, and whose trace satisfies $\tr_{s}(f,g)=f+g$ by~\Cref{Exa:trace-sum}.

By~\eqref{eq:Alp-special-addition} and~\eqref{eq:Alp-zero}, addition of morphisms admits an opposite: $-[P,f]=[P,-f]$. Therefore $\Alp(G)$ is not only semi-additive but plain additive.
\end{proof}

\begin{Rem}
\label{Rem:Alp-add}%
We record for future use the special addition of morphisms~\eqref{eq:Alp-special-addition} in case they have the same $\GG$-part.
One proves in similar fashion the special addition of objects with same $\GG$-part (instead of~\eqref{eq:Alp-oplus} in general):
\[
(H,x_1)\oplus(H,x_2)\cong (H,x_1\oplus x_2).
\]
\end{Rem}

We now turn to the variance of~$\Alp(G)$ in~$G$ with respect to restriction.

\begin{Cons}\label{Cons:Alp-restriction}
Let $j\colon G'\to G$ be a morphism in~$\GG$. The restriction functor
\[
j^*\colon \Alp(G)\to \Alp(G')
\]
uses~\Cref{Cons:j!j*-GG/} on the groupoid part.
On an object $(H,x)$ we set
\[
j^*(H,x):=\big(j^*H,\ \pr_2^*(x)\big)
\]
where $j^*H=G'\times_G H$ is the pullback, which lives over~$G'$ via~$\pr_1$ (\Cref{Cons:j!j*-GG/}) and where the morphisms~$\pr_1$ and~$\pr_2$ come from the isocomma square
\begin{equation}
\label{eq:j^*H}%
\vcenter{
\xymatrix@C=1em@R=.5em{
& j^*H \ar@{ ->}[rd]^{\pr_2} \ar@{ >->}[ld]_{\pr_1} \ar@{}[dd]|{\isocell{\gamma_{j/\stH}}} & \\
G' \ar@{ ->}[rd]_j &  & H \ar@{ >->}[ld]^\stH \\
& G &}}
\end{equation}
so that $\pr_2^*(x)$ belongs to~$\cA(j^*H)$ as required.
To give~$j^*\colon \Alp(G)\to \Alp(G')$ on morphisms, let $(P,p_1,p_2;f)\colon (H,x)\to (K,y)$ be a morphism representative in~$\Alp(G)$. Applying~$j^*=(j/-)$ as in~\Cref{Lem:cube-exists} (or~\Cref{Rem:j!j*-Span}) we obtain a commutative cube of 2-cells (\Cref{Def:cube})
\begin{equation}\label{eq: restriction}
\vcenter{
\xymatrix@L=2pt@C=20pt@R=20pt{
& P'
	\ar@{->}@{ ->}[drrr]^-{p}
	\ar@{..>}@{ >->}[dr]^(.6){p'_2}
	\ar@{ >->}[dl]_-{p'_1}
    \\
j^*H
	\ar[drrr]_(.6){\pr_2}
	\ar@[gray]@{ >->}[dr]_{\gray{\pr_1}}
&& j^*K
	\ar@[lightgray]@{ >->}[dl]_(.3){\gray{\pr_1}}
	\ar@{ ->}[drrr]|(.53){\hole}^(.3){\pr_2}
&& P
	\ar@{}[dd]
	\ar[dr]^{p_2}
	\ar[dl]_(.3){p_1}
\\
& G'
	\ar@{->}[drrr]_-{j}
&& H
	\ar@[gray]@{ >->}[dr]_(.4){\gray{\stH}}
	\ar@{}[ll]
&& K
	\ar@[gray]@{ >->}[dl]^{\gray{\stK}}
\\
&&&& G
}}
\end{equation}
all of whose four side faces are Mackey squares
with moreover $\pr_2 p'_1=p_1 p\colon P'\to H$ and $\pr_2 p'_2=p_2 p\colon P'\to K$. Consequently $p^*(f)$ is a morphism from~$(p'_1)^*(\pr_2^*x)$ to~$(p'_2)^*(\pr_2^* y)$ in~$\cA(P')$ and we can define
\[
j^*(P,p_1,p_2;f)=\big(P',p'_1,p'_2;p^*(f)\big)
\]
as a morphism representative from $j^*(H,x)$ to~$j^*(K,y)$ in~$\Alp(G')$. This construction preserves $\appr$-equivalence of morphisms:
\end{Cons}
\begin{Prop}
\label{Prop:Alp-j*}%
For every $j\colon G'\to G$,~\Cref{Cons:Alp-restriction} gives a well-defined additive functor $j^*\colon \Alp(G)\to \Alp(G')$.
\end{Prop}
\begin{proof}
Straightforward by applying~$j^*=(j/-)$ to all groupoids (\Cref{Lem:cube-exists}) and pulling-back the data in~$\cA(-)$ accordingly.
For the $\appr$-equivalence of morphisms, use that pull-backs of local equivalences remain local equivalences (\Cref{Lem:leq-pb}) and the base-change for traces (\Cref{Lem:trace-base-change}). For $j^*$ preserving composition, use that $j^*\colon \GG/G\to \GG/G'$ preserves Mackey squares (\Cref{Lem:cube-pb}).
\end{proof}

\begin{Cons}
\label{Cons:Alp-alpha*}%
Let $\alpha\colon j \To k$ be a $2$-cell in~$\GG$, for $j,k\colon G'\to G$. We define a natural transformation
$\alpha^*\colon j^* \To k^*\colon \Alp(G)\to \Alp(G')$ as follows. Let $(H,x)\in \Alp(G)$ be an object. Recall that
$j^*(H,x) = \big(j^*H,\ \pr_2^*(x)\big)$
where
\[
\vcenter{
	\xymatrix@C=1em@R=1em{
		& j^*H\ar@{->}[rd]^{\pr_2} \ar@{->}[ld]_{\pr_1}
		\ar@{}[dd]|{\isocell{\gamma_{j/\stH}}}
		& \\
		G' \ar@{->}[rd]_j && H \ar@{ >->}[ld]^{\stH} \\
		& G &
}}
\]
is the isocomma square~$(j/\stH)$, and similarly $k^*(H,x)=\big(k^*H,\ \pr_2^*(x)\big)$.
The $2$-cell $\alpha$ induces a morphism $(\alpha/\stH)\colon j^*H=(j/\stH)\to (k/\stH)=k^*H$ (\Cref{Cons:(j/-)=>(k/-)}) defined by $(\alpha/\stH) =\langle \pr_1,\ \pr_2,\ \gamma_{j/\stH} \cast \alpha^{-1}\rangle$.
We define the morphism
\[
\alpha^*_{(H,x)}
:=
[j^*H,\id_{j^*H},(\alpha/\stH);\id_{\pr_2^*(x)}]
\colon j^*(H,x)\to k^*(H,x)
\]
in~$\Alp(G')$, which makes sense because $\pr_2\circ (\alpha/\stH)=\pr_2$.
This construction is natural in~$(H,x)$ and therefore defines a natural transformation
\[
\alpha^*\colon j^* \To k^*\colon \Alp(G)\to \Alp(G').
\]
\end{Cons}
\begin{Prop}
\label{Prop:Alp-2-functor}%
The above~\Cref{Cons:Alp-category,Cons:Alp-restriction,Cons:Alp-alpha*} define a 2-functor $\Alp\colon \GG^\op\to \ADD$.
\end{Prop}
\begin{proof}
The data has been made explicit and verifications are straightforward.
\end{proof}
\begin{Rem}
If the reader does not want to treat the canonical groupoid identification $G''\times_{G''}G'\cong G'$ as an equality, as we decided to do in~\Cref{Rem:j*-GG}, then the above~$\Alp$ is only a \emph{pseudo}-functor. See strictification in~\Cref{Rec:2-functors}.
\end{Rem}

We now turn to induction for~$\Alp$.
\begin{Cons}\label{Cons:Alp-induction}
Let $j\colon G'\into G$ be faithful. We define induction
\[
j_*\colon \Alp(G')\to \Alp(G)
\]
on objects by
\[
j_*(H',x'):=(j_!H',x')
\]
where $j_!\colon \GG/G'\to \GG/G$ is as in~\Cref{Cons:j!j*-GG/}, namely simply
$j_!H'=H'$ with structure morphism over~$G$ given by~$\st{j_!H'}=j\circ\stHp$.
Similarly on morphisms
\[
j_*[P',f']
:=
[j_!P',f'].
\]
\end{Cons}

\begin{Prop}
\label{Prop:Alp-induction-functor}%
For every $j\colon G'\into G$, the above defines an additive functor
\[
j_*\colon \Alp(G')\to \Alp(G).
\]
\end{Prop}
\begin{proof}
The functor $j_*$ leaves the spans and the morphisms in $\cA$ unchanged and
only composes the structural $2$-cells with $j$. Hence it preserves identities,
composition, and addition. Since the equivalence relation is generated by
traces along local equivalences and $j_*$ does not alter the data involved in
the trace construction, it respects this relation. It follows that $j_*$ is a
well-defined additive functor.
\end{proof}

\begin{Prop}\label{Prop:Alp-adjunction}%
For every faithful $j\colon G'\into G$, the functor $j_*\colon \Alp(G')\to \Alp(G)$ is a special Frobenius two-sided adjoint $j_*\adj j^*\adj j_*$ of $j^*\colon \Alp(G)\to \Alp(G')$.
\end{Prop}
\begin{proof}
For $j_*\adj j^*$ and~$j^*\adj j_*$ we need natural transformations as in~\eqref{eq:4-units}
\begin{equation}
\label{eq:4-units-j}%
\vcenter{
\xymatrix@R=.5em{
	\leta\colon \Id \To j^*j_*
	& \reta\colon \Id \To j_*j^*
	\\
	\leps\colon j_* j^*\To \Id
	& \reps\colon j^*j_*\To \Id.\!
}}
\end{equation}

We begin with~$\leta\colon \Id_{\Alp(G')}\To j^*j_*$. Let $(H',x')\in \Alp(G')$.  We have
\[
j^*j_*(H',x')
=
(j^*j_!H',\ \pr_2^*(x'))
\]
where $j^*j_!H'=(j/j \st{H'})=G'\times_G H'$ and $\pr_2\colon j^*j_!H'\to H'$ is the second projection.
\[
\vcenter{
\xymatrix@C=2em@R=1em{
& H' \ar@{ ->}@/_1em/[ldd]_-{\st{H'}} \ar@{=}@/^1em/[rdd]^-{\id} \ar@{ >->}[d]^(.6){\eta_{H'}}  \ar@{}[ldd] \ar@{}[rdd] & \\
& j^*j_!H' \ar@{ >->}[ld]^-{\pr_1} \ar@{ >->}[rd]_-{\pr_2} \ar@{}[dd]|{\isocell{\gamma_{j/j\st{H'}}}} & \\
G' \ar@{ >->}[rd]_{j} & & H' \ar@{ >->}[ld]^{j\st{H'}} \\
& G &
}}
\]
We already saw in~\eqref{eq:j!j*-eta} the unit $\eta_{H'}=<\st{H'},\id_{H'},\id_{j \st{H'}}>\colon H'\into j^*j_!H'$ for the $j_!\adj j^*$ adjunction on comma categories.
Define the morphism in~$\Alp(G')$
\begin{equation}
\label{eq:Alp-leta}%
\leta_{(H',x')}:=
[{H'},\id_{H'},\eta_{H'};\id_{x'}]\colon (H',x')\to j^*j_*(H',x')=(j^*j_!H',\pr_2^*(x'))
\end{equation}
using that $\pr_2\circ\eta_{H'}=\id_{H'}$. In our representation~\eqref{eq:Alp-morphism}, this $\leta_{(H',x')}$ is given by
\[
\vcenter{\xymatrix@C=2em@R=.5em{
& {H'} \ar@{=}[ld]_-{} \ar@{ ->}[rd]^-{\eta_{H'}} \ar@{}[dd] |{\gray{\isocell{\id}}} & \\
{H'} \ar@[gray]@{ >->}[rd]_{\gray{\st{H'}}} & & j^*j_!H' \ar@[gray]@{ >->}[ld]^{\gray {\pr_1}} \\
& \gray{G'} &
}}
\quadtext{with}
\id\colon x'\to \eta_{H'}^*(\pr_2^*(x'))=x'\textrm{ in }\cA(H').
\]
We must check that $\leta_{H'}$ is natural in~$H'$. We are not going to expand every single one of those verifications as proof of concept but we do provide the details in this first occurrence. Let $[P',f']\colon (H',x')\to(K',y')$ be a morphism in $\Alp(G')$, where $P'=(P',p'_1,p'_2)\in \SpanfG[G'](H',K')$ and where
$f'\colon {p'}_1^*(x')\to {p'}_2^*(y')$ is a morphism in~$\cA(P')$. We must prove that
\begin{equation}
\label{eq:leta-natural}%
j^*j_*([P',f'])\circ\leta_{(H',x')}
=
\leta_{(K',y')}\circ[P',f'].
\end{equation}
For the right-hand side, the composition diagram in~$\GGf/G'$ is
\[
\xymatrix@C=2em@R=.5em{
&& P'\ar@{ =}[ld] _{}\ar@{ >->}[rd]^{p'_2} &&\\
&P' \ar@{ >->}[ld]_-{p'_1} \ar@{ >->}[rd]_-{p'_2} & {\isocell{\id_{p'_2}}}	& K' \ar@{ =}[ld]^-{} \ar@{->}[rd]^-{\eta_{K'}}  &  \\
H'
&& K'
&& j^*j_! K'
}
\]
whose square is Mackey by~\Cref{Exa:equi-Mackey}. Since $\id\odot_{\id}f'=f'$, the right-hand side of~\eqref{eq:leta-natural} is~$[P',p'_1,\eta_{K'}p'_2;f']$.
For the left-hand side, the composition diagram is
\[
\xymatrix@C=2em@R=.5em{
&&P'\ar@{ ->}[ld] _{p'_1}\ar@{ >->}[rd]^{\eta_{P'}} &&\\
&H' \ar@{ =}[ld]_-{} \ar@{ >->}[rd]_-{\eta_{H'}} & {\isocell{\id}}	& j^*j_!P' \ar@{ ->}[ld]^-{j^*j_*p'_1} \ar@{ ->}[rd]^-{j^*j_*p'_2}  &  \\
H'
&& j^*j_!H'
&& j^*j_!K'
}
\]
where the 2-cell identity comes from the naturality of~$\eta$ in the adjunction $j_!\dashv j^*$, as in the left-hand square below:
\[
\xymatrix@C=3em@R=2em{
P' \ar[r]^-{\eta_{P'}} \ar[d]_{p'_1}
& j^*j_!P' \ar[r]^-{\pr_2} \ar[d]|-{j^*j_*p'_1=G'\times_G p'_1}
& P' \ar[d]^{p'_1}
\\
H' \ar[r]_-{\eta_{H'}}
&
j^*j_!H' \ar[r]_-{\pr_2}
& H'.\!\!
}
\]
In this commutative diagram the composite square is Mackey (\Cref{Exa:equi-Mackey}) for $\pr_2\eta=\id$, and the right-hand square is Mackey by~\Cref{Lem:pb-Mackey}. Hence the left-hand square is Mackey by~\Cref{Lem:Mackey-square-arith}. Using similarly that $j^*j_*p'_2\circ\eta_{P'}=\eta_{K'}\circ p'_2$ the left-hand side of~\eqref{eq:leta-natural} is $[P',p'_1,\eta_{K'}p'_2;f']$, matching the right-hand side.

\smallbreak

We now define~$\leps\colon j_*j^*\To\Id_{\Alp(G)}$ in~\eqref{eq:4-units-j}. Let $(H,x)\in \Alp(G)$. We have
\[
j_*j^*(H,x)=(j_!j^*H,\pr_2^*(x))
\]
where $j_!j^*H=G'\times_G H=(j/\stH)$
and $\pr_2\colon j^*H\into H$ is the projection.
We already saw in~\eqref{eq:j!j*-eps} the counit $\eps_{H}=\pr_2\colon j_!j^*H\to H$ and we can define the morphism
\begin{equation}
\label{eq:Alp-leps}%
\leps_{(H,x)}:=
[j_!j^*H,\id,\eps_H;\id_{\pr_2^*x}]\colon j_*j^*(H,x)=(j_!j^*H,\pr_2^*(x))\too (H,x)
\end{equation}
or represented in the style of~\eqref{eq:Alp-morphism}
\[
\vcenter{\xymatrix@C=3em@R=12pt{
		& j_!j^*H \ar@{=}[ld]_-{} \ar@{ >->}[rd]^-{\eps_H}
		\ar@[gray]@{ >->}[dd]|-{\gray{\st{j_!j^*H}}} \ar@{}[rdd]|-{\gray{\isocell{\gamma_{j/\stH}}}} \ar@{}[ldd]|-{\gray{=}}
		& \\
		j_!j^*H \ar@[gray]@/_.5em/@{ >->}[rd]_{\gray{j\pr_1}} & & H \ar@[gray]@/^.5em/@{ >->}[ld]^{\gray{\stH}} \\
		& \gray{G} &
}}
\quadtext{with}
\id_{\pr_2^*(x)}\colon \pr_2^*(x) \to \eps_H^*(x)=\pr_2^*(x).
\]
The proof of the naturality of~$\leps_{(H,x)}$ in~$(H,x)\in\Alp(G)$ is similar to the above one, using naturality of~$\eps$ and~\Cref{Lem:pb-Mackey} again. It is omitted.
In particular, it does not use any trace relation.

\smallbreak

We verify the triangle identities for $j_*\dashv j^*$.  Let
$(H',x')\in\Alp(G')$. To see that $\leps_{j_*(H',x')}\circ j_*(\leta_{(H',x')})=\id_{j_*(H',x')}$, we compute
\[
j_*(\leta_{(H',x')})=[H',\id_{H'},\eta_{H'};\id]
\quadtext{and}
\leps_{j_*(H',x')}= [j_!j^*j_!H',\id,\eps_{j_!H'};\id]
\]
and compose the spans as in the following diagram
\[
\xymatrix@C=2em@R=.5em{
&& H' \ar@{ =}[ld] _{}\ar@{ >->}[rd]^{\eta_{H'}} &&\\
&H' \ar@{ =}[ld]_-{} \ar@{ >->}[rd]_-{\eta_{H'}}
& {\isocell{\id}}
& j^*j_!H' \ar@{ =}[ld]^-{} \ar@{ >->}[rd]^-{\eps_{j_!H'}=\pr_2}
&  \\
H'
&& j^*j_!H'
&& H'
}
\]
whose square is Mackey by~\Cref{Exa:equi-Mackey}. The morphism part is the identity. Since $\pr_2\circ\eta_{H'}=\id$, we get the result. For the other unit-counit relation, let ${(H,x)}\in\Alp(G)$ and let us check $j^*(\leps_{(H,x)})\circ \leta_{j^*{(H,x)}}=\id_{j^*{(H,x)}}$. Direct computation gives
\[
\leta_{j^*{(H,x)}}
=
[j^*H,\id_{j^*H},\eta_{j^*H};\id]
\quadtext{and}
j^*(\leps_{(H,x)})
=
[j^*j_!j^*H,\id_{j^*j_!j^*H},j^*(\eps_H);\id].
\]
The composition diagram is
\[
\xymatrix@C=2em@R=.5em{
&& j^*H \ar@{ =}[ld] _{}\ar@{ >->}[rd]^{\eta_{j^*H}} &&\\
&j^*H \ar@{ =}[ld]_-{} \ar@{ ->}[rd]_-{\eta_{j^*H}}
& {\isocell{\id}}
& j^*j_!j^*H \ar@{ =}[ld]^-{} \ar@{ ->}[rd]^-{j^*(\eps_H)}
&  \\
j^*H
&& j^*j_!j^*H
&& j^*H
}
\]
whereas the morphism part is again the identity. We conclude from the triangle identity for $j_!\dashv j^*$ on comma categories,
namely $j^*(\eps_H)\circ\eta_{j^*H}=\id_{j^*H}$.
\smallbreak

Thus $j_*\dashv j^*$.

\smallbreak

Similarly, we can give the unit and counit~$\reta$ and~$\reps$ in~\eqref{eq:4-units-j} for the right adjunction $j^*\dashv j_*$.
In fact they can be described from the left ones~\eqref{eq:Alp-leta} and~\eqref{eq:Alp-leps} by applying an obvious transposition on the spans $\Spanname(H,K)\to \Spanname(K,H)$, swapping the legs, and by keeping the identities in the morphism-parts, giving:
\begin{equation}
\label{eq:Alp-reta}%
\reta_{(H,x)}
:=
[j_!j^*H,\eps_H,\id_{j_!j^*H};\id_{\pr_2^*x}]\colon
(H,x)\to j_*j^*(H,x)
\end{equation}
for every $(H,x)\in \Alp(G)$ -- compare~\eqref{eq:Alp-leps} -- and
\begin{equation}
\label{eq:Alp-reps}%
\reps_{(H',x')}
:=
[H',\eta_{H'},\id_{H'};\id_{x'}]\colon
j^*j_*(H',x')\to (H',x')
\end{equation}
for every $(H',x')\in\Alp(G')$ -- compare~\eqref{eq:Alp-leta}.
The reader will verify naturality of~$\reta$ and~$\reps$ as well as the unit-counit relations for~$j^*\adj j_*$.

\smallbreak

It remains to check the special Frobenius property. For
$(H',x')\in\Alp(G')$, the composite of~$\leta_{(H',x')}=[H',\id_{H'},\eta_{H'};\id_{x'}]$ in~\eqref{eq:Alp-leta} followed by
$\reps_{(H',x')}=[H',\eta_{H'},\id_{H'};\id_{x'}]$ in~\eqref{eq:Alp-reps} is easily seen to be the identity~$[H',\id]$ once we observe that the commutative square
\[
\xymatrix@C=2em@R=.5em{
& H' \ar[rd]^-{\id_{H'}} \ar[ld]_{\id_{H'}} \ar@{}[dd]|-{=}
\\
H' \ar[rd]_-{\eta_{H'}}
&&
H' \ar[ld]^-{\eta_{H'}}
\\
& j^*j_!H'
}
\]
is Mackey by~\Cref{Exa:equi-Mackey}, since $\eta_{H'}$ is fully faithful (\Cref{Rem:eta-eps-j!j*}).
\end{proof}

\begin{Prop}[Base-change for $\Alp$]
\label{Prop:Alp-base-change}%
Let
\[
\xymatrix@C=2em@R=.5em{
& L \ar@{ >->}[rd]^{j} \ar@{ ->}[ld]_{v} \ar@{}[dd]|{\isocell{\alpha}} & \\
H \ar@{ >->}[rd]_i &  & K \ar@{ ->}[ld]^u \\
& G &
}
\]
be a Mackey square in $\GG$, with $i$ and $j$ faithful. Then the left and right mates of~$\alpha^*$
\[
(\alpha^{-1})_*\colon u^*i_* \isoTo j_*v^*
\qquad\text{and}\qquad
\alpha_!\colon j_*v^*\isoTo u^*i_*
\]
are inverse isomorphisms between functors from~$\Alp(H)$ to~$\Alp(K)$.
\end{Prop}
\begin{proof}
We compute the two mates on $(X,x)\in\Alp(H)$, with~$\st{X}\colon X\into H$.
We have
\begin{equation}
\label{eq:aux-Alp-BC}%
u^*i_*(X,x)=\big(u^*i_!X, \pr_2^*x\big)
\quadtext{and}
j_*v^*(X,x)=\big(j_!v^*X,\pr_2^*x\big)
\end{equation}
where $u^*i_!X=(u/i\stX)=K\times_G X$ and $j_!v^*X=(v/\stX)=L\times_H X$.
These two objects of~$\GG$ appear in the following diagram (at the top):
\begin{equation}
\label{eq:aux-Alp-BC-big}%
\vcenter{\xymatrix@C=2em@R=1em{
&& K\times_G X \ar@/_2em/[llddd]_-{\pr_1} \ar@/^2em/[rdd]^-{\pr_2} \ar@{<-}[d]_-{\cong}^-{e_X}
\\
&& L\times_H X \ar[ld]_-{\pr_1} \ar[rd]^-{\pr_2} \ar@{}[dd]|{\isocell{\gamma_{v/\stX}}}
& \\
& L \ar[rd]_v \ar[ld]_-{j} \ar@{}[dd]|-{\isocell{\alpha\inv}}
&& X \ar[ld]^{\stX}  \ar@/^2em/[lldd]^(.3){i\st{X}}
\\
\ K \ \ar[rd]_-{u}
&& H \ar[ld]^-{i}
&
\\
& G &
}}
\end{equation}
The outer square is an isocomma and the juxtaposition of the two inner isocommas is a Mackey square. Hence we have an equivalence
\[
e_X=
\left\langle
j\pr_1,\ \pr_2,\
\gamma_{v/\stX}\cast\alpha^{-1}
\right\rangle\colon L\times_H X\isoto K\times_G X.
\]
Note in particular that $e_X^*(\pr_2^*x)=\pr_2^*x$.
We claim that the two base-change mates are given as follows:
\[
((\alpha^{-1})_*)_{(X,x)}=[L\times_H X,\ e_X,\ \id;\ \id_{\pr_2^*x}]\colon u^*i_*(X,x)\to j_*v^*(X,x)
\]
where $u^*i_*(X,x)$ and $j_*v^*(X,x)$ are described in~\eqref{eq:aux-Alp-BC}, while
\[
(\alpha_!)_{(X,x)}=[L\times_H X,\ \id,\ e_X;\ \id_{\pr_2^*x}]\colon j_*v^*(X,x)\to u^*i_*(X,x).
\]
This is a direct verification using $\alpha_!=\leps^{(j)}\cast\alpha^*\cast\leta^{(i)}$ and~$(\alpha^{-1})_*=\reps^{(i)}\cast(\alpha^{-1})^*\cast\reta^{(j)}$  and the formulas for those (co)units in~\eqref{eq:Alp-leta},~\eqref{eq:Alp-leps},~\eqref{eq:Alp-reta} and~\eqref{eq:Alp-reps}.
Even easier is the verification that the above two morphisms are inverse to each other, using the Mackey squares over equivalences provided by~\Cref{Exa:equi-Mackey}.
\end{proof}

Finally, we allow the restriction 2-functor~$\cA$, that was fixed so far, to vary.

\begin{Cons}
\label{Cons:Alp-transfos}%
Let $t\colon \cA\to \cA'$ be a morphism in~$\resfun{\GG}$ and let us define $\Alpt\colon \Alp\to \Alpp$. For every $G\in \GG$ define the additive functor $(\Alpt)_G\colon \Alp(G)\to \Alpp(G)$ on objects by $(H,x)\mapsto (H,t_H(x))$ and on morphisms by $(P,f)\mapsto (P,t_P(f))$, where $t_P(f)$ is `adjusted' by the restriction-compatibility isomorphisms of~$t$:
\[
\xymatrix@C=4em{
p_1^*t_H(x) \ar[r]^-{(t_{p_1})_x}_-{\cong}
& t_P(p_1^*x) \ar[r]^-{\displaystyle t_P(f)}
& t_P(p_2^*y) \ar[r]^-{(t_{p_2})_y^{-1}}_-{\cong}
& p_2^*t_K(y)
}
\]
For every~$j\colon G'\to G$ in~$\GG$, the compatibility isomorphism~$(\Alpt)_j\colon \Alpp(j)\circ(\Alpt)_G\isoTo(\Alpt)_{G'}\circ\Alp(j)$ is given on~$(H,x)\in\Alp(G)$ and with the notation of~\eqref{eq:j^*H} by
\[
[j^*H,\id,\id;t_{\pr_2}]\colon \big(j^*H,\pr_2^*(t_H(x))\big)\isoto \big(j^*H,t_{j^*H}(\pr_2^*(x))\big)
\]
in~$\Alpp(G')$, which is trivial on the `$\GG$-part' and~$t$ from the $\cA$-part to the~$\cA'$-part.
Finally, let $m\colon t\To t'\colon \cA\to\cA'$ be a modification. We define a modification
$\Alp[m]\colon \Alp[t]\To\Alp[t']$ as follows. For every~$G\in\GG$, the component
$ {(\Alp[m])}_G\colon(\Alpt)_G \To (\Alp[t'])_{G}$ is the natural transformation whose component at $(H,x)\in\Alp(G)$ is \[ (\Alp[m])_{G,(H,x)} := [H,\id_H,\id_H;m_{H,x}] \colon (H,t_H(x)) \to (H,t'_H(x)) \] in~$\Alpp(G)$.
One verifies that $\cA\mapsto \Alp$ is a 2-functor $\resfun{\GG}\to \resfun{\GG}$.
\end{Cons}

\begin{Thm}
\label{Thm:Alp-Mackey}%
For every restriction 2-functor~$\cA$, the 2-functor $\Alp\colon \GG^\op\to \ADD$ is a Mackey 2-functor.
Together with~\Cref{Cons:Alp-transfos} we obtain a 2-functor
\[
\Alp[(-)]\colon \resfun{\GG}\too \mackfun{\GG}
\]
meaning that for every transformation~$t\colon \cA\to \cA'$ in~$\resfun{\GG}$, the associated transformation~$\Alpt\colon \Alp\to \Alpp$ is compatible with induction.
\end{Thm}
\begin{proof}
We already have all ingredients to prove that $\Alp$ is a Mackey 2-functor. Additivity~\Mack{1} is a straightforward consequence of additivity of~$\cA$. The two-sided special Frobenius adjunction~\Mack{2} is~\Cref{Prop:Alp-adjunction}.
The base-change~\Mack{3} is~\Cref{Prop:Alp-base-change}.
To see that $\Alpt$ commutes with induction of~\Cref{Cons:Alp-induction}, we use here that nothing at all happens on the $\cA$-part, and hardly anything to the $\GG/G$-part as a matter of fact:
\[
\xymatrix@C=4em{
	{\big(} H',x'\in\cA(H')\big) \ar@{|->}[r]^-{(\Alpt)_{G'}} \ar@{|->}[d]_-{j_*}^-{\textrm{ in }\Alp}
	& {\big(} H',t_{H'}(x')\in\cA'(H')\big) \ar@{|->}[d]_-{j_*}^-{\textrm{ in }\Alpp}
	\\
	{\big(} j_!H',x'\in\cA(H')\big) \ar@{|->}[r]^-{(\Alpt)_G}
	& {\big(} j_!H',t_{H'}(x')\in\cA'(H')\big).
}
\]
Similarly for morphisms, which gives the result.
\end{proof}

%
\section{Biadjunction for~$\Alp[(-)]$}
\label{sec:Alp-Thm}%
%

We write $U:=\forget\colon \mackfun{\GG}\to \resfun{\GG}$ for the forgetful $2$-functor from Mackey $2$-functors to restriction $2$-functors.
We prove the universal property of the 2-functor $\Alp[(-)]$ constructed in~\Cref{sec:Alp}, as announced in~\Cref{Thm:Alp-UP-intro}.
In fact, we prove slightly more, namely we establish that $\Alp[(-)]$ is a left biadjoint to~$U$.
To this end, we construct a morphism of restriction $2$-functors $\etalp_{\cA}\colon \cA\to U(\Alp)$ for every $\cA\in \resfun{\GG}$ and a morphism of Mackey 2-functors
$\epslp_{\cM}\colon \Alp[(U\cM)]\to \cM$ for every~$\cM\in\mackfun{\GG}$.

\begin{Cons}
\label{Cons:Alp-unit}%
Let $\cA\in\resfun{\GG}$ be a restriction 2-functor (\Cref{Def:restriction-2-functor}). We define a morphism of restriction $2$-functors
\[
\etalp_{\cA}\colon \cA\to U(\Alp).
\]
Let $G\in\GG$. On an object~$x\in\cA(G)$, set
\[
\etalp_{\cA,G}(x):=(G,x)=(G\xto{\id_G}G,\ x)\in\Alp(G).
\]
On a morphism $f\colon x\to y$ in $\cA(G)$, set
\[
\etalp_{\cA,G}(f):=[G,\id_G,\id_G;f]
\colon
(G,x)\to(G,y).
\]
\end{Cons}

\begin{Prop}
\label{Prop:Alp-unit}
\Cref{Cons:Alp-unit} yields a morphism of $2$-functors
\[
\etalp_{\cA}\colon \cA\to U(\Alp).
\]
\end{Prop}
\begin{proof}
First, for every $G\in\GG$, the functor $\etalp_{\cA,G}\colon \cA(G)\to\Alp(G)$ is additive by~\Cref{Rem:Alp-add}.
For the compatibility with restriction, we need for every~$j\colon G'\to G$ a natural isomorphism $\etalp_{\cA,j}\colon j^*\circ \etalp_{\cA,G}\isoTo\etalp_{\cA,G'}\circ j^*$ of functors $\cA(G)\to \Alp(G')$. For every $x\in\cA(G)$, we have $j^*(\etalp_{\cA,G}(x))=(G'\times_G G,\pr_2^*x)$ and $\etalp_{\cA,G'}(j^*(x))=(G',j^*x)$. We define our isomorphism $(\etalp_{\cA,j})_x$ by
\[
\vcenter{
\xymatrix@C=3em@R=1em{
& G'\times_G G
\ar[ld]_-{\id}
\ar[rd]^-{\pr_1}
\ar@[gray]@{->}[dd]|-{\gray{\pr_1}}
\ar@{}[ldd]|(.45){\gray{\isocell{\id}}}
\ar@{}[rdd]|(.45){\gray{\isocell{\id}}}
& \\
G'\times_G G
\ar@[gray]@{->}[rd]_{\gray{\pr_1}}
&&
G'
\ar@[gray]@{->}[ld]^{\gray{\id_{G'}}}
\\
& \gray{G'} &
}}
\quadtext{with}
(\gamma_{G'\times_G G}\inv)^*\colon\pr_2^*x\to \pr_1^*j^*x
\]
where $\gamma_{G'\times_G G}\colon j\pr_1\isoto \pr_2$ is the 2-cell in the definition of~$G'\times_G G=(j/\id_G)$.
We now have all the necessary data and the reader can verify that they assemble in a morphism of $2$-functors $\etalp_{\cA}\colon \cA\to U(\Alp)$, as claimed.
\end{proof}

\begin{Cons}
\label{Cons:Alp-counit}%
Let $\cM$ be a Mackey $2$-functor (\Cref{Def:2-Mackey}). We define a morphism of Mackey 2-functors
\[
\epslp_{\cM}\colon \Alp[(U\cM)]\to \cM.
\]
Let $G\in\GG$. On an object~$(H,x)\in \Alp[(U\cM)](G)$ with~$H\in\GGf/G$ and~$x\in \cM(H)$,~set
\[
\epslp_{\cM,G}(H,x):=(\stH)_*(x)\in\cM(G)
\]
where $\stH\colon H\into G$ is the (tacit) structure morphism of~$H$ over~$G$ and~$(\stH)_*$ is induction in~$\cM$.
On a morphism $(P,f)\colon (H,x)\to (K,y)$ in $\Alp[(U\cM)](G)$, given by a span~$P=(P,p_1,p_2)\in \SpanfG(H,K)$ and $f\colon p_1^*(x)\to p_2^*(y)$ in~$\cM(P)$, define $\epslp_{\cM,G}([P,f])$ to be the push-forward of~$f$ along $\stP\colon P\into G$ in~$\cM$, suitably corrected by units and counits (as in the definition of the trace~\Cref{Def:trace}):
\begin{equation}
\label{eq:Alp-counit-f}%
\vcenter{
\xymatrix@C=3em{
{\stH}_*(x) \ar[d]_-{\reta}^-{({p_1}^*\adj {p_1}_*)} \ar[rrr]^-{\epslp_{\cM,G}([P,f])}_-{:=}
&&&{\stK}_*(y)
\\
{\stH}_*{p_1}_*p_1^*(x) \ar@{<-}[r]^-{(\st{p_1})_*}_-{\cong}
& {\stP}_*p_1^*(x)
 \ar[r]^-{\,{\stP}_*(f)\,}
& {\stP}_*p_2^*(y) \ar@{->}[r]^-{(\st{p_2})_*}_-{\cong}
& {\stK}_*{p_2}_*p_2^*(y)
 \ar[u]_-{\leps}^-{({p_2}_*\adj {p_2}^*)}
}}
\end{equation}
The vertical arrows use induction in the Mackey 2-functor~$\cM$.
The horizontal isomorphisms $(\st{p_i})_*$ come from pseudofunctoriality of~$(-)_*$ in~$\cM$ on the structure 2-cells $\st{p_1}\colon\stP\isoTo \stH p_1$ and $\st{p_2}\colon\stP\isoTo \stK p_2$ of the wings of~$P$, see~\eqref{eq:Span-object}.
\end{Cons}
\begin{Prop}
\label{Prop:Alp-counit}
\Cref{Cons:Alp-counit} yields a morphism of Mackey 2-functors
\[
\epslp_{\cM}\colon \Alp[(U\cM)]\to \cM.
\]
\end{Prop}
\begin{proof}
Let $G\in \GG$.
We first prove that $\epslp_{\cM,G}\colon \Alp[(U\cM)](G)\to \cM(G)$ is a well-defined functor.
Let us check that the formula for~$\epslp_{\cM,G}(P,f)$ in~\eqref{eq:Alp-counit-f} is independent of the chosen morphism representative~$(P,f)\colon (H,x)\to (K,y)$.
This is not immediate, as it explains our equivalence relation on morphisms.
Suppose that $(P,f)\appr(P',f')$ is a one-step $\appr$-equivalence given by a local equivalence $s\colon P\to P'$ in $\SpanfG(H,K)$, with wing cells~$\sigma_i\colon p_i\isoTo p'_i s$ for~$i=1,2$, such that $f'=\tr_s(f^\sigma)$, where $f^\sigma\colon s^*{p'_1}^*(x)\to s^*{p'_2}^*(y)$ is the adjusted morphism of~\eqref{eq:f-adjusted}.
With this notation, we can verify that the following diagram commutes in~$\cM(G)$:

\[
\xymatrix@C=4em@R=3em{
	{\stH}_*(x) \ar[r]^-{{\st{p_1}\inv}_*\,\reta^{(p_1)}} \ar@{=}[dd]^-{}
	& {\stP}_*p_1^*(x) \ar[r]^-{}^{{\stP}_*(f)} \ar[d]^-{{\st{s}}_*\sigma_1^*}_-{\cong}
	& {\stP}_*p_2^*(y) \ar[r]^-{{\st{p_2}}_*\leps^{(p_2)}} \ar[d]^-{{\st{s}}_*\sigma_2^*}_-{\cong}
	& {\stK}_*(y) \ar@{=}[dd]^-{}
	\\
	& {\st{P'}}_*s_*s^*{p'_1}^*(x) \ar[r]^-{}^{{\st{P'}}_*s_*(f^\sigma)}
	& {\st{P'}}_*s_*s^*{p'_2}^*(y)
	\\
	{\stH}_*(x) \ar[r]^-{{\st{p'_1}\inv}_*\,\reta^{(p'_1)}}
	& {\st{P'}}_*{p'_1}^*(x) \ar[u]^-{}_{\reta^{(s)}} \ar[r]^-{}^{{\st{P'}}_*(f')}
	& {\st{P'}}_*{p'_2}^*(y) \ar@{<-}[u]^-{}^{\leps^{(s)}} \ar[r]^-{{\st{p'_2}}_*\leps^{(p'_2)}}
	& {\stK}_*(y)
}
\]
The crucial square is the bottom middle one, which commutes by applying the functor $(\st{P'})_*\colon \cM(P')\to \cM(G)$ to the relation~$\tr_{s}(f^\sigma)=f'$. The top middle square is essentially the definition of~$f^\sigma$, combined with the structure 2-cell $\st{s}\colon \stP\isoTo \st{P'}s$.
The sides of the diagram commute by general compatibility of (co)units and push-forward in the (rectified) Mackey 2-functor~$\cM$, using that the wing 2-cells $\sigma_i\colon p_i\isoTo p_i' s$ are compatible with the structure 2-cells $\st{p_i}$, $\st{p_i'}$ and~$\st{s}$, as in~\eqref{eq:Span-morphism}.
The composite in the top row is $\epslp_{\cM,G}(P,f)$ while the composite in the bottom row is $\epslp_{\cM,G}(P',f')$.

The reader can verify that $\epslp_{\cM,G}\colon \Alp[(U\cM)](G)\to \cM(G)$ is functorial and additive. (This does not use traces.)

For compatibility with restriction, we need for every $j\colon G'\to G$ an isomorphism
\[
\epslp_{\cM,j}\colon j^*\circ \epslp_{\cM,G}\isoTo \epslp_{\cM,G'}\circ j^*
\]
of functors~$\Alp[(U\cM)](G)\to \cM(G')$.
For $(H,x)\in\Alp[(U\cM)](G)$ form the isocomma
\[
\vcenter{
\xymatrix@C=1.5em@R=.5em{
& j^*H \ar@{ >->}[ld]_-{\pr_1} \ar[rd]^-{\pr_2}
\ar@{}[dd]|{\isocell{\gamma}}
\\
G' \ar[rd]_-j && H \ar@{ >->}[ld]^-{\stH}
\\
& G
}}
\]
and note that $j^*\epslp_{\cM,G}(H,x)=j^*{\stH}_*(x)$, whereas $\epslp_{\cM,G'}j^*(H,x)={\pr_1}_*\pr_2^*(x)$. We can relate these two objects by the isomorphism
\begin{equation}
\label{eq:epslp_j}%
(\epslp_{\cM,j})_{(H,x)}:=\big(\gamma_*\big)_{x}\ \colon\
j^*{\stH}_*(x)\isoto {\pr_1}_*\pr_2^*(x)
\end{equation}
given by the Mackey base-change isomorphism of Axiom~\Mack{3} for~$\cM$.
We omit the proof that this data defines a pseudo-natural transformation $\epslp_{\cM}\colon \Alp[(U\cM)]\to \cM$ of restriction 2-functors. (Again, it holds without using the trace.)

It remains to see compatibility with induction. Let $j\colon G'\into G$ be
faithful. For an object $(H',x')\in\Alp[(U\cM)](G')$, we have
$\epslp_{\cM,G}\circ j_*(H',x')=(j\st{H'})_*(x')$, whereas
$j_*\epslp_{\cM,G'}(H',x')=j_*{\st{H'}}_*(x')$. They are related by the canonical pseudofunctoriality isomorphism for induction in the (rectified) Mackey 2-functor~$\cM$
\[
(j\st{H'})_*\isoTo j_*{\st{H'}}_*\colon \cM(H')\to \cM(G)
\]
evaluated at~$x'\in\cM(H')$.
The reader will verify that this isomorphism is natural in~$(H',x')\in \Alp[(U\cM)](G')$ and that it is indeed the mate of the restriction-comparison isomorphism~$\epslp_{\cM,j}$ of~\eqref{eq:epslp_j}.
Thus $\epslp_{\cM}$ is compatible with induction, \ie it is a morphism
of Mackey $2$-functors.
\end{proof}

\begin{Prop}
\label{Prop:Alp-unit-counit}%
For every Mackey 2-functor $\cM\in\mackfun{\GG}$ and for every restriction 2-functor $\cA\in\resfun{\GG}$ we have equalities of pseudo-natural transformations
\[
U(\epslp_{\cM})\circ \etalp_{U(\cM)}=\Id_{U(\cM)}
\qquadtext{and}
\epslp_{(\Alp)}\circ \Alp[(\etalp_{\cA})]=\Id_{\Alp}.
\]
\end{Prop}
\begin{proof}
For the first one, let $G\in\GG$ and~$x\in \cM(G)$. Using~\Cref{Cons:Alp-unit,Cons:Alp-counit}, the composite $\epslp_{\cM,G}\circ \etalp_{U(\cM),G}\colon \cM(G)\to \cM(G)$ maps~$x$ to~$(\id_G)_*(x)=x$ by our tacit choice of `rectified' induction in~$\cM$ (see~\Cref{Rem:rectify-Mackey}). Same holds on morphisms.

For the second one, let $G\in \GG$ and~$(H,x)\in\Alp(G)$, for~$H\in\GGf/G$ and~$x\in\cA(H)$. Using~\Cref{Cons:Alp-unit,Cons:Alp-transfos,Cons:Alp-counit}, the composite $\epslp_{(\Alp),G}\circ (\Alp[(\etalp_{\cA})])_{G}\colon$ $\Alp(G)\to \Alp(G)$ maps~$(H,x)$ to~$(H,(\id_H)_*(x))$, using that $\id_H$ is a local equivalence. Again, we chose $\id_*=\id$ (\Cref{Rem:rectified-leq}) so we get the result.

The same holds on morphisms. Indeed, let $[P,f]\colon (H,x)\to (K,y)$ be represented by $P=(P,p_1,p_2)\in\SpanfG(H,K)$ and $f\colon p_1^*(x)\to p_2^*(y)$ in~$\cA(P)$. Applying $\Alp[(\etalp_{\cA})]$ replaces the $\cA$-part~$f$ by $\etalp_{\cA,P}(f)=[P,\id_P,\id_P;f]$ as a morphism in~$\Alp(P)$. Then applying $\epslp_{(\Alp),G}$ pushes this morphism forward along the structure morphism $\stP\colon P\into G$. By the definition of induction in~$\Alp$, and using again the rectified choice $\id_*=\id$, this gives precisely the original representative $[P,f]$. Hence the composite is the identity also on morphisms.

We leave the agreement of compatibility isomorphisms to the reader.
\end{proof}

\begin{Rem}
\label{Rem:Alp-almost!}%
\Cref{Prop:Alp-unit-counit} tells us that the unit-counit relations hold strictly. The reader might then ask why $U\adj \Alp[(-)]$ is not a strict 2-adjunction.
The reason is that although $\etalp_{\cA}$ is strictly 2-natural in~$\cA$, the counit~$\epslp_{\cM}$ is only pseudo-natural in~$\cM$.
Indeed, for a morphism $t\colon \cM\to \cM'$ of Mackey 2-functors, we have an invertible modification
\[
\epslp_{t}\colon t\circ \epslp_{\cM}\isoTo \epslp_{\cM'}\circ \Alpt\colon \Alp[(U\cM)]\to \cM'
\]
(writing~$\Alpt$ for~$\Alp[U(t)]$), given for every $G\in\GG$ by the natural transformation
\[
\epslp_{t,G}\colon t_G\circ \epslp_{\cM,G}\isoTo \epslp_{\cM',G}\circ (\Alpt)_{G}\colon \Alp[(U\cM)]\to \cM'
\]
which on an object~$(H,x)\in\Alp[(U\cM)]$ is the induction-compatibility isomorphism of~$t$ with respect to~$\stH\colon H\into G$, evaluated at~$x\in\cM(H)$:
\[
(\epslp_{t,G})_{(H,x)}\colon
t_G({\stH}_*(x))
\xto{((t_{\stH})_*)_x}
{\stH}_*(t_H(x))\,.
\]
Further verifications are left to the reader.
\end{Rem}

\begin{Thm}
\label{Thm:Alp-biadjunction}%
We have a biadjunction between the 2-category of (rectified) Mackey 2-functors and that of restriction 2-functors on~$\GG$
\[
\xymatrix@M=.5em{
\mackfun{\GG} \ar@{->}@/^1em/[d]^-{U=\forget} \ar@{}[d]|-{\adj}
\\
\resfun{\GG} \ar@/^1em/[u]^-{\Alp[(-)]}
}
\]
with unit~$\etalp$ and counit~$\epslp$ as in~\Cref{Cons:Alp-unit,Cons:Alp-counit}.
\end{Thm}

\begin{proof}
By~\Cref{Thm:Alp-Mackey}, the construction
$\Alp[(-)]\colon \resfun{\GG}\to\mackfun{\GG}$
is a $2$-functor, while $U\colon\mackfun{\GG}\to\resfun{\GG}$
is the forgetful $2$-functor.
By~\Cref{Prop:Alp-unit}, for every restriction $2$-functor~$\cA$
we have a morphism $\etalp_{\cA}\colon\cA\to U(\Alp).$
These morphisms are strictly $2$-natural in~$\cA$.
By~\Cref{Prop:Alp-counit}, for every Mackey $2$-functor~$\cM$
we have a morphism of Mackey $2$-functors
$\epslp_{\cM}\colon\Alp[(U\cM)]\to\cM.$
These morphisms are pseudo-natural in~$\cM$ as explained in~\Cref{Rem:Alp-almost!}.
Triangle equalities hold strictly by \Cref{Prop:Alp-unit-counit}, \ie the corresponding invertible modifications of \Cref{Rec:biadjunction} may be chosen to be identity
modifications. With this choice, the coherence conditions for the triangle modifications are automatic.
\end{proof}

We can now clarify \Cref{Thm:Alp-UP-intro} of the Introduction:
\begin{Thm}
\label{Thm:Alp-UP}%
Let $\cA\colon \GG^{\op}\to \ADD$ be a restriction 2-functor and $\cM$ be a Mackey 2-functor.
We have an equivalence of categories between the category of morphisms of restriction 2-functors $t\colon \cA\to \cM$ (and modifications) and the category of morphisms of Mackey 2-functors $s\colon \Alp\to \cM$ (and modifications), given by
\[
t\mapsto \hat{t}=\epslp_{\cM}\circ \Alpt
\]
with inverse equivalence given by
\[
s\mapsto U(s)\circ\etalp_{\cA}.
\]
\end{Thm}
\begin{proof}
This is a direct consequence of \Cref{Thm:Alp-biadjunction}. See \Cref{Rec:biadjunction}.
\end{proof}

\begin{Rem}
\label{Rem:etalp-split}%
Let~$\cA\colon \GG^\op\to \ADD$ be a restriction 2-functor. The unit morphism $\etalp_{\cA}\colon \cA\to \Alp$ of~\Cref{Cons:Alp-unit} is `split injective' in the following sense. We can define an additive $\sigma_{\cA}\colon \Alp\to \cA$ characterized by the property that for every $G$ \emph{indecomposable} (connected) and for every object $(H,x)\in\Alp(G)$ with $H$ \emph{indecomposable} as well, we have in~$\cA(G)$
\[
\sigma_{\cA,G}(H,x)=\left\{\begin{array}{cl}
(\stH^*)\inv(x)
&\textrm{if $\stH\colon H\isoto G$ is an equivalence}
\\
0&\textrm{otherwise}.
\end{array}\right.
\]

More precisely, for every~$G$ and every $(H,x)\in\Alp(G)$, with structure morphism $\stH\colon H\into G$ and~$x\in\cA(H)$, we consider $H^\approx:=H^{\stH\approx}$ the union of the components of~$H$ on which~$\stH$ is full, as in~\Cref{Def:approx} for~$p=\stH$. Then we define
\begin{equation}
\label{eq:sigma}%
\sigma_{\cA,G}(H,x)=(\stH^\approx)_*(x\restr{H^\approx})
\end{equation}
in~$\cA(G)$, where $x\restr{H^\approx}=\incl_{H^\approx}^*(x)\in\cA(H^\approx)$ for the inclusion~$H^\approx\hook H$, and where~$(\stH^\approx)_*\colon \cA(H^\approx)\to \cA(G)$ is the folding induction in~$\cA$ of~\Cref{Prop:leq-adjoints} for the local equivalence~$s=\stH^\approx$.

We define $\sigma_{\cA,G}$ on morphisms similarly. Let $[P,f]\colon (H,x)\to (K,y)$
be represented by a span $P\in\SpanfG(H,K)$ and a morphism $f\colon p_1^*(x)\to p_2^*(y)$
in~$\cA(P)$. Set $P^\approx:=P^{\stP\approx}$, where $\stP\colon P\into G$ is the structure morphism of~$P$ over~$G$ and $f^\approx:=f\restr{P^\approx}$. We can then define~$\sigma_{\cA,G}([P,f])$ as $(\stP^\approx)_*(f^\approx)$ suitably `adjusted'.
We leave the details to the reader. (Note that checking that this construction is well-defined with respect to~$\approx$-equivalence requires the compatibility of traces with base-change established in~\Cref{Lem:trace-base-change}.)

For instance, $\sigma_{\cA,G}(G,x)=x$ for every $x\in\cA(G)$ from which it follows that
\[
\sigma_{\cA}\circ \etalp_{\cA}=\id_{\cA}.
\]
The reader can verify that $\sigma_{\cA}\colon \Alp\to \cA$ is a morphism of conjugation 2-functors (only). It need not to be compatible with restrictions.
\end{Rem}

%
\section{Construction and biadjunction for~{$\Bup[(-)]$}}
\label{sec:Bup}%
%

Recall our 2-category $\GG$ of groupoids of interest (\Cref{Conv:GG}) and the 2-subcategory $\GGi$ with only local equivalences as morphisms (\Cref{Def:GGi}).

In this section, we construct $\Bup[(-)]\colon \conjfun{\GG}\to \resfun{\GG}$ and show that it is a right biadjoint to the forgetful functor. This will be substantially simpler than what we did for~$\Alp[(-)]$ in~\Cref{sec:Alp,sec:Alp-Thm} because we do not need to discuss induction. The latter will appear in~\Cref{sec:Bup-Mackey}, when we establish that $\Bup$ is actually always a Mackey 2-functor.

Let us temporarily \emph{fix a conjugation 2-functor}, \ie an additive 2-functor
\[
\cB\colon \GGi^\op\to \ADD.
\]
Thus~$\cB$ has only restriction along local equivalences.
Associated to~$\cB$, we construct for every $G$ in~$\GG$ an additive category $\Bup(G)$ by adjoining restriction data essentially via a Grothendieck construction.
For comma categories~$\GG/G$ see~\Cref{Def:GG/G}.

\begin{Cons}
\label{Cons:Bup-category}%
Let $G\in\GG$. We construct a category $\Bup(G)$ as follows.

\smallskip
\noindent
\textbf{Objects.}
An object $x=x_\sbull$ in~$\Bup(G)$ is the data of an object~$x_H$ in the category~$\cB(H)$ for every groupoid $H\in\GG/G$ over~$G$ together with an isomorphism
\[
x_{s}\colon x_H\isoto s^*(x_K)
\]
in~$\cB(H)$ for every morphism~$s\colon H \to K$ in~$\GG/G$
for which the underlying morphism $s\colon H\to K$ is a \emph{local equivalence} in~$\GG$ (\Cref{Def:loc-equiv}), hence $s\in\GGi$ and the functor $s^*=\cB(s)\colon \cB(K)\to \cB(H)$ makes sense. We refer to $x_H$ as the value of~$x_\sbull$ at \emph{level}~$H$. We refer to $x_s$ as the `coherence isomorphisms' of~$x_\sbull$.
These coherence isomorphisms are subject to three conditions:
\begin{enumerate}[\rm(1)]
\item
\label{it:x-trivial}%
For all~$H_1,H_2$, under the additivity equivalence $(\incl_{H_1}^*,\incl_{H_2}^*)\colon \cB(H_1\sqcup H_2)\isoto \cB(H_1)\oplus\cB(H_2),$ the object~$x_{H_1\sqcup H_2}$ is identified with $(x_{H_1},x_{H_2})$ via the pair of coherence isomorphisms
\[
(x_{\incl_{H_1}},x_{\incl_{H_2}})\colon (x_{H_1},x_{H_2}) \isoto \big( \incl_{H_1}^*(x_{H_1\sqcup H_2}), \incl_{H_2}^*(x_{H_1\sqcup H_2}) \big).
\]
For the identity morphism, we require $x_{\id_H}=\id_{x_H}$ for every~$H$.

\smallbreak
\item
\label{it:x-ts}%
For every two composable morphisms $s\colon H\to K$ and~$t\colon K\to L$ in~$\GG/G$, with~$s$ and~$t$ local equivalences in~$\GG$, the following triangle commutes in~$\cB(H)$
\[
\xymatrix{
x_H \ar[rr]^-{ x_{t\circ s}}_-{\simeq} \ar[rd]_-{ x_{s}}^-{\simeq}
&& s^*t^*(x_L).\!\!
\\
& s^*(x_K) \ar[ru]_-{s^*( x_{t})}^-{\simeq}
}
\]
\smallbreak
\item
\label{it:x-alpha}%
For every 2-cell $\alpha\colon s\isoTo t\colon H\to K$ in~$\GG/G$ for which $s,t\colon H\to K$ are local equivalences, the two isomorphisms $x_s\colon x_H\isoto s^*(x_K)$ and~$x_t\colon x_H\isoto t^*(x_K)$ are compatible with the natural isomorphism $\alpha^*=\cB(\alpha)\colon s^*\isoTo t^*\colon \cB(K)\to \cB(H)$ provided by~$\cB$, namely we require $\alpha^*_{x_{K}}\circ x_s= x_t$ in~$\cB(H)$.
\end{enumerate}

\smallskip
\noindent
\textbf{Morphisms.}
A morphism~$f=f_\sbull$ from~$x=x_\sbull$ to~$y=y_\sbull$ is a collection of morphisms $f_H\colon x_H\to y_H$ in~$\cB(H)$ for every~$H\in\GG/G$, compatible with the isomorphisms~$x_s$ and~$y_s$, in the sense that for every $s\colon H\to K$ in~$\GG/G$ whose underlying morphism is a local equivalence, the following square commutes in~$\cB(H)$:
\begin{equation}
\label{eq:Bup-morphisms}%
\vcenter{\xymatrix@R=2em{
x_H \ar[r]_-{\simeq}^-{ x_{s}} \ar[d]_-{f_H}
& s^*(x_K) \ar[d]^-{s^*(f_K)}
\\
y_H \ar[r]_-{\simeq}^-{ y_{s}}
& s^*(y_K).
}}
\end{equation}

\smallskip
\noindent
\textbf{Composition}
is levelwise $(g_\sbull)\circ(f_\sbull)=(g_\sbull\circ f_\sbull)$ with obvious identity~$(\id_{\sbull})$.
\end{Cons}
\begin{Rem}
When we write~$x_H$ the index~$H$ really refers to the structure morphism~$\stH\colon H\to G$ since $H$ is short for~$(H\xto{\stH}G)\in\GG/G$. Note also that, unlike what we did in~\Cref{sec:Alp}, we are now allowing arbitrary morphisms~$H\to G$ not only faithful ones (in case~$\GG$ has any non-faithful morphism, of course).
\end{Rem}

The following is much simpler than the counterpart~\Cref{Prop:Alp-category}:
\begin{Prop}\label{Prop:Bup-category}
For every groupoid $G\in\GG$, the above~\Cref{Cons:Bup-category} defines an additive category $\Bup(G)$.
\end{Prop}

\begin{proof}
Addition of morphisms and objects is performed levelwise: $(x_\sbull\oplus y_\sbull)_H=x_H\oplus y_H$ and~$(x_\sbull\oplus y_\sbull)_s=\smat{x_{s}&0\\0& y_{s}}$.
\end{proof}

We now turn to the variance of~$\Bup(G)$ in~$G$ with respect to restriction.
\begin{Cons}
\label{Cons:Bup-restriction}%
Let $j\colon G'\to G$ be a morphism in~$\GG$. \emph{Restriction} along~$j$
\[
j^*\colon \Bup(G)\to \Bup(G')
\]
is given by levelwise restriction along~$j_!\colon \GG/G'\to \GG/G$, that is, the objects are $(j^*x)_{H'}=x_{j_!H'}$ (essentially $(j^*x)_{H'}$ is just~$x_{H'}$ remembering that the structure morphism changes $\st{j_!H'}=j\circ \st{H'}$) and the isomorphisms are $(j^*x)_{s'}=x_{j_!s'}$. Similarly, $j^*$ is defined on morphisms by applying~$j_!$ levelwise, that is, $(j^*f)_{H'}=f_{j_!H'}$.
This defines an additive functor $j^*\colon \Bup(G)\to \Bup(G')$.
\end{Cons}

The 2-functoriality is also defined by applying $\alpha_!$ levelwise. In detail:
\begin{Cons}
\label{Cons:Bup-alpha*}%
Let $j,k\colon G'\to G$ be morphisms in~$\GG$ and $\alpha\colon j \isoTo k$ be a $2$-cell. We define the natural transformation
$\alpha^*\colon j^* \isoTo k^*\colon \Bup(G)\to \Bup(G')$ by
\[
((\alpha^*)_{x_\sbull})_{H'}=
x_{((\alpha_!)_{H'})}
\]
for every~$x_\sbull\in\Bup(G)$ and every~$H'\in\GG/G'$, where $\alpha_!\colon j_!\isoTo k_!\colon \GG/G'\to \GG/G$ is as in~\Cref{Cons:(j/-)=>(k/-)}. The isomorphism~$(\alpha_!)_{H'}\colon j_!H'=(H',j\stHp)\isoTo k_!H'=(H',k\stH')$ in~$\GG/G$ is given by $\id_{H'}$ on the object and by~$\alpha\st{H'}$ as structure 2-cell.
As part of the data of~$x_\sbull$ we have an isomorphism $x_s$ in~$\cB(H')$ for $s=(\alpha_!)_{H'}$
\[
x_{((\alpha_!)_{H'})}\colon
x_{j_!H'}
\isoto
\id_{H'}^*(x_{k_!H'})
=
x_{k_!H'}.
\]
The left-hand $x_{j_!H'}$ is~$j^*(x_\sbull)$ at~$H'$ and the right-hand $x_{k_!H'}$ is~$k^*(x_\sbull)$ at~$H'$.
\end{Cons}

\begin{Prop}
\label{Prop:Bup-2-functor}%
The above~\Cref{Cons:Bup-category,Cons:Bup-restriction,Cons:Bup-alpha*} define a 2-functor $\Bup\colon \GG^\op\to \ADD$, which is additive. In other words, $\Bup$ is a restriction 2-functor on~$\GG$ (\Cref{Def:restriction-2-functor}).
\end{Prop}
\begin{proof}
Verifications are left to the reader. Additivity follows from $\GG/(G_1\sqcup G_2)\cong (\GG/G_1)\times (\GG/G_2)$ and Condition~\eqref{it:x-trivial} in~\Cref{Cons:Bup-category}.
\end{proof}

We now allow~$\cB$ to vary in~$\conjfun{\GG}$.
\begin{Cons}
\label{Cons:Bup-transfos}%
Let $t\colon \cB\to \cB'$ be a morphism of conjugation 2-functors.
We define a strict morphism of restriction 2-functors $\Bup[t]\colon \Bup\to \Bupp$ by applying $t$ levelwise.
More precisely, for every $G\in\GG$ we define~$\Bupt_G\colon \Bup(G)\to \Bupp(G)$ on objects $x_\sbull\in\Bup(G)$ by
\[
(\Bupt_G(x))_H=t_H(x_H)
\]
at level~$H\in\GG/G$ and on local equivalences $s\colon H\apprto K$ in~$\GG/G$ by
\[
(\Bupt_G(x))_s
:=
(t_s)^{-1}_{x_K}\circ t_H(x_s)
\colon
t_H(x_H)\isoto s^*t_K(x_K),
\]
essentially just $t_H(x_s)$, but adjusted by the restriction-compatibility isomorphism $t_s\colon s^*\circ t_K\isoto t_H\circ s^*$ provided with~$t$.
For every morphism $f\colon x_\sbull\to y_\sbull$ define $(\Bupt_G(f))_H=t_H(f_H)$.
These functors are strictly compatible with restriction:
For every $j\colon G'\to G$ both functors $j^*\circ \Bupt_G$ and~$\Bupt_{G'}\circ j^*$ from~$\Bup(G)$ to~$\Bup(G')$ send every object~$x_\sbull\in\Bup(G)$ to $t_{H'}(x_{j\stHp})$ at every level~$H'\in\GG/G'$.
For modifications $m\colon t\To t'\colon \cB\to \cB'$, we also apply the levelwise formula, namely $\Bup[m]\colon \Bup[t]\To \Bup[t']\colon \Bup\to \Bupp$ is given for every $G\in\GG$ by the natural transformation $\Bup[m]_G\colon \Bup[t]_G\To \Bup[t']_G\colon \Bup(G)\to \Bupp(G)$ which at every $x_\sbull\in\Bup(G)$ is the morphism $\Bup[m]_{G,x_\sbull}\colon \Bup[t]_G(x_\sbull)\To \Bup[t']_G(x_\sbull)$ in~$\Bupp(G)$ which at level~$H\in\GG/G$ is
\[
(\Bup[m]_{G,x_\sbull})_{H}=m_{H,x_H}\colon t_H(x_H)\to t'_H(x_H).
\]
This assembles into a well-defined 2-functor
\[
\Bup[(-)]\colon \conjfun{\GG}\to \resfun{\GG}.
\]
We shall strengthen this statement in~\Cref{Thm:Bup-Mackey} where we show that $\Bup[(-)]$ defines a functor from~$\conjfun{\GG}$ to~$\mackfun{\GG}$.
\end{Cons}

Let us now turn to the 2-adjunction with the forgetful functor.

\begin{Not}
\label{Not:forget-V}%
Let $V=\forget\colon \resfun{\GG}\hook \conjfun{\GG}$ be the 2-functor induced by the inclusion~$\GGi\hook \GG$. It is faithful (on 2-cells) since $\GGi$ and~$\GG$ have the same objects.
\end{Not}

\begin{Cons}
\label{Cons:Bup-unit}%
Let $\cA\in\resfun{\GG}$ be a restriction 2-functor.
We define a (strict) morphism of $2$-functors
\[
\etaup_{\cA}\colon \cA\to \Bup[(V\!\cA)]
\]
as follows. Let $G\in \GG$. For an object $x\in\cA(G)$, define $\etaup_{\cA,G}(x)_\sbull$ in~$\Bup[(V\!\cA)](G)$ at every level $H\in\GG/G$ by
\[
(\etaup_{\cA,G}(x))_H=\stH^*(x)
\]
where $\stH\colon H\to G$ is the structure morphism of~$H$,
and for every local equivalence $s\colon H\to K$ in~$\GG/G$, define $\etaup_{\cA,G}(x)_s$ to be $\st{s}^*\colon \st{H}^*\isoto s^*\st{K}^*$ evaluated at~$x$, for~$\st{s}\colon \stH\isoTo \stK s$ the structure 2-cell of the morphism~$s$.
In colloquial terms, we create out of an object in~$\cA(G)$ all the levelwise objects~$x_H$ by using the restrictions that exist since~$\cA$ is an actual restriction 2-functor.
The functor $\etaup_{\cA,G}$ is defined similarly on morphisms: $(\etaup_{\cA,G}(f))_H=\stH^*(f)$.
\end{Cons}

\begin{Prop}
\label{Prop:Bup-unit}
The above~\Cref{Cons:Bup-unit} yields a strictly 2-natural transformation of restriction $2$-functors
\[
\etaup_{\cA}\colon \cA\to \Bup[(V\!\cA)].
\]
\end{Prop}
\begin{proof}
For every $j\colon G'\to G$ in~$\GG$, the image in~$\Bup[(V\cA)](G')$ of an object~$x\in\cA(G)$ by both functors~$j^*\circ\etaup_{\cA,G}$ and~$\etaup_{\cA,G'}\circ j^*$ is equal at level~$H'\in\GG/G'$ to the object $(\st{j_!H'})^*(x)=(j\stHp)^*(x)=\stHp^*(j^*(x))$ in~$\cA(H')$. The result follows easily.
\end{proof}

\begin{Cons}
\label{Cons:Bup-counit}%
Let $\cB$ be a conjugation $2$-functor. We define a morphism of conjugation 2-functors~$\epsup_{\cB}\colon V(\Bup)\to \cB$
as follows. Let $G\in\GG$.
For an object $x_\sbull\in V(\Bup)(G)$, define in~$\cB(G)$
\[
\epsup_{\cB,G}(x_\sbull):=x_G.
\]
On a morphism $f_\sbull\colon x_\sbull\to y_\sbull$ in $V(\Bup)(G)$, we define similarly
\[
\epsup_{\cB,G}(f)=f_G.
\]
In telegraphic style, $\epsup_{\cB,G}$ is evaluation at the object~$G$ of~$\GG/G$ (with $\st{G}=\id_G$).
\end{Cons}
\begin{Prop}
\label{Prop:Bup-counit}
The above~\Cref{Cons:Bup-counit} defines a morphism of conjugation $2$-functors
\[
\epsup_{\cB}\colon V(\Bup)\to \cB
\]
where for every $s\colon G'\xto{\approx} G$ in~$\GGi$ the isomorphism
$\epsup_{\cB,s}\colon s^*\circ \epsup_{\cB,G}\isoTo \epsup_{\cB,G'}\circ s^*$
of functors $V(\Bup)(G) \to\cB(G')$ is provided on every object~$x_\sbull$ by the inverse of the coherence isomorphism $x_s$.
\end{Prop}
\begin{proof}
With the above notation, $x_s\colon x_{G'}\isoto s^*(x_G)$ is an isomorphism in~$\cB(G')$ and its inverse indeed goes
from~$s^*\circ \epsup_{\cB,G}(x_\sbull)=s^*(x_G)$ to~$\epsup_{\cB,G'}\circ s^*(x_\sbull)=(s^*x_\sbull)_{G'}=x_{s_!G'}=x_{G'}$. Further compatibilities are left to the reader.
\end{proof}

\begin{Prop}
\label{Prop:Bup-unit-counit}%
For every conjugation 2-functor $\cB\in\conjfun{\GG}$ and every restriction 2-functor $\cA\in\resfun{\GG}$, we have equalities of pseudo-natural transformations
\[
\Bup[(\epsup_{\cB})]\circ \etaup_{\Bup}= \Id_{\Bup}
\qquadtext{and}
\epsup_{V\!\cA}\circ V(\etaup_{\cA})= \Id_{V\!\cA}.
\]
\end{Prop}
\begin{proof}
Let $G\in\GG$. The functor $\etaup_{\Bup,G}\colon \Bup(G)\to\Bup[(V\Bup)](G)$ maps an object~$x_\sbull$ to~$\big(\stH^{*,\Bup}(x_\sbull)\big)_{H\in\GG/G}$ where $\stH^{*,\Bup}\colon \Bup(G)\to \Bup(H)$ is restriction in~$\Bup$ with respect to~$\stH\colon H\to G$ as in~\Cref{Cons:Bup-restriction}. The functor $\Bup[(\epsup_{\cB})]_G\colon \Bup[(V\Bup)](G)\to\Bup(G)$ applies $\epsup_{\cB}$ in each level~$H$, that is, evaluation at~$\id_H\in\GG/H$; hence it sends our $\big(\stH^{*,\Bup}(x_\sbull)\big)_{H\in\GG/G}$ to
\[
\Big(\big(\stH^{*,\Bup}(x_\sbull)\big)_H\Big)_{H\in\GG/G}=
\big((x_\sbull)_{\stH!H}\big)_{H\in\GG/G}=
(x_H)_{H\in\GG/G}=x_\sbull
\]
since ${\stH}_!H=(H\xto{\stH\circ\id_H}G)$ is just~$H\in\GG/G$. We leave the $x_s$ to the reader.
This gives the first equality.

For the second one, the functor $V(\etaup_{\cA})_G\colon \cA(G)\to \Aup(G)$ sends an object~$x\in\cA(G)$ to $(\stH^*x)_H$. The functor
$\epsup_{V\!\cA,G}\colon \Aup(G)\to \cA(G)$ evaluates at~$G$ which maps~$(\stH^*x)_H$ to~$\id_G^*x=x$. This yields the second equality.

We leave the agreement of compatibility isomorphisms to the reader.
\end{proof}

\begin{Rem} \label{Rem:Bup-almost!}%
We have a situation dual to that of \Cref{Rem:Alp-almost!}: The counit~$\epsup_{\cB}$ of \Cref{Cons:Bup-counit} is strictly 2-natural in~$\cB$ but the unit~$\etaup_{\cA}\colon\cA\to\Bup[(V\!\cA)]$ of \Cref{Cons:Bup-unit} is only pseudo-natural in~$\cA$.
Indeed, let $t\colon\cA\to\cA'$ be a morphism of restriction $2$-functors.
There is an invertible modification
\[
\etaup_t\colon \Bup[V(t)]\circ\etaup_{\cA} \isoTo \etaup_{\cA'}\circ t.
\]
As usual, we leave further details to the reader.
\end{Rem}
\begin{Thm}
\label{Thm:Bup-biadjunction}%
We have a biadjunction between the 2-category of restriction 2-functors and that of conjugation 2-functors on~$\GG$
\[
\xymatrix@M=.5em{
\resfun{\GG} \ar@{->}@/_1em/[d]_-{V=\forget} \ar@{}[d]|-{\adj}
\\
\conjfun{\GG} \ar@/_1em/[u]_-{\Bup[(-)]}
}
\]
with unit~$\etaup$ and counit~$\epsup$ as in~\Cref{Cons:Bup-unit,Cons:Bup-counit}.
\end{Thm}

\begin{proof} By~\Cref{Cons:Bup-transfos}, our $\Bup[(-)]\colon\conjfun{\GG}\to\resfun{\GG}$ is a $2$-functor, while $V\colon\resfun{\GG}\to\conjfun{\GG}$ is the forgetful $2$-functor.
By~\Cref{Prop:Bup-unit}, for every restriction $2$-functor~$\cA$ we have a morphism $\etaup_{\cA}\colon \cA\to\Bup[(V\!\cA)]$.
These morphisms are pseudo-natural in~$\cA$ as explained in~\Cref{Rem:Bup-almost!}.
By~\Cref{Prop:Bup-counit}, for every conjugation $2$-functor~$\cB$ we have a morphism $\epsup_{\cB}\colon V(\Bup)\to\cB$.
These morphisms are strictly $2$-natural in~$\cB$.
The triangle equalities hold by~\Cref{Prop:Bup-unit-counit}.
The corresponding invertible triangle modifications may be chosen to be identity modifications.
With this choice, their coherence conditions are automatic.
\end{proof}

We can now clarify \Cref{Thm:Bup-UP-intro} of the Introduction:
\begin{Thm}
\label{Thm:Bup-UP}%
Let $\cB\colon \GGi^{\op}\to \ADD$ be a conjugation 2-functor and let $\cA$ be a restriction 2-functor.
We have an equivalence of categories between the category of morphisms of conjugation 2-functors $t\colon \cA\to \cB$
(and modifications) and the category of morphisms of restriction 2-functors $s\colon \cA\to\Bup$
(and modifications), given by
\[
t\mapsto \tilde{t}:=\Bupt\circ \etaup_{\cA}
\]
with inverse equivalence given by
\[
s\mapsto \epsup_{\cB} \circ V(s).
\]
\end{Thm}
\begin{proof}
This is a direct consequence of \Cref{Thm:Bup-biadjunction}. See \Cref{Rec:biadjunction}.
\end{proof}

\begin{Rem}
\label{Rem:Bup-not-THAT-adjoint}%
In~\Cref{Thm:Bup-UP}, even if~$\cB$ is a restriction 2-functor, the transformation~$\epsup_{\cB}\colon\Bup\to\cB$
need not commute with restrictions and, even if~$\cA$ is Mackey, the transformation
$\tilde{t}\colon \cA\to \Bup$ need not preserve inductions.
\end{Rem}

%
\section{The 2-functor $\Bup$ is Mackey}
\label{sec:Bup-Mackey}%
%

In this section, $\cB$ is a conjugation 2-functor on~$\GG$, that is, an additive 2-functor
\[
\cB\colon \GGi^\op\to\ADD.
\]
We want to explain~\Cref{Thm:Bup-Mackey-intro}, \ie why the restriction 2-functor $\Bup$ of~\Cref{sec:Bup} is a Mackey 2-functor.
We recall from~\Cref{Prop:leq-adjoints} that any conjugation 2-functor already has induction along local equivalences.
The induction construction for $\Bup$ will involve morphisms
$p\colon X\to Y$ which are not local equivalences, so ordinary induction~$p_*$ is not
available on all connected components of~$X$.
We shall induce only from the components of~$X$ on which $p$ is a local equivalence, that is the $\approx$-locus~$X^{p\approx}$ of~\Cref{Def:approx}. Let us do a preparation in this direction.

\begin{Cons}
\label{Cons:j-approx}%
Let $j\colon G'\into G$ in~$\GG$ be a \emph{faithful} morphism in~$\GG$. (See~\Cref{Rem:approx-faithful}.)
For every $H\in \GG/G$ over~$G$, we define $j^\approx H:=(j^*H)^{(\pr_2){\approx}}$ as the part of the pullback~$j^*H=G'\times_G H$ on which $\pr_2\colon G'\times_G H\to H$ is a local equivalence (\Cref{Def:approx}). By construction, it comes with a local equivalence that we denote~$\najH$ (in reference to~\Cref{Exa:folding})
\[
\najH:=(\pr_2)^\approx\colon j^\approx H\to H.
\]
Since $j^\approx H$ is a subgroupoid of~$j^*H\in\GG/G'$, it is also an object over~$G'$ with the restricted structure morphism, that is, $\st{j^\approx H}=(\pr_1)\restr{j^\approx H}\colon j^\approx H\to G'$.

Let now $s\colon H\apprto K$ be a morphism in~$\GG/G$ with $s$ a local equivalence in~$\GG$.
By~\Cref{Lem:pb-Mackey}, the pull-back $j^*s\colon j^*H\to j^*K$ is a local equivalence such that $\pr_2\,j^*s=s\,\pr_2$. Thus, by~\Cref{Lem:approx-Mackey}, the morphism $j^*s$ restricts to a local equivalence $(j^*H)^{\pr_2\approx}\to (j^*K)^{\pr_2\approx}$, that we denote~$j^\approx s\colon j^\approx H\to j^\approx K$. We also record from the same~\Cref{Lem:approx-Mackey} that the following commutative square is Mackey:
\begin{equation}
\label{eq:aux-approx}%
\vcenter{\xymatrix@C=1em@R=1em{
& j^\approx H \ar[rd]^-{\najH} \ar[ld]_-{j^\approx s} \ar@{}[dd]|{=}
\\
j^\approx K \ar[rd]_-{\najK}
&& H \ar[ld]^-{s}
\\
& K.
}}
\end{equation}
\end{Cons}
\begin{Exa}
For~$H$ connected, $j^\approx H$ is essentially a coproduct of copies of~$H$ and $\najH$ is a folding. See~\Cref{Lem:leq-folding}.
Conceptually, $j^\approx H$ picks up those components of the pullback~$j^*H$ which are equivalent to~$H$. If we think of $G'\into G$ as an inclusion of a subgroup, we are picking the $G$-conjugates of~$H$ inside~$G'$.
\end{Exa}

\begin{Rem}
\label{Rem:najH}%
Summarizing~\Cref{Cons:j-approx}, we have the following diagram in~$\GG$:
\begin{equation}
\vcenter{\xymatrix@C=3em@R=2em{
& j^\approx H \ar@{_(->}[d]|-{\incl} \ar@/^2em/[rdd]^-{\najH=\pr_2\restr{j^\approx H}}_(.6){\appr\!} \ar@/_2em/[ldd]_-{\st{j^\approx H}=\pr_1\restr{j^\approx H}} \ar@{}[rdd]|(.25){=} \ar@{}[ldd]|(.25){=}
\\
& j^*H =G'\times_G H \ar[ld]|-{\st{j^*H}=\pr_1} \ar[rd]_-{\pr_2} \ar@{}[dd]|-{\isocell{\gamma_{j/\stH}}}
&
\\
G' \ar@{ >->}[rd]_-{j}
&& H \ar[ld]^-{\stH} &
\\
& G.\!\! &
}}
\end{equation}
In particular, $\najH$ can be viewed as a morphism over~$G$, that is a morphism in~$\GG/G$ between $j_!j^\approx H$ (with structure morphism~$j\,\pr_1\restr{j^\approx H}$) and $H$. The structure 2-cell of~$\najH$ is the restriction of the 2-cell of the isocomma:~$\st{\najH}=(\gamma_{j/\stH})\restr{j^\approx H}$.
\end{Rem}
\begin{Cons}
\label{Cons:Bup-induction}%
Let $j\colon G'\into G$ be a faithful morphism in~$\GG$.
We define a functor
\[
j_*\colon \Bup(G')\to \Bup(G).
\]
Let $x'_\sbull$ be an object in~$\Bup(G')$. For every $H\in\GG/G$ we define in~$\cB(H)$ the object
\[
(j_*x'_\sbull)_H:=(\najH)_*(x'_{j^\approx H})
\]
using~\Cref{Cons:j-approx}, where $x'_{j^\approx H}\in\cB(j^\approx H)$ is provided by~$x'_\sbull$ since~$j^\approx H$ lives over~$G'$ and where $(\najH)_*$ is induction (\Cref{Prop:leq-adjoints}) along the local equivalence~$\najH$.
Similarly for every $s\colon H\to K$ in~$\GG/G$ such that $s$ is a local equivalence in~$\GG$, we define the coherence isomorphism $(j_*x'_\sbull)_{s}$ associated to~$s$ to be the push-forward $(\najH)_*(x'_{j^\approx s})$ of the one for~$x'$, suitably corrected by base change, namely
\[
\xymatrix@C=1em@R=1em{
(j_*x'_\sbull)_H \ar@{=}[d] \ar[rrrrr]^-{\displaystyle(j_*x'_{\sbull})_{s}}_-{:=}
&&&&& s^*\big((j_*x'_\sbull)_K\big) \ar@{=}[d]
\\
(\najH)_*(x'_{j^\approx H}) \ar[rrrr]^-{\displaystyle(\najH)_*(x'_{j^\approx s})}
&&&& (\najH)_*(j^\approx s)^*(x'_{j^\approx K}) \ar@{}[r]|-{\cong}
& s^*(\najK)_*(x'_{j^\approx K})
}
\]
where the base-change isomorphism $(\najH)_*(j^\approx s)^*\cong s^*(\najK)_*$ in the second row comes from~\Cref{Prop:leq-base-change}\,\eqref{it:leq-bc-2} on the Mackey square~\eqref{eq:aux-approx}.

We define~$j_*$ on morphisms $f'_\sbull\colon x'_\sbull\to y'_\sbull$ in~$\Bup(G')$ via the similar
\[
(j_*f'_\sbull)_H:=(\najH)_*(f'_{j^\approx H}).
\]

All coherence conditions follow from functoriality of the isocomma construction and the compatibility of the base-change isomorphisms.
They are left to the reader.
\end{Cons}

\begin{Prop}\label{Prop:Bup-induction-functor}
For every faithful $j\colon G'\into G$, the above~\Cref{Cons:Bup-induction} defines an additive functor~$j_*\colon \Bup(G')\to \Bup(G)$.
\end{Prop}

\begin{proof}
Additivity follows from the additivity of $(\najH)_*$.
\end{proof}

\begin{Prop}\label{Prop:Bup-adjunction}
For every faithful $j\colon G'\into G$, the functor $j_*\colon \Bup(G')\to \Bup(G)$ is a special Frobenius two-sided adjoint of the restriction~$j^*\colon \Bup(G)\to \Bup(G')$.
\end{Prop}

\begin{proof}
We need the four (co)units of~\eqref{eq:4-units-j} and for that we need to compose the functors $j^*$ and~$j_*$ of~\Cref{Cons:Bup-restriction,Cons:Bup-induction}.
Let $x=x_\sbull\in\Bup(G)$ then $j_*j^*(x)\in\cB(H)$ at level $H\in \GG/G$ is equal to the following object of~$\cB(H)$
\begin{equation}
\label{eq:aux-j_*j^*}%
(j_*j^*x)_{H}=(\najH)_*\big((j^*x)_{j^\approx H}\big)= (\najH)_*\big(x_{j_!j^\approx H}\big)\cong (\najH)_*(\najH)^*(x_H)
\end{equation}
where the final isomorphism is the coherence isomorphism for~$x_\sbull$ at the local equivalence~$\najH\colon j_!j^\approx H\apprto H$ in~$\GG/G$, as explained in~\Cref{Rem:najH}.

The term $\nabla_*\nabla^*$ appearing in~\eqref{eq:aux-j_*j^*}
suggests using the special Frobenius two-sided adjunction for the local
equivalence $\nabla=\najH$ (\Cref{Prop:leq-adjoints}), which is available even
though the $2$-functor~$\cB$ is only a conjugation $2$-functor.
We can make a similar $\nabla^*\nabla_*$ appear for the other composite $j^*j_*$ but it is more hidden.

Let $x'=x'_\sbull\in\Bup(G')$ then $j^*j_*(x')\in\cB(H')$ at level $H'\in \GG/G'$ is equal to the following object of~$\cB(H')$
\begin{equation}
\label{eq:aux-j^*j_*-temp}%
(j^*j_*x')_{H'}=(j_*x')_{j_!H'}=(\najHp)_*(x'_{j^\approx j_!H'}).
\end{equation}
(It is tempting, but \emph{false}, to invoke the local equivalence $\najHp\colon j^\approx j_!H'\apprto j_!H'$ to replace the final object as we did after~\eqref{eq:aux-j_*j^*}. Unfortunately, this $\najHp$ need not be a morphism over~$G'$, hence $x'_\sbull$ does not carry a corresponding coherence isomorphism.)
We still modify~\eqref{eq:aux-j^*j_*-temp} to make a $(\najHp)_*(\najHp)^*$ appear, as suggested after~\eqref{eq:aux-j_*j^*}.
For this, we recall $\eta_{H'}=<\stHp,\id_{H'},\id>\colon H'\to j^*j_!H'$ as in~\eqref{eq:j!j*-eta}. We know that $\eta_{H'}$ is fully faithful (\Cref{Rem:eta-eps-j!j*})  and satisfies $\pr_2\eta_{H'}=\id_{H'}$. Hence it lands in~$j^\approx j_!H'$ and defines a section $\eta_{H'}\colon H'\to j^\approx j_!H'$ of~$\najHp$ in~$\GGi$. Consequently $(\eta_{H'})^*(\najHp)^*=\id_{\cB(H')}$ and we can make the $\nabla^*$ appear in~\eqref{eq:aux-j^*j_*-temp} as follows
\begin{equation}
\label{eq:aux-j^*j_*}%
(j^*j_*x')_{H'}=\id(j^*j_*x')_{H'}=(\eta_{H'})^*(\najHp)^*(\najHp)_*(x'_{j^\approx j_!H'}).
\end{equation}

In view of~\eqref{eq:aux-j_*j^*} and~\eqref{eq:aux-j^*j_*}, we can now invoke the (co)units of~$\nabla_*\adj \nabla^*\adj \nabla_*$ (in~$\cB$!) for $\nabla\in\{\najH,\najHp\}$ to define $(\leta_{x'})_{H'},(\leps_{x})_{H},(\reta_{x})_{H}, (\reps_{x'})_{H'}$ in~$\Bup$. The left unit $\leta=\leta^{(j)}\colon \Id_{\Bup(G')}\To j^*j_*$ is given on~$x'\in\Bup(G')$ at level~$H'\in\GG/G'$ by
\begin{equation}
\label{eq:leta-Bup-x'-H'}%
\vcenter{
\xymatrix@C=8em@R=1em@M=.5em{
x'_{H'} \ar[r]^-{\displaystyle(\leta_{x'})_{H'}}_-{:=} \ar[d]_-{x'_{\eta_{H'}}}^-{\simeq}
& (j^*j_*x')_{H'} \ar@{=}[d]^-{\textrm{(\ref{eq:aux-j^*j_*})}}
\\
\eta_{H'}^*(x'_{j^\approx j_!H'}) \ar[r]^-{\displaystyle\leta^{(\najHp\textrm{ in }\cB)}}
& \eta_{H'}^*(\najHp)^*(\najHp)_*(x'_{j^\approx j_!H'})
}}
\end{equation}
where we use that $\eta_{H'}\colon H'\to j^\approx j_!H'$ is a local equivalence to invoke the corresponding coherence isomorphism for~$x'_\sbull$ on the left-hand vertical side.
In full detail, the morphism in the second row reads $\eta_{H'}^*\big((\leta^{(\najHp)})_{x'_{j^\approx j_!H'}}\big)$.
The left counit $\leps=\leps^{(j)}\colon j_*j^*\To \Id_{\Bup(G)}$ is given on~$x\in\Bup(G)$ at level~$H\in\GG/G$ by
\begin{equation}
\label{eq:leps-Bup-x-H}%
\vcenter{
\xymatrix@C=6em@R=1em@M=.5em{
(j_*j^*x)_{H} \ar[r]^-{\displaystyle(\leps_{x})_{H}}_-{:=} \ar[d]_-{\cong}^-{\textrm{(\ref{eq:aux-j_*j^*})}}
& x_{H} \ar@{=}[d]
\\
(\najH)_*(\najH)^*(x_H) \ar[r]^-{\displaystyle\leps^{(\najH\textrm{ in }\cB)}}
& x_H\,.\!\!
}}
\end{equation}
The right unit $\reta=\reta^{(j)}\colon \Id_{\Bup(G)}\To j_*j^*$ is given on~$x\in\Bup(G)$ at level~$H\in\GG/G$ by
\begin{equation}
\label{eq:reta-Bup-x-H}%
\vcenter{
\xymatrix@C=6em@R=1em@M=.5em{
x_{H} \ar[r]^-{\displaystyle(\reta_{x})_{H}}_-{:=} \ar@{=}[d]
& (j_*j^*x)_{H} \ar[d]_-{\cong}^-{\textrm{(\ref{eq:aux-j_*j^*})}}
\\
x_H \ar[r]^-{\displaystyle\reta^{(\najH\textrm{ in }\cB)}}
& (\najH)_*(\najH)^*(x_H)\,.\!\!
}}
\end{equation}
And finally the right counit $\reps=\reps^{(j)}\colon j^*j_*\To \Id_{\Bup(G')}$ is given on~$x'\in\Bup(G')$ at level~$H'\in\GG/G'$ by
\begin{equation}
\label{eq:reps-Bup-x'-H'}%
\vcenter{
\xymatrix@C=8em@R=1em@M=.5em{
(j^*j_*x')_{H'} \ar[r]^-{\displaystyle(\reps_{x'})_{H'}}_-{:=} \ar@{=}[d]^-{\textrm{(\ref{eq:aux-j^*j_*})}}
& x'_{H'} \ar[d]^-{x'_{\eta_{H'}}}_-{\simeq}
\\
\eta_{H'}^*(\najHp)^*(\najHp)_*(x'_{j^\approx j_!H'}) \ar[r]^-{\displaystyle\reps^{(\najHp\textrm{ in }\cB)}}
& \eta_{H'}^*(x'_{j^\approx j_!H'})
}}
\end{equation}
where the right-hand vertical isomorphism is the same as in~\eqref{eq:leta-Bup-x'-H'}.
Again, the morphism in the second row can be expanded to $\eta_{H'}^*\big((\reps^{(\najHp)})_{x'_{j^\approx j_!H'}}\big)$.

Note right away that juxtaposing~\eqref{eq:leta-Bup-x'-H'} and
\eqref{eq:reps-Bup-x'-H'} gives the special Frobenius property
$\reps\circ\leta=\id$, since
$(\najHp)_*\adj(\najHp)^*\adj(\najHp)_*$ is special Frobenius.

It remains to verify the unit-counit relations.
We spell out one of them:
\[
(j^*(\leps_x))\circ \leta_{j^*x}=\id_{j^*x}
\]
in~$\Bup(G')$, for every $x\in\Bup(G)$. At level~$H'\in\GG/G'$ we need to prove
\begin{equation}\label{eq:aux-unit-counit}%
(\leps_x)_{j_!H'}\circ (\leta_{j^*x})_{H'}=\id_{x_{j_!H'}}.
\end{equation}
Unpacking \eqref{eq:leta-Bup-x'-H'} on~$x'=j^*x$ and \eqref{eq:leps-Bup-x-H} at level~$H=j_!H'$ we get respectively the top row and the right column of the following diagram in~$\cB(H')$ where we abbreviate $\nabla$ for~$\najHp$
\[
\xymatrix@C=1em@R=2em{
x_{j_!H'} \ar[rr]^-{x_{j_!\eta_{H'}}}_-{\simeq} \ar@{=}[rrd]
&& \eta_{H'}^*(x_{j_!j^\approx j_!H'}) \ar[rrr]^-{\displaystyle\leta^{(\nabla)}} \ar[d]_-{\cong}^-{x_{\nabla}}
&&&
\nabla_*(x_{j_!j^\approx j_!H'}) \ar@{=}[r] \ar[d]_-{\cong}^-{x_{\nabla}}
& (j_*j^*x)_{j_!H'} \ar[d]_-{\cong}^-{x_{\nabla}}
\\
&& \kern-1em \eta_{H'}^*\nabla^*(x_{j_!H'}) \ar[rrr]^-{\eta_{H'}^*\leta^{(\nabla)}\nabla^*} \ar@{=}[rrrd]^(.6){\displaystyle(\adj)}
&&& \eta_{H'}^*\nabla^*\nabla_*\nabla^*(x_{j_!H'}) \ar@{=}[r] \ar[d]^-{\eta_{H'}^*\nabla^*(\leps^{(\nabla)})}
& \nabla_*\nabla^*(x_{j_!H'}) \ar[d]^-{\displaystyle\leps^{(\nabla)}}
\\
&&&&& \eta_{H'}^*\nabla^*(x_{j_!H'}) \ar@{=}[r]
& x_{j_!H'}
}
\]
The isomorphisms between first and second row are given by the coherence isomorphism $x_s$ for $s=\najHp$ as in~\eqref{eq:aux-j_*j^*}.
The reader will verify that this diagram commutes. The critical triangle is the one marked~$(\adj)$ which is the image of the unit-counit relation for the $\nabla_*\adj \nabla^*$ adjunction at the object~$x_{j_!H'}\in\cB(j_!H')$ under the functor~$\eta_{H'}^*\colon \cB(j_!H')\to \cB(H')$, all of which only depend on~$\cB$ being a conjugation 2-functor.
The outer commutativity of the above diagram in~$\cB(H')$ gives~\eqref{eq:aux-unit-counit}. This proves the first unit-counit relation.

The reader will verify that the other three unit-counit relations also reduce to unit-counit relations for~$\nabla_*\adj \nabla^*\adj \nabla_*$ where $\nabla=\najHp$ as above (in the proof of $(\reps j^*)\circ(j^*\reta)=\id_{j^*}$) or with $\nabla=\najH$ (for the other two relations).
\end{proof}

\begin{Prop}[Base-change for $\Bup$]
\label{Prop:Bup-base-change}%
Let
\[
\xymatrix@C=2em@R=.5em{
& L \ar@{ >->}[rd]^{j} \ar@{ ->}[ld]_{v} \ar@{}[dd]|{\isocell{\alpha}} & \\
H \ar@{ >->}[rd]_i &  & K \ar@{ ->}[ld]^u \\
& G &
}
\]
be a Mackey square in $\GG$, with $i$ and $j$ faithful. Then the left mate of~$\alpha^*$ and the right mate of~$(\alpha^{-1})^*$
\[
\alpha_!\colon j_*v^*\isoto u^*i_*
\qquad\text{and}\qquad
(\alpha\inv)_*\colon u^*i_* \isoto j_*v^*
\]
are inverse isomorphisms between functors from~$\Bup(H)$ to~$\Bup(K)$.
\end{Prop}
\begin{proof}
For $x\in\Bup(H)$, we compute the two composites $j_*v^*(x)$ and~$u^*i_*(x)$ in~$\Bup(K)$. At level~$Y\in\GG/K$ we have in~$\cB(Y)$
\begin{equation}
\label{eq:aux-Bup-BC-1}%
(j_*v^*x)_Y \equalby{\ref{Cons:Bup-induction}} (\najY)_*\big((v^*x)_{j^\approx Y}\big) \equalby{\ref{Cons:Bup-restriction}} (\najY)_*\big(x_{v_!j^\approx Y}\big)
\end{equation}
where $j^\approx Y$ is the part of $j^*Y=L\times_K Y$ on which $\pr_2\colon L\times_K Y\to Y$ is a local equivalence, the latter being called $\najY=(\pr_2)\restr{j^\approx Y}\colon j^\approx Y\apprto Y$ (see~\Cref{Cons:j-approx}). On the other hand, still in~$\cB(Y)$, we have
\begin{equation}
\label{eq:aux-Bup-BC-2}%
(u^*i_*x)_Y \equalby{\ref{Cons:Bup-restriction}} (i_*x)_{u_!Y} \equalby{\ref{Cons:Bup-induction}} (\naiuY)_*(x_{i^\approx u_!Y})
\end{equation}
where $i^\approx(u_!Y)$ is the part of $i^*(u_!Y)=H\times_G Y$ on which $\pr_2\colon H\times_G Y\to Y$ is a local equivalence, called $\naiuY=(\pr_2)\restr{i^\approx u_!Y}\colon i^\approx u_!Y\apprto Y$.
To compare~\eqref{eq:aux-Bup-BC-1} and~\eqref{eq:aux-Bup-BC-2}, we need to compare the objects~$j^\approx Y$ and $i^\approx u_!Y$ in~$\GG$.
They appear in the following commutative diagram:
\begin{equation}
\label{eq:aux-Bup-BC-big}%
\vcenter{\xymatrix@C=2em@R=1.5em@M=.5em{
&&& \green{i^\approx u_!Y} \ar@[ForestGreen]@{_(->}[ld]_-{\green{\incl}} \ar@[ForestGreen]@/^2.5em/[ddd]^-{\green{\naiuY}}
\\
&& H\times_G Y \ar@/_2em/[llddd]_-{\pr_1} \ar@/^1em/[rdd]_-{\pr_2}|(.35){\hole} \ar@{<-}[d]^-{\cong}_-{e_Y}
& \green{j^\approx Y} \ar@[ForestGreen]@{_(->}[ld]_(.4){\green{\incl}} \ar@[ForestGreen][dd]^-{\green{\najY}} \ar@[ForestGreen][u]^-{\green{e_Y^\approx}}_-{\green{\cong}}
\\
&& L\times_K Y \ar[ld]_-{\pr_1} \ar[rd]_-{\pr_2} \ar@{}[dd]|{\isocell{\gamma_{j/\stY}}}
& \\
& L \ar[rd]_j \ar[ld]_-{v} \ar@{}[dd]|-{\isocell{\alpha}}
&& Y \ar[ld]^{\stY}  \ar@/^2em/[lldd]^(.3){\st{u_!Y}=u\stY}
\\
\kern1em H \kern1em \ar[rd]_-{i}
&& K\ar[ld]^-{u}
&
\\
& G &
}}
\end{equation}
Ignoring the green part ($j^\approx Y$ and~$i^\approx u_!Y$) for a moment, the above (black part of the) diagram is essentially the same as~\eqref{eq:aux-Alp-BC-big}; its outer square is the isocomma $H\times_G Y=(i/u\stY)$ and its central part consists of the given Mackey square and the isocomma square for $L\times_K Y=(j/\stY)$. Hence we have an equivalence
\[
e_Y=
\left\langle
v\pr_1,\ \pr_2,\
\gamma_{j/\stY}\cast\alpha
\right\rangle\colon L\times_K Y\isoto H\times_G Y
\]
as shown in~\eqref{eq:aux-Bup-BC-big}. Since this equivalence is compatible with the projections to~$Y$, namely $\pr_2\circ e_Y=\pr_2$, it restricts to a local equivalence on the $\approx$-loci of~$\pr_2$, \ie it induces the (green) equivalence $e_Y^\approx:=e_Y\restr{j^\approx Y}$ at the top right of~\eqref{eq:aux-Bup-BC-big}. We have
\begin{equation}
\label{eq:aux-Bup-BC-3}%
\naiuY\circ e_Y^\approx=\najY\colon j^\approx Y\to Y
\end{equation}
by construction. Furthermore, we also see in~\eqref{eq:aux-Bup-BC-big} that $e_Y^\approx\colon j^\approx Y\to i^\approx u_!Y$ is compatible with the structure morphisms of those two objects when seen above~$H$. Since technically $j^\approx Y$ lives over~$L$ and since we use $v\colon L\to H$ to see it over~$H$, we are more precisely getting an equivalence in~$\GG/H$ as follows:
\[
e_Y^\approx\colon v_!j^\approx Y\isoto i^\approx u_!Y.
\]
But then the object~$x=x_\sbull\in\Bup(H)$ has an associated coherence isomorphism
\begin{equation}
\label{eq:aux-Bup-BC-4}%
x_{e_Y^\approx}\colon x_{v_!j^\approx Y}\isoto (e_Y^\approx)^* x_{i^\approx u_!Y}
\end{equation}
in~$\cB(v_!j^\approx Y)$.
We can now compare~$(j_*v^*x)_Y$ and~$(u^*i_*x)_Y$:
\begin{align*}
(j_*v^*x)_Y
& \underset{\textrm{{(\ref{eq:aux-Bup-BC-1})}}}{=}
(\najY)_*\big(x_{v_!j^\approx Y}\big)
\underset{\textrm{(\ref{eq:aux-Bup-BC-4})}}{\cong}
(\najY)_*(e_Y^\approx)^* (x_{i^\approx u_!Y})
\\
& \underset{\textrm{(\ref{eq:aux-Bup-BC-3})}}{\cong}
(\naiuY)_*(e_Y^\approx)_*(e_Y^\approx)^*(x_{i^\approx u_!Y})
\cong
(\naiuY)_*(x_{i^\approx u_!Y})
\underset{\textrm{{(\ref{eq:aux-Bup-BC-2})}}}{=}
(u^*i_*x)_Y
\end{align*}
where the middle isomorphism in the second row holds since $e_Y^\approx$ is an equivalence and therefore $(e_Y^\approx)^*$ is an equivalence whose inverse is also its adjoint~$(e_Y^\approx)_*$.
The reader will verify that the above isomorphism $(j_*v^*x)_Y\isoto (u^*i_*x)_Y$ is indeed the mate of $\alpha_!$, on~$x$, at level~$Y$. Since this is an isomorphism for all~$Y$, and for all~$x$, we get the wanted isomorphism $\alpha_!\colon j_*v^*\isoto u^*i_*$.
One further verifies that the inverse is~$(\alpha\inv)_*$. We leave this verification to the reader.
\end{proof}

We now allow the conjugation 2-functor~$\cB\in\conjfun{\GG}$, which was fixed so far in this section, to vary.
The following spells out the equality~\eqref{eq:fake-Mackey}.
\begin{Lem}
\label{Lem:leq-transfo-induction}
Let $t\colon \cB\to\cB'$ be a morphism of conjugation $2$-functors, and let $s\colon K\to H$ be a local equivalence in~$\GG$. Then we have natural isomorphisms
\[
(t_s)_*\colon t_H\,s_*^{\cB}\isoTo s_*^{\cB'}\,t_K\,,
\qquadtext{and}
(t_s\inv)_!\colon s_*^{\cB'}\,t_K\isoTo t_H\,s_*^{\cB}
\]
of functors $\cB(K)\to \cB'(H)$, obtained as the mates of the restriction-compatibility isomorphisms $t_s\colon s^*t_H\isoto t_K s^*$
(and its inverse) with respect to the folding 2-sided adjoints~$s_*\dashv s^*\adj s_*$ in the conjugation 2-functors~$\cB$ and~$\cB'$ (\Cref{Prop:leq-adjoints}).
\end{Lem}

\begin{proof}
Using~\Cref{Lem:leq-folding}, we reduce to the case where $s=\nabla^{(n)}\colon H^{\sqcup n}\to H$ is a folding as in~\Cref{Exa:folding}.
In that case, the result follows from additivity of~$t_H$. Details are left to the reader.
\end{proof}

\begin{Thm}
\label{Thm:Bup-Mackey}%
Let $\cB\colon \GGi^\op\to \ADD$ be a conjugation $2$-functor.
Then the restriction 2-functor $\Bup \colon \GG^\op\to \ADD$ of~\Cref{sec:Bup} is a Mackey $2$-functor.

Together with~\Cref{Cons:Bup-transfos} we obtain a well-defined 2-functor
\[
\Bup[(-)]\colon \conjfun{\GG}\too \mackfun{\GG}
\]
meaning that for every morphism~$t\colon \cB\to \cB'$ of conjugation 2-functors, the associated morphism of restriction 2-functors~$\Bupt\colon \Bup\to \Bupp$ is compatible with induction.
\end{Thm}

\begin{proof}
Let us verify the Mackey axioms for the restriction 2-functor~$\Bup$.
Additivity~\Mack{1} is already part of $\Bup$ being a restriction 2-functor (\Cref{Prop:Bup-2-functor}).
The two-sided special Frobenius adjunction~\Mack{2} is
\Cref{Prop:Bup-adjunction}. Base-change~\Mack{3} is exactly~\Cref{Prop:Bup-base-change}.
Therefore $\Bup$ is Mackey.

We now explain the functoriality in~$\cB$. Let $t\colon \cB\to\cB'$
be a transformation of conjugation $2$-functors. It yields a strict morphism of
restriction 2-functors $\Bupt\colon \Bup\to\Bupp$ by~\Cref{Cons:Bup-transfos}.

We claim that this is a morphism of Mackey $2$-functors, \ie that it is compatible with induction as in~\Cref{Rec:Mack(G)}.
For every faithful~$j\colon G'\into G$, the restriction-comparison isomorphism of the induced transformation~$\Bupt\colon\Bup\to\Bupp$ is \[ (\Bupt)_j=\id\colon j^*\Bupt_G\isoTo\Bupt_{G'}j^*. \] Its right mate is obtained by composing with the unit and counit of $j^*\adj j_*$: \[ \Bupt_Gj_* \xrightarrow{\ \reta^{(j)}\ } j_*j^*\Bupt_Gj_* \xrightarrow{\ j_*(\Bupt)_j j_*\ } j_*\Bupt_{G'}j^*j_* \xrightarrow{\ j_*\Bupt_{G'}\reps^{(j)}\ } j_*\Bupt_{G'}. \] At every level~$H\in\GG/G$, this mate identifies with \[ (t_{\najH})_*\colon t_H(\najH)_*^{\cB} \isoTo (\najH)_*^{\cB'}t_{j^\approx H}, \] where $\najH\colon j^\approx H\apprto H$ is the local equivalence of \Cref{Cons:j-approx}. This is an isomorphism by \Cref{Lem:leq-transfo-induction}. Hence the mate of~$(\Bupt)_j$ is an isomorphism, so~$\Bupt$ is compatible with induction.
\end{proof}

%
\section{The mark transformation}
\label{sec:mark}
%

In this section we compare the left and right mackeyfications of a restriction 2-functor~$\cA$. From~\Cref{Thm:Alp-UP,Thm:Bup-UP} we have
\begin{equation}
\label{eq:Aup-A-Alp}%
\xymatrix@C=5em{
\Aup \ar[r]^-{\displaystyle\epsup_{\cA}}
& \cA \ar[r]^-{\displaystyle\etalp_{\cA}}
& \Alp.
}
\end{equation}
Writing the forgetful functors $U\colon \mackfun{\GG}\hook \resfun{\GG}$ and~$V\colon \resfun{\GG}\hook\conjfun{\GG}$, the right-hand~$\Alp$ in~\eqref{eq:Aup-A-Alp} is the restriction 2-functor~$U(\Alp)$  and the left-hand $\Aup$ is the left mackeyfication of~$\cA$ viewed as a conjugation 2-functor, that is, $\Aup:=\Bup[(V\!\cA)]$. In the same vein, $\epsup_{\cA}$ is truly $\epsup_{V\!\cA}$.

Let us also recall the morphism of restriction 2-functors $\etaup_{\cA}\colon \cA\to \Bup[(V\!\cA)]=\Aup$ of~\Cref{Cons:Bup-unit}, namely the unit of the $(V\adj\Bup[(-)])$-adjunction.
Applying~\Cref{Thm:Alp-UP} to $t=\etaup_{\cA}$ gives a morphism of Mackey $2$-functors $\widehat{\etaup_{\cA}}\colon \Alp\to\Aup$ such that the following equality of morphisms of restriction 2-functors holds
\begin{equation}
\label{eq:mark-ancestor}%
U(\widehat{\etaup_{\cA}})\circ \etalp_{\cA}=\etaup_{\cA}
\end{equation}
where $U\colon \mackfun{\GG}\hook\resfun{\GG}$ is the other forgetful functor.
We trust the reader can restore~$U$ and~$V$ as needed, and we suppress them from now on, unless their mention helps cognition.

\begin{Def}
\label{Def:mark}%
Let $\cA\colon \GG^\op\to \ADD$ be a restriction $2$-functor.
The \emph{mark transformation} of~$\cA$ is the above morphism of Mackey 2-functors
\[
\mu_{\cA}:=\widehat{\etaup_{\cA}}\colon \Alp\to\Aup
\]
given by the explicit formula in~\Cref{Thm:Alp-UP}:
\[
\mu_{\cA} = \epslp_{(\Aup)}\circ \Alp[(\etaup_{\cA})]
\]
where $\epslp_{(\Aup)}\colon \Alp[(\Aup)]\to \Aup$ is the counit of the $(\Alp[(-)]\adj U)$-adjunction at the Mackey 2-functor~$\Aup$, as described in~\Cref{Cons:Alp-counit} and $\Alp[(\etaup_{\cA})]$ is the image of~$\etaup_{\cA}$ under~$\Alp[(-)]\colon \resfun{\GG}\to \mackfun{\GG}$ as in~\Cref{Cons:Alp-transfos}.
It is characterized by the equality~\eqref{eq:mark-ancestor} in~$\resfun{\GG}$, that is, forgetting~$U$ and~$V$,
\begin{equation}
\label{eq:mark}%
\mu_{\cA}\circ \etalp_{\cA}=\etaup_{\cA}.
\end{equation}
\end{Def}

\begin{Rem}
\label{Rem:mark-all}%
Completing~\eqref{eq:Aup-A-Alp} with~$\mu_{\cA}$ and~$\etaup_{\cA}$, we get a diagram of pseudo-natural transformations with varying degrees of naturality (`mack', `res', and `conj' indicate morphisms of Mackey, restriction, and conjugation 2-functors respectively):
\begin{equation}
\label{eq:mark-all-morphisms}%
\vcenter{
\xymatrix@C=5em@R=5em{
& \cA \ar@/^2em/[rd]^-{\etalp_{\cA}}_(.45){\scriptstyle\textrm{(res)}\!\!\!} \ar@/_2em/[ld]_-{\etaup_{\cA}}^(.45){\!\!\!\scriptstyle\textrm{(res)}}
\\
\Aup \ar@/_2em/[ru]^-{\epsup_{\cA}\!\!}_(.55){\!\!\!\scriptstyle\textrm{(conj)}}
&& \Alp \ar@/^2em/[ll]^-{\mu_{\cA}}_-{\scriptstyle\textrm{(mack)}}
}}
\end{equation}
By~\eqref{eq:mark}, the outer triangle in~\eqref{eq:mark-all-morphisms} commutes.
The left-hand composite in~\eqref{eq:mark-all-morphisms} $\epsup_{\cA}\circ\etaup_{\cA}\colon \cA\to \Aup\to \cA$ is truly $\epsup_{V\!\cA}\circ V(\etaup_{\cA})\colon V\!\cA\to V(\Bup[(V\!\cA)])\to V\!\cA$, hence equals the identity, by the unit-counit relation of~\Cref{Prop:Bup-unit-counit}. Therefore
\begin{equation}
\label{eq:for-later}%
\epsup_{\cA}\circ \mu_{\cA}\circ \etalp_{\cA}
\equalby{(\ref{eq:mark})}
\epsup_{\cA}\circ \etaup_{\cA}
=\id_{\cA}
\end{equation}
as morphisms of conjugation 2-functors, \ie $\epsup_{V\!\cA}\circ VU(\mu_{\cA})\circ V(\etalp_{\cA}) = \id_{V\!\cA}$.
We return to the composite $\epsup_{\cA}\circ \mu_{\cA}\colon \Alp\to \cA$ in~\Cref{Rem:mark-other-side}.
\end{Rem}

We want to give an explicit formula for~$\mu_{\cA,G}\colon \Alp(G)\to \Aup(G)$ for every $G\in \GG$.
Recall~\Cref{Cons:Alp-category} for~$\Alp(G)$ and~\Cref{Cons:Bup-category} for~$\Aup(G)$.

\begin{Thm}
\label{Thm:mark-explicit}%
Let $G\in\GG$. Let $(H,x)\in\Alp(G)$ be an object. Then $\mu_{\cA,G}(H,x)$ in~$\Aup(G)$ is given at level~$K\in\GG/G$ by the following object of~$\cA(K)$
\[
\big(\mu_{\cA,G}(H,x)\big)_K =
(p_2)_* (p_1)^*(x)
\]
where $p_1=\pr_1\restr{\stH^\approx K}$ and $p_2=\pr_2\restr{\stH^\approx K}$ are the restrictions of the two projections $\pr_1$ and~$\pr_2$ of the
isocomma square defining~$H\times_G K=(\stH/\stK)$
\begin{equation}
\label{eq:mark-explicit}%
\vcenter{
\xymatrix@C=1em@R=.5em{
& H\times_G K \ar[ld]_-{\pr_1} \ar[rd]^-{\pr_2}
\ar@{}[dd]|{\isocell{\gamma}}
\\
H \ar[rd]_-{\stH}
&& K \ar[ld]^-{\stK}
\\
& G
}}
\end{equation}
on $\stH^\approx K$, the part of~$H\times_G K$ where $\pr_2$ is fully faithful (\Cref{Cons:j-approx}).
The functor~$(p_2)_*$ in~\eqref{eq:mark-explicit} is `folding induction' as in \Cref{Prop:leq-adjoints}.
\end{Thm}

\begin{proof}
In the formula $\mu_{\cA} = \epslp_{(\Aup)}\circ \Alp[(\etaup_{\cA})]$ of~\Cref{Def:mark}, we can unpack $\epslp_{(\Aup)}$ from~\Cref{Cons:Alp-counit} and $\Alp[(-)]$ from~\Cref{Cons:Alp-transfos}. This gives
\begin{equation}
\label{eq:aux-mark-1}%
\mu_{\cA,G}(H,x)
=
(\stH)_*\big(\etaup_{\cA,H}(x)\big)
\end{equation}
as an object of $\Aup(G)$, where $(\stH)_*$ is induction in the Mackey 2-functor~$\Aup$.
Unpacking $j_*$ of~\Cref{Cons:Bup-induction} for~$j=\stH$, we have at level~$K\in\GG/G$
\begin{equation}
\label{eq:aux-mark-2}%
\big(\mu_{\cA,G}(H,x)\big)_K
\equalby{(\ref{eq:aux-mark-1})}\Big((\stH)_*\big(\etaup_{\cA,H}(x)\big)\Big)_K
=(\naj[\stH]{K})_*\big((\etaup_{\cA,H}(x))_{\stH^\approx K}\big).
\end{equation}
Unpacking $\najH$ of~\Cref{Cons:j-approx} for $j=\stH$, we see that the local equivalence $\naj[\stH]{K}$ is exactly the morphism~$p_2$ of the statement.
The $\approx$-locus $\stH^\approx K$, which is a subobject of $H\times_G K$, is viewed over~$H$ by restricting~$\pr_1\colon H\times_G K\to H$. In other words, as an object of~$\GG/H$, our $\stH^\approx K$ has structure morphism $\st{\stH^\approx K}=\pr_1\circ \incl_{\stH^\approx K}=p_1$ the morphism~$p_1$ of the statement.
Finally, unpacking $\etaup_{\cA}$ from~\Cref{Cons:Bup-unit} we have by the above discussion
\[
(\etaup_{\cA,H}(x))_{\stH^\approx K}=\st{\stH^\approx K}^*(x)=p_1^*(x).
\]
Plugging this in~\eqref{eq:aux-mark-2} and replacing $\naj[\stH]{K}$ by~$p_2$ we get
\[
\big(\mu_{\cA,G}(H,x)\big)_K = (p_2)_*(p_1^*(x))
\]
as announced.
\end{proof}

\begin{Rem}
\label{Rem:mark-other-side}%
We have introduced and computed the mark transformation~$\mu_{\cA}$ from the perspective of right mackeyfication. We can also approach it from the other side. Recall that in~\Cref{Rem:mark-all} we left one composite of the diagram~\eqref{eq:mark-all-morphisms} in the air, namely $\epsup_{\cA}\circ \mu_{\cA}\colon \Alp\to \cA$. We proved in~\eqref{eq:for-later} that it is a retraction of~$\etalp_{\cA}$, as morphisms of conjugation 2-functors.
Actually, it is equal to the other retraction of~$\etalp_{\cA}$ that we have in store, namely the~$\sigma_{\cA}$ of~\Cref{Rem:etalp-split}.
Indeed, for every $(H,x)$ in~$\Alp(G)$, let us unpack the explicit formula in~\Cref{Thm:mark-explicit} for~$K=G$, using the identification $H\times_G G\cong H$ under which~$\pr_2=\stH\colon H\into G$ and therefore $\stH^\approx K=(H\times_G K)^{\pr_2\approx}=H^{\stH\approx}=H^\approx$ in the notation of~\Cref{Rem:etalp-split}; the two morphisms $p_1$ and~$p_2$ of~\Cref{Thm:mark-explicit} become $p_1=(\pr_1)\restr{H^\approx}=\incl_{H^\approx}\colon H^\approx\hook H$ and
$p_2=(\pr_2)\restr{H^\approx}=\stH^\approx$.
Therefore we obtain in~$\cA(G)$
\[
\epsup_{V\!\cA,G}\circ \mu_{\cA,G}(H,x)
\equalby{\ref{Cons:Bup-counit}}
\big(\mu_{\cA,G}(H,x)\big)_G
\equalby{\ref{Thm:mark-explicit}}
(\stH^\approx)_*(\incl_{H^\approx}^*(x))
\equalby{\ref{eq:sigma}}
\sigma_{\cA}(H,x).
\]
Thus we have a natural identification of morphisms of conjugation 2-functors
\begin{equation}
\label{eq:mark-other}%
\epsup_{\cA}\circ\mu_{\cA}=\sigma_{\cA}.
\end{equation}
This equation gives the alternative description of the mark transformation~$\mu_{\cA}$, by applying~\Cref{Thm:Bup-UP} to~$t=\sigma_{\cA}$. We must have $U(\mu_{\cA})\cong\widetilde{\sigma_{\cA}}$,
which characterizes $\mu_{\cA}$ as a morphism of restriction 2-functors, hence as a morphism of Mackey 2-functors since the latter are morphisms of restriction 2-functors with additional properties.
Again, uniqueness is up to unique modification.
Unpacking the explicit description of~$\tilde{t}$ in~\Cref{Thm:Bup-UP} applied to~$t=\sigma_{\cA}$, we have
\begin{equation}
\label{eq:mark-alternative}%
\mu_{\cA}\cong\Bup[(\sigma_{\cA})]\circ \etaup_{(\Alp)}.
\end{equation}
where $\Bup[(-)]$ is as in~\Cref{Cons:Bup-transfos} and~$\etaup$ is the unit of~\Cref{Cons:Bup-unit}.
\end{Rem}

\begin{Cor}
\label{Cor:mark-via-sigma}%
Let $G\in\GG$ and let $(H,x)\in\Alp(G)$. For every object
$\stK\colon K\to G$ of $\GG/G$, consider the restriction
$\stK^*(H,x)\in\Alp(K)$. Then, with $\sigma_{\cA}$ as in
\Cref{Rem:etalp-split}, we have a canonical isomorphism in $\cA(K)$
\[
\big(\mu_{\cA,G}(H,x)\big)_K
\cong
\sigma_{\cA,K}\big(\stK^*(H,x)\big).
\]
\end{Cor}

\begin{proof}
Unpack~\eqref{eq:mark-alternative} and~\Cref{Cons:Bup-unit} for~$G$, on~$(H,x)$, at level~$K$.
\end{proof}

%
\section{Examples}
\label{sec:examples}%
%

In this final section, we want to describe $\Alp$ and~$\Aup$ in the case of a constant restriction 2-functor~$\cA$.
Let us set up the notation.

%
\subsection{The constant restriction functor}\
\label{ssec:cst-A}%

%
\begin{Not}
\label{Not:Ak}%
\label{Not:cst-A}%
Let $\cA_0\in\ADD$ be a fixed additive category.
Define a 2-functor from groups (as a full 2-subcategory of~$\gpd$) by setting for every finite group~$G$
\[
\cA(G)=\cA_0
\]
independently of~$G$, with all restrictions $u^*=\cA(u)$ and all conjugation transformations~$\alpha^*=\cA(\alpha)$ equal to identity.
By~\Cref{Rem:additive=>groups}, we can equivalently think of~$\cA$ as a restriction 2-functor on the 2-category~$\gpd$ of finite groupoids by setting
\[
\cA(G)=\cA_0^{\pi_0(G)}=\bigoplus_{\pi_0(G)}\cA_0
\]
the coproduct of as many copies of our constant category~$\cA_0$ as there are connected components in the groupoid~$G$. For every $u\colon H\to G$, the restriction functor $u^*\colon\cA(G)\to \cA(H)$ is restriction along the map~$\pi_0(u)\colon \pi_0(H)\to \pi_0(G)$. All 2-cells in~$\gpd$ yield the identity transformation on~$\cA(-)$.
Of course we can restrict~$\cA$ to $\gpdf$ if we care only about faithful morphisms. We shall do this in the case of~$\Aup$.

For instance, let $\kk$ be a commutative ring (\eg a field) and $\cA_0=\kk\ffree$ the category of finitely generated free~$\kk$-modules.
Applying the above construction gives a restriction 2-functor that we shall denote~$\Ak$, namely
\[
\Ak(G)=(\kk\ffree)^{\pi_0(G)}=\bigoplus_{\pi_0(G)}(\kk\ffree).
\]
for every groupoid~$G$.
\end{Not}

\begin{Rem}
The reader could be excused for thinking that constant 2-functors as above are Mackey 2-functors but they are not. Indeed, all restrictions do admit 2-sided adjoints very much as in the proof of~\Cref{Prop:leq-adjoints}: After all, every $u^*$ is made of zeros and identities and $u_*$ is just a `folding'. However, these adjoints do not satisfy the Mackey formulas of~\Mack{3}. For instance, for the Mackey square of~\Cref{Exa:Mackey-square} associated to subgroups~$i\colon H\hook G$ and~$u\colon K\hook G$, the composite $u^*i_*\colon \cA(H)\to\cA(K)$ via~$\cA(G)$ is the identity but the composite $j_*v^*$ via~$\cA(H\times_G K)$ is the direct sum of $|\KGH|$ copies of the identity.
So it is legitimate to submit~$\cA$ to left and right mackeyfication.
\end{Rem}

%
\subsection{The $\kk$-linear Burnside category}\
\label{ssec:Burnside}%

We recall a classical construction going back to early days of Mackey functors.

\begin{Rec}
\label{Rec:Burnside-cat}%
Let $G$ be a finite group. The \emph{Burnside category} $\Burn[\bbZ](G)$ has objects the finite $G$-sets $X$ and morphisms defined as follows. For every pair of finite $G$-sets~$X$ and~$Y$, consider all possible spans $X\lto W\to Y$ of finite $G$-sets. Keeping~$X$ and~$Y$ fixed, there are obvious notions of isomorphism of spans and of sum ($\amalg$) of spans from~$X$ to~$Y$, performed on the middle part. The set of isomorphism classes of spans from~$X$ to~$Y$ is an abelian monoid under~$\amalg$. Its group-completion (Grothendieck group) is
\[
K_0\big(\{X\lto W\to Y\},\amalg\big)=:\Hom_{\Burn[\bbZ](G)}(X,Y)
\]
used as the Hom-group from~$X$ to~$Y$ in~$\Burn[\bbZ](G)$.
Composition is done via the cartesian product of $G$-sets: $[X\lto W\to Y]_{\simeq}$ followed by~$[Y\lto V\to Z]_{\simeq}$ compose to $[X\lto W\times_Y V\to Z]_{\simeq}$ with the usual legs. This passes to~$K_0$.

For our ring~$\kk$, the \emph{$\kk$-linear Burnside category} $\Burn(G)$ has the same objects as above (finite $G$-sets) and $\kk$-linearly extended Hom-groups
\[
\Hom_{\Burn(G)}(X,Y):=\kk\otimes_\bbZ \Hom_{\Burn[\bbZ](G)}(X,Y),
\]
with $\kk$-linearly extended composition.
This construction~$G\mapsto \Burn(G)$ is actually part of a Mackey 2-functor on~$\gpd$, as explained in \cite[Section~7.2]{BalmerDellAmbrogio20}.
\end{Rec}

\begin{Rec}
\label{Rec:transporter}%
Let $G$ be a group. The connection between $G$-sets, as in the Burnside category construction, and groupoids over~$G$ in~$\gpd$, as in the right mackeyfication~\Cref{Cons:Alp-category}, is given by the \emph{transporter groupoid}.
For every finite $G$-set~$X$ the transporter groupoid $G\ltimes X$ has objects the set~$X$ and morphisms $x\to y$ given by $\SET{g\in G}{g x=y}\subseteq G$. Composition is multiplication in~$G$.
See \cite[Definition~B.0.6]{BalmerDellAmbrogio20}.
We can view~$G\ltimes X$ as an object in~$\gpdf/G$ for it comes with a faithful morphism~$\st{G\ltimes X}\colon G\ltimes X\into G$ mapping~$x\in X$ to the unique object of~$G$ and every $g$ to~$g$. This turns $G$-maps~$f\colon X\to Y$ into morphism~$G\ltimes f\colon G\ltimes X\to G\ltimes Y$ given by $f$ on objects and the identity on morphisms; this defines a morphism in~$\GG/G$ with structure 2-cell~$\st{G\ltimes f}=\id$. As explained in \cite[Proposition B.0.9]{BalmerDellAmbrogio20} this construction yields a biequivalence
\[
G\ltimes-\colon G\sset\isoto \gpdf/G
\]
where the source 1-category is viewed as a 2-category with only identity 2-cells. In particular, it induces an equivalence of ordinary categories on the 1-truncations
\[
G\ltimes-\colon G\sset\isoto \tau_1(\gpdf/G)
\]
where the 1-category $\tau_1(\gpdf/G)$ has the same objects as~$\gpdf/G$ and morphisms the \emph{isomorphism classes} of morphisms in~$\gpdf/G$ (\ie modulo invertible 2-cells).

This transporter-groupoid construction can be extended to spans of finite $G$-sets, sending~$(X\xlto{w_1} W\xto{w_2} Y)$ to the object~$(G\ltimes W,G\ltimes w_1,G\ltimes w_2)$ in the 2-category~$\SpanfG(G\ltimes X\,,\,G\ltimes Y)$.
\end{Rec}

\begin{Thm}
\label{Thm:Alp-cst}%
Let $G$ be a finite group, viewed in~$\gpd$ (\Cref{Exa:gps}) and consider the right Mackeyfication~$\Alpk$ of the constant restriction 2-functor~$\Ak$ of~\Cref{Not:Ak}.
The transporter groupoid $G\ltimes-$ yields a well-defined equivalence
\[
\theta_G\colon \Burn(G)\isoto \Alpk(G)
\]
mapping a $G$-set~$X$ to~$(G\ltimes X,\ukk)$ where $\ukk$ is~$(\kk,\ldots,\kk)\in \!\!\!\bigoplus\limits_{\pi_0(G\ltimes X)}\!\!\kk\ffree=\Ak(G\ltimes X)$.
\end{Thm}
\begin{proof}
Let us abbreviate $\cA_0=\kk\ffree$.
For every $H\in\gpd$, the object $\ukk=\ukk(H)$ in $\Ak(H)=\cA_0^{\pi_0(H)}$ which is~$\kk$ in every spot (as in the statement for~$H=G\ltimes X$) has the amusing property that for all~$u\colon H\to K$ we have $u^*(\ukk)=\ukk$, or in expanded notation $u^*(\ukk(K))=\ukk(H)$.

Therefore the functor~$\theta_G$ can simply be defined on spans~$S=(X\xlto{w_1} W\xto{w_2} Y)$ as $[P,p_1,p_2;\id_{\ukk}]$, where the span $P\in\SpanfG(G\ltimes X,G\ltimes Y)$ is~$G\ltimes W$, with the wings~$p_i=G\ltimes w_i$ for $i=1,2$, as in~\Cref{Rec:transporter}, and where $\id_{\ukk}$ means the identity of~$\ukk=p_1^*(\ukk)=p_2^*(\ukk)$ in~$\Ak(P)$. One verifies that this is a well-defined functor, using that $G\ltimes-$ turns the cartesian squares of $G$-sets to Mackey squares (\cite[Remark~B.0.5]{BalmerDellAmbrogio20}).
We leave to the reader to verify that $\theta_G$ is $\kk$-linear, and in particular additive.

To see that $\theta_G\colon \Burn(G)\to \Alpk(G)$ is essentially surjective, it suffices to observe (\Cref{Rem:Alp-add}) that for every $(H,x)$ with $x\in\Ak(H)$, if $H=H_1\sqcup \cdots \sqcup H_r$ and $x=(x_1,\ldots,x_r)$ in~$\Ak(H)=\Ak(H_1)\oplus \cdots \oplus\Ak(H_r)$ then $(H,x)\cong\oplus_{i=1}^{r} (H_i,x_i)$ in~$\Alpk(G)$. So we can assume that $H$ is connected. Then if~$x\cong\kk^n$ we have $(H,x)\cong(H,\kk)^{\oplus n}$ so we can assume that $x=\kk$. But then $(H,\kk)=\theta_G(X)$ for any $G$-set~$X$ such that $G\ltimes X\simeq H$, for instance $X=G/\Aut_{H}(a)$ for any object~$a\in H$.

Therefore it suffices to show that $\theta_G$ is fully faithful, \ie that for every pair of $G$-sets~$X$ and~$Y$, with transporter groupoids~$H:=G\ltimes X$ and~$K:=G\ltimes Y$, the functor $\theta_G$ induces a bijection
\begin{equation}
\label{eq:aux-theta}%
\Hom_{\Burn(G)}(X,Y)\isoto \Hom_{\Alpk(G)}\big((H,\ukk),(K,\ukk)\big).
\end{equation}
By additivity again, we can assume that $X$ and~$Y$ are orbits, so $H$ and~$K$ are connected, and we can write~$\kk$ instead of~$\ukk$ as $\pi_0(H)=\pi_0(K)=\ast$.

The inverse of~\eqref{eq:aux-theta} is relatively easy to construct. A morphism representative as in~\Cref{Cons:Alp-category} from~$(H,\kk)$ to~$(K,\kk)$ is given by a span~$P\in\SpanfG(H,K)$ and a morphism~$f\colon p_1^*\kk\to p_2^*\kk$. As the latter is~$\ukk\in\Ak(P)$, the morphism~$f$ is a collection of scalars~$(f_C)_{C\in\pi_0(P)}$ indexed by the connected components of~$P$. By additivity, we can assume~$P$ to be connected, so~$f=f_P\in\kk$ is a scalar~$\lambda(f)$.
Using~\Cref{Rec:transporter}, we can assume up to replacing~$P$ by equivalence, that $P=G\ltimes W$ and~$p_i=G\ltimes w_i$ for a span~$X\xlto{w_1} W\xto{w_2} Y$ of~$G$-sets. We then send~$(P,f)$ to~$\lambda(f)\cdot[X\lto W\to Y]$. Changing $P$ up to strong equivalence, or changing the choice of~$W,w_1,w_2$ gives the same isomorphism class of span. Let us check that this construction is well-defined up to $\approx$-equivalence of~\Cref{Cons:Alp-category}. Using additivity again, we reduce to the case where $s\colon P\to P'$ is an equivalence, in which case $s^*=\id$ and~$s_*=\id$ as well. The critical place where one uses that $\Ak$ is constant is when we `adjust'~$f\colon p_1^*\kk\to p_2^*\kk$ before computing its trace with respect to~$s$, as explained in~\eqref{eq:f-adjusted}. Since (in the notation of~\eqref{eq:f-adjusted}) we have $\sigma_i^*=\id$, where $\sigma_i$ are the wing 2-cells of~$s=(s,\sigma_1,\sigma_2)$, we get indeed that $\tr_s(f)=\lambda(f)$. So the relation~$\tr_s(f)=f'$ forces the scalars $\lambda(f)$ and~$\lambda(f')$ to be the same.
This discussion yields a well-defined map backwards from~\eqref{eq:aux-theta}
\[
\Hom_{\Alpk(G)}\big((H,\ukk),(K,\ukk)\big)\to \Hom_{\Burn(G)}(X,Y).
\]
It is now easy to verify, using additivity, that the latter is an inverse of~\eqref{eq:aux-theta}.
\end{proof}

\begin{Rem}
The reader can verify that the equivalence of~\Cref{Thm:Alp-cst} upgrades to an equivalence of Mackey 2-functors $\theta\colon\Burn\isoto \Alpk$. One can actually exploit the fact that~$\Burn$ is a Mackey 2-functor to construct the inverse of~$\theta$ more abstractly than in the above proof. Indeed, there is a ($\kk$-linear) morphism of restriction $2$-functors on~$\gpd$
\[
\beta\colon \Ak\to \Burn
\]
characterized by the fact that for $G$ a finite group~$\beta_G(\kk)=G/G$.
(In this argument it might be better to only use finite groups, as in \cite[Section~4.3]{BalmerDellAmbrogio20}.)
By our~\Cref{Thm:Alp-UP}, the morphism~$\beta\colon \Ak\to \Burn$ induces a morphism of Mackey 2-functors~$\hat\beta\colon \Alpk\to \Burn$. It is essentially characterized by the fact that for every subgroup~$H\le G$ the functor~$\hat\beta_G\colon \Alpk(G)\to \Burn(G)$ maps the object~$(H,\kk)$ to~$G/H$. This uses that $(\incl_H)_*(H/H)=G/H$ in the Mackey 2-functor~$\Burn$.

The reader will verify that this morphism $\hat{\beta}\colon \Alpk\to \Burn$ is an inverse of the above morphism~$\theta$ in~$\mackfun{\gpdf}$.
\end{Rem}

\begin{Rem}
\Cref{Thm:Alp-cst} gives~\Cref{Thm:Burnside-intro}\,\eqref{it:Burnside-lp} in the Introduction.
\end{Rem}

%
\subsection{Left Mackeyfication of constant 2-functors}\
\label{ssec:cst^+}%

%
\begin{Not}
\label{Not:GG/G-approx}%
Recall the 2-comma category~$\GG/G$ of~\Cref{Def:GG/G}. We denote by~$\GGisoconn/G$ the 2-subcategory with objects $(H\to G)$ that are indecomposable, meaning that~$H$ is non-empty indecomposable in~$\GG$ (that is, \emph{connected} in the underlying~$\gpd$) and whose 1-morphisms $s\colon H\isoto K$ are equivalences in~$\GG/G$, which is the same as asking the underlying $s\colon H\to K$ to be an equivalence in~$\GG$, or a local equivalence (\Cref{Def:loc-equiv}) for that matter. The 2-cells stay as in~$\GG/G$.
As for any 2-category, this $\GGisoconn/G$ admits a 1-truncated (ordinary) category
\[
\tau_1(\GGisoconn/G)
\]
as discussed in~\Cref{Rec:transporter}; see~\cite[Notation~A.1.14]{BalmerDellAmbrogio20}.
\end{Not}
\begin{Exa}
\label{Exa:Weyl-as-endomorphisms}%
Continuing on \Cref{Exa:end(H)-in-gpdf/G}, for $\GG=\gpdf$, let $G$ be a group and~$H$ a subgroup.
The object~$H=(H\xinto{\incl}G)$ is connected, so it is an object of~$\GGisoconn/G$. Its automorphism group in~$\tau_1(\GGisoconn/G)$ is the Weyl group of~$H$:
\[
N_G(H)/H\isoto \End_{\tau_1(\GGisoconn/G)}(H)
\]
where the isomorphism sends a class~$[g]\in N_G(H)/H$ to~$c_g\colon H\apprto H$ with structure 2-cell~$\st{c_g}=\gamma_g$.
This is immediate from \Cref{Exa:end(H)-in-gpdf/G}.
\end{Exa}

\begin{Thm}
\label{Thm:Bup-cst}%
Let $\cB_0\in\ADD$ be a fixed additive category and~$\cB\in\conjfun{\GG}$ be the associated constant conjugation 2-functor~$\cB(G)=\cB_0^{\pi_0(G)}$ as in~\Cref{Not:cst-A}.
Let $G\in\GG$. Then with~\Cref{Not:GG/G-approx} we have an equivalence
\[
\Bup(G)\cong \Fun\big(\tau_1(\GGisoconn/G),\cB_0\big)
\]
between $\Bup(G)$ and the category of $\cB_0$-valued functors on the category $\tau_1(\GGisoconn/G)$. It is given by sending~$x_\sbull\in\Bup(G)$ to the functor that maps $H\in\GGisoconn/G$ to $x_H\in\cB_0$ and every morphism $[s]_{\simeq}\colon H\to K$ in~$\tau_1(\GGisoconn/G)$ to~$x_s$.
\end{Thm}
\begin{proof}
This equivalence is almost an equality, or perhaps more intuitively a `trimming of redundancies'.
Consider the very definition of~$\Bup(G)$ in~\Cref{Cons:Bup-category}.
The additivity Condition~\eqref{it:x-trivial} tells us that an object~$x_\sbull$ of~$\Bup(G)$ is
characterized by the data of~$x_H$ with $H$ indecomposable and of~$x_s$ for $s\colon H\apprto K$ in~$\GGisoconn/G$, that is, between indecomposable objects. But a local equivalence~$s$ between indecomposable objects gives the identity on~$\pi_0$ and therefore the induced functor $s^*\colon \cB(K)=\cB_0\to \cB(H)=\cB_0$ is the identity of~$\cB_0$. It follows that Condition~\eqref{it:x-ts} becomes simply $x_{t\circ s}=x_t\circ x_s$. Similarly, for every 2-cell $\alpha\colon s\isoTo t\colon H\to K$ in~$\GGisoconn/G$, the transformation $\cB(\alpha)=\alpha^*\colon \Id_{\cB_0}\To \Id_{\cB_0}$ is the identity since~$\cB$ is constant. It follows from Condition~\eqref{it:x-alpha} that isomorphic 1-cells~$s,t\colon H\to K$ in~$\GGisoconn/G$ induce the same isomorphism~$x_s=x_t$. In short, $x_\sbull$ is just a functor from the category $\tau_1(\GGisoconn/G)$ to~$\cB_0$. It is clear that morphisms $f_\sbull$ correspond simply to natural transformations, since again~$s^*$ disappears from~\eqref{eq:Bup-morphisms}.
\end{proof}
\begin{Cor}
\label{Cor:Bup-cst}%
Let $\cB\in\conjfun{\gpdf}$ be the constant conjugation 2-functor associated to a fixed additive category~$\cB_0\in\ADD$ as in~\Cref{Not:cst-A}, on the (2,1)-category $\GG=\gpdf$ of finite groupoids with faithful morphisms.
Let $G$ be a finite group. Choose representatives $H_1,\ldots,H_r$ of conjugacy classes of subgroups of~$G$ and let $W_i=N_G(H_i)/H_i$ be the corresponding Weyl group in~$G$, for~$i=1,\ldots,r$.
Then we have an equivalence
\[
\Bup(G)\cong \bigoplus_{i=1}^{r} \Fun(W_i,\cB_0)
\]
where $\Fun(W_i,\cB_0)$ is the category of representations of~$W_i$ in~$\cB_0$ (\Cref{Exa:gps}).
\end{Cor}
\begin{proof}
View each $W_i$ as a one-object groupoid and let~$\cW:=W_1\sqcup \cdots \sqcup W_r$, that is, the category with $\{1,\ldots,r\}$ as objects and $\End_{\cW}(i)=W_i$ and~$\Mor_{\cW}(i,j)=\varnothing$ for~$i\neq j$. Define a functor
\begin{equation}
\label{eq:aux-H-tau}%
\omega\colon\cW\to \tau_1(\GGisoconn/G)
\end{equation}
by sending $i$ to~$(H_i\xinto{\incl}G)$ and, $[g]\in W_i=N_G(H_i)/H_i$, to $[(c_g,\gamma_g)]$, using
\[
W_i\isoto \End_{\tau_1(\GGisoconn/G)}(H_i)
\]
as in~\Cref{Exa:Weyl-as-endomorphisms}.
Moreover, every indecomposable object over~$G$ is equivalent to one of the chosen subgroups~$H_i$, since $\GG=\gpdf$ has only faithful morphisms. There are no morphisms between distinct representatives~$H_i$ and~$H_j$ in~$\tau_1(\GGisoconn/G)$ as we saw in~\Cref{Exa:end(H)-in-gpdf/G}. Hence~$\omega$ is an equivalence.

This equivalence~$\omega$ in~\eqref{eq:aux-H-tau} induces by restriction an equivalence on (ordinary) functor categories $\Fun(\tau_{1}(\GGisoconn/G),\cB_0)\cong\Fun(\cW,\cB_0)$ and the result now follows from~\Cref{Thm:Bup-cst} and the explicit description
\[
\Fun(\cW,\cB_0)\cong \bigoplus_{i=1}^r \Fun(W_i,\cB_0).
\]
This gives the result.
\end{proof}
\begin{Rem}
\label{Rem:k-mod}%
In the Introduction, we assumed that $\kk$ is a field for simplicity.
Then \Cref{Thm:Burnside-intro}\,\eqref{it:Burnside-up} is~\Cref{Cor:Bup-cst} applied to~$\cB_0=\kk\ffree$.
For a general commutative ring~$\kk$, the notation $\kk G\mmod$ should be read as the category of finitely generated $\kk G$-modules whose underlying $\kk$-module is free (of finite rank).
\end{Rem}

\subsection{The mark transformation for constant input}\
\label{ssec:cst-mark}%

We finally identify the mark transformation of~\Cref{sec:mark}
\[
\mu_{\Ak}\colon \Alpk\to \Aupk
\]
for the restriction 2-functor~$\Ak$ of~\Cref{Not:Ak}, that is, $\Ak(G)=(\kk\ffree)^{\pi_0(G)}$.
We set~$\GG=\gpdf$ for the end of this section.
We assume that $\kk$ is a field, or use \Cref{Rem:k-mod}.

\begin{Thm}
\label{Thm:mark-Burnside}%
Let $G$ be a finite group and
$H_1,\ldots,H_r$ a complete set of representatives of conjugacy classes of subgroups of~$G$.
Under the equivalences of~\Cref{Thm:Alp-cst,Cor:Bup-cst} the mark
transformation
\[
\mu_{\Ak,G}\colon \Alpk(G)\to\Aupk(G)
\]
sends a finite $G$-set $X$ in~$\Burn(G)$ to the tuple in~$\oplus_{i=1}^r \kk(N_G(H_i)/H_i)\mmod$ given by
\[
\big(\kk(X^{H_i})\big)_{i=1,\ldots,r}
\]
where $\kk(X^H)$ is regarded as a permutation $\kk\big(N_G(H)/H\big)$-module for every~$H\le G$.
\end{Thm}

\begin{proof}
Let $H\le G$ be a subgroup and $W=N_G(H)/H$ be its Weyl group in~$G$. Consider the functor
\[
\omega\colon W\to \tau_1(\GGisoconn/G)
\]
sending the only object~$\sbull$ of~$W$ to~$H=(H\xinto{\incl} G)\in\GG/G$ and every morphism~$[g]\in W$ to the isomorphism class of the equivalence~$c_g\colon H\isoto H$ as in~\Cref{Exa:Weyl-as-endomorphisms}, that is, with structure 2-cell~$\st{c_g}=\gamma_g$.
We are claiming that the following diagram commutes

\[
\xymatrix@C=4em{
\Burn(G) \ar[r]^-{\textrm{Thm.\,\ref{Thm:Alp-cst}}}_-{\simeq} \ar@/_1em/@{-->}[rrd]_-{X\mapsto \kk(X^H)}
& \Alpk(G) \ar[r]^-{\mu_{\Ak}}_-{}
& \Aupk(G) \ar[r]_-{\simeq}^-{\textrm{Thm.\,\ref{Thm:Bup-cst}}}
& \Fun\big(\tau_1(\GGisoconn/G),\kk\ffree\big) \ar[d]^-{\omega^*}
\\
&& \kk(W)\mmod \ar@{=}[r]^-{\sim}
& \Fun(W,\kk\ffree)\,.\!
}
\]
(Indeed, we claim this for each $H=H_i$ in the statement.)

Under the equivalence $\theta_G$ of~\Cref{Thm:Alp-cst}, a finite $G$-set $X$ is sent to the object
\[
(G\ltimes X,\ukk)\in\Alpk(G),
\]
where $\ukk\in\Ak(G\ltimes X)$ denotes the constant object with value $\kk$ on every connected
component of~$G\ltimes X$. We want to compute its image in~$\Aupk(G)$ under~$\mu_{\Ak}$.

At level~$H\in\GG/G$,~\Cref{Thm:mark-explicit} tells us that $\mu_{\Ak,G}(G\ltimes X,\ukk)$ is
\[
\big(\mu_{\Ak,G}(G\ltimes X,\ukk)\big)_H = (p_2)_*p_1^*(\ukk),
\]
where $p_1$ and $p_2$ are the restrictions of~$\pr_1$ and~$\pr_2$ in the isocomma
\[
\xymatrix@C=1em@R=1em{
& (G\ltimes X)\times_G H\ar@{ >->}[ld]_-{\pr_1} \ar@{ >->}[rd]^-{\pr_2} \ar@{}[dd]|-{\isocell{\gamma}}
\\
G\ltimes X \ar@{ >->}[rd]_-{\st{G\ltimes X}}
&& H \ar@{ >->}[ld]^-{\incl}
\\
& G
}
\]
on the part where $\pr_2$ is fully faithful, denoted~$\st{G\ltimes X}^\approx H$ in~\Cref{Cons:j-approx}. Note that all morphisms above are faithful.
One easily gets an equivalence
\[
H\ltimes X\isoto (G\ltimes X)\times_G H
\]
(where the left-hand~$X$ is of course $X$ with action restricted to~$H$) and the above isocomma is equivalent to the following commutative Mackey square
\begin{equation}
\label{eq:aux-markk}%
\vcenter{\xymatrix@C=1em@R=1em{
& H\ltimes X \ar@{ >->}[ld]_-{\incl\ltimes \id_X\ } \ar@{ >->}[rd]^-{\st{H\ltimes X}} \ar@{}[dd]|-{=}
\\
G\ltimes X \ar@{ >->}[rd]_-{\st{G\ltimes X}}
&& H \ar@{ >->}[ld]^-{\incl}
\\
& G
}}
\end{equation}
By~\Cref{Rec:transporter} the transporter groupoid~$H\ltimes X$ has structure morphism over~$H$ sending every object of~$X$ to~$\sbull$. Therefore, this morphism is fully faithful exactly on the subgroupoid of~$H\ltimes X$ spanned by the objects~$X^H\subseteq X$. Every object~$a\in X^H$ has automorphism group~$H$ in~$H\ltimes X$ and therefore we can identify $\st{G\ltimes X}^\approx H$ with $\sqcup_{X^H}H$ and describe
\[
p_1^*(\ukk)=\ukk\in \Ak(\sqcup_{X^H}H)=\bigoplus_{X^H}\kk\ffree.
\]
Since $H$ has a single connected component, the local equivalence $p_2\colon \st{G\ltimes X}^\approx=\sqcup_{X^H}H\too H$ induces a folding push-forward $(p_2)_*$ that sends the constant~$\ukk$ to
\[
(p_2)_*(\ukk)=\oplus_{X^H}\kk.
\]
This is indeed the correct vector space~$\kk(X^H)$, namely the free vector space on the set~$X^H$, and we still need to trace the action of~$W=N_G(H)/H$.

Now let $g\in N_G(H)$. It defines an automorphism~$c_g\colon H\isoto H$ in~$\GG/G$ that we should follow in the above construction, and in particular in the key Mackey square~\eqref{eq:aux-markk}. The automorphism~$c_g$ induces an automorphism~$c_g\ltimes(g\cdot -)\colon H\ltimes X\isoto H\ltimes X$ which sends an object~$a\in X$ to~$ga$ and a morphism $h\colon a\to ha$ to $g h g^{-1}\colon ga\to g(ha)$. Since $g\in N_G(H)$, this restricts to an automorphism~$\st{G\ltimes X}^\approx H\isoto\st{G\ltimes X}^\approx H=\sqcup_{X^H}H$ that is given on objects~$X^H\to X^H$ by multiplication by~$g$. This yields an automorphism of
\[
p_2\colon \sqcup_{X^H}H \xto{\nabla} H
\]
that shuffles the connected components of~$\sqcup_{X^H}H$ via the action of~$g$ on~$X^H$. On~$(p_2)_*\colon \Ak(\sqcup_{X^H}H)=\oplus_{X^H}\kk\ffree\xto{\oplus} \kk\ffree=\Ak(H)$ this gives the automorphism~$(g\cdot)_*$ on~$(p_2)_*\big((x_a)_{a\in X^H})=\oplus_{a}x_a$ that shuffles the indices. In particular on~$x=\ukk=(\kk,\kk,\ldots,\kk)$ this gives the action of~$g$ on the free $\kk$-module~$\kk(X^H)$ that shuffles the basis~$X^H$. That is exactly the action as permutation module.
\end{proof}

\begin{Rem}
\label{Exa:mark-noniso}%
In view of~\Cref{Thm:mark-Burnside}, the mark transformation is not an equivalence in general. It is not essentially surjective, even up to direct summands. Indeed, it suffices to take $\kk$ a field of positive characteristic~$p>0$, for which most $\kk G$-modules are not $p$-permutation (trivial source) modules; this applies to any group of order divisible by~$p$ except the cyclic group of order~2 when~$p=2$.
\end{Rem}



\end{document}